\documentclass{amsart}

\usepackage{amssymb}
\usepackage{mathtools}
\usepackage{mathrsfs}
\usepackage{enumerate}
\usepackage{tikz}
\usepackage{esint}
\usepackage{hyperref}
\usepackage{tcolorbox}

\newtheorem{theorem}{Theorem}[section]
\newtheorem{statement}[theorem]{Statement}
\newtheorem{lemma}[theorem]{Lemma}
\newtheorem{proposition}[theorem]{Proposition}
\newtheorem{corollary}[theorem]{Corollary}

\newtheorem{assumption}[theorem]{Assumption}

\theoremstyle{definition}
\newtheorem{definition}[theorem]{Definition}

\theoremstyle{remark}
\newtheorem{remark}[theorem]{Remark}

\DeclarePairedDelimiter{\abs}{\lvert}{\rvert}

\DeclarePairedDelimiter{\bkt}{\lbrack}{\rbrack}
\DeclarePairedDelimiter{\pth}{(}{)}
\DeclarePairedDelimiter{\nor}{\lVert}{\rVert}
\DeclarePairedDelimiter{\set}{\lbrace}{\rbrace}

\makeatletter
\let\oldpth\pth
\def\pth{\@ifstar{\oldpth}{\oldpth*}}
\let\oldbkt\bkt
\def\bkt{\@ifstar{\oldbkt}{\oldbkt*}}
\let\oldset\set
\def\set{\@ifstar{\oldset}{\oldset*}}
\let\oldabs\abs
\def\abs{\@ifstar{\oldabs}{\oldabs*}}
\let\oldnor\nor
\def\nor{\@ifstar{\oldnor}{\oldnor*}}
\newcommand{\absetnostar}[1]{\abs{\set{#1}}}
\newcommand{\absetstar}[1]{\abs*{\set*{#1}}}
\def\abset{\@ifstar{\absetstar}{\absetnostar}}
\newcommand{\ptsetnostar}[1]{\pth{\set{#1}}}
\newcommand{\ptsetstar}[1]{\pth*{\set*{#1}}}
\def\ptset{\@ifstar{\ptsetstar}{\ptsetnostar}}

\def\ind{\@ifstar{\indWithSet}{\indWithoutSet}}
\makeatother

\renewcommand{\div}{\operatorname{div}}
\newcommand{\curl}{\operatorname{curl}}
\newcommand{\grad}{\nabla}
\newcommand{\La}{\Delta}
\newcommand{\mm}{\mathcal{M}}
\newcommand{\inn}{\text{ in }}
\newcommand{\onn}{\text{ on }}
\renewcommand*{\d}{\mathop{\kern0pt\mathrm{d}}\!{}}
\newcommand{\dt}{\d t}
\newcommand{\dx}{\d x}
\newcommand{\dy}{\d y}
\newcommand{\pt}{\partial _t}
\newcommand{\e}{\varepsilon}

\newcommand{\loc}{\mathrm{loc}}
\newcommand{\hfsq}[1]{\frac{\abs{#1} ^2}{2}}

\newcommand{\mins}[2][]{\min _{#1} \set{#2}}
\newcommand{\maxs}[2][]{\max _{#1} \set{#2}}
\newcommand{\cci}{C _c ^\infty}
\newcommand{\ssubset}{\subset \subset}

\DeclareMathOperator{\meas}{meas}
\DeclareMathOperator{\Reg}{Reg}
\DeclareMathOperator{\Sing}{Sing}
\newcommand{\dist}{\operatorname{dist}}
\newcommand{\esssup}{\operatorname*{ess\,sup}}
\newcommand{\R}{\mathbb{R}}
\newcommand{\Rd}{\R ^D}
\newcommand{\RR}[1]{\R ^{#1}}
\newcommand{\domain}{\R ^{1 + D}}

\newcommand{\OmegaT}{(0, T) \times \Omega}
\renewcommand{\P}{\mathcal P}
\renewcommand{\OmegaT}{\Omega _T}
\newcommand{\radm}{r _{\mathrm{adm}}}
\newcommand{\rint}{r _{\mathrm{int}}}
\newcommand{\rbar}{\bar r}
\newcommand{\rstar}{r _*}

\renewcommand{\dim}{\mathrm{dim}}
\newcommand{\Sa}{\mathscr S _\alpha}
\newcommand{\Saw}{\mathscr S _\alpha ^\wedge}

\newcommand{\Aa}{\mathscr A _\alpha}

\newcommand{\indWithSet}[1]{\mathbf1_{\set{#1}}}
\newcommand{\indWithoutSet}[1]{\mathbf1_{#1}}

\newcommand{\dfr}[2]{\frac{\mathrm{d} #1}{\mathrm{d} #2}}
\newcommand{\pthf}[2]{\pth{\frac{#1}{#2}}}

\newcommand{\intRd}{\int _{\Rd}}

\newcommand{\Lip}{\operatorname{Lip}}
\DeclareMathOperator{\Id}{Id}
\newcommand{\A}{\mathscr A}

\newcommand{\mres}{\mathbin{\vrule height 1.6ex depth 0pt width 0.13ex\vrule height 0.13ex depth 0pt width 1.3ex}}
\newcommand{\half}{\frac12}

\makeatletter
\@namedef{subjclassname@2020}{%
  \textup{2020} Mathematics Subject Classification}
\makeatother

\begin{document}

\title[Trace estimates via blow-up]{Vorticity interior trace estimates and higher derivative estimates via blow-up method}

\author{Jincheng Yang}
\address{School of Mathematics, 
Institute for Advanced Study,
1 Einstein Dr,
Princeton, NJ 08540, USA}
\email{jcyang@ias.edu}

\date{\today}

\thanks{\textit{Acknowledgement}. The author would like to thank Alexis Vasseur for suggesting the problem. The author would also like to thank Luis Silvestre and Seyhun Ji for helpful discussions on the appendices.}

\begin{abstract}
    We derive several nonlinear a priori trace estimates for the 3D incompressible Navier--Stokes equation, which extend the current picture of higher derivative estimates in the mixed norm. The main ingredient is the blow-up method and a novel averaging operator, which could apply to PDEs with scaling invariance and $\varepsilon$-regularity, possibly with a drift.
\end{abstract}

\maketitle

\tableofcontents

\section{Introduction}
\label{sec:introduction}

This paper aims to provide a family of a priori trace estimates and higher regularity estimates for the vorticity of the three-dimensional incompressible Navier--Stokes equation in a general Lipschitz domain $\Omega$. For some $T \in (0, \infty]$, denote $$\OmegaT = (0, T) \times \Omega.$$ Let $u: \OmegaT \to \RR3$ and $P: \OmegaT \to \R$ be a classical solution to the Navier--Stokes equation with no-slip boundary condition:
\begin{align}
    \label{eqn:navier-stokes}
        &\pt u + u \cdot \grad u + \grad P = \nu \La u, \quad \div u = 0 \quad && \inn (0, T) \times \Omega, \\
        \label{eqn:no-slip}
        &u = 0 && \onn (0, T) \times \partial \Omega.
\end{align} 
\eqref{eqn:no-slip} is dropped in the absence of a boundary.
By rescaling, we renormalize the equation with a unit kinematic viscosity $\nu = 1$. We study a priori bounds on the norm of derivatives of vorticity $\omega = \curl u$ over a lower-dimensional set $\Gamma _t \subset \Omega$, which is allowed to change in time. 

This work is motivated by the study of vortex sheet, vortex filament/vortex ring, and point vortex type solutions to the Euler equation in dimensions 2 and 3. Note that the vorticity of solutions to \eqref{eqn:navier-stokes} satisfies
\begin{align*}
    \partial _t \omega + u \cdot \grad \omega - \omega \cdot \grad u = \nu \La \omega, \quad \div \omega = 0 && \inn (0, T) \times \Omega.
\end{align*}
For the Euler equation, due to the absence of viscosity $\nu = 0$, vorticity is not dissipated but only transported and stretched. More specifically, if at $t = 0$ vorticity is supported over some lower dimensional manifold $\Gamma _0$, then at time $t$, the vorticity should still be supported on some manifold $\Gamma _t$ that is transported and stretched from $\Gamma _0$ by the velocity field $u$, which in turn can be recovered from $\omega$ via Biot--Savart law. 

This type of solution is not possible for the Navier--Stokes equation due to the dissipation. Nevertheless, we want to understand, for a solution of finite energy, how much vorticity can concentrate on lower dimensional manifolds. 

\subsection{Main results}

Throughout the article, we assume $\partial \Omega$ is uniformly Lipschitz and $\Gamma _t$ is a Lipschitz graph for each $t \in (0, T)$. We define $\OmegaT = (0, T) \times \Omega$ and the space-time graph to be 
\begin{align*}
    \Gamma _T = \set{(t, x): t \in (0, T), x \in \Gamma _t}. 
\end{align*}
$\rstar: (0, T) \times \Omega \to \R$ is a function defined by \eqref{eqn:defn-rstar} which characterizes the parabolic distance to the parabolic boundary. The precise setups are given in Section \ref{sec:preliminary}. The main theorem is the following a priori estimate on the vorticity.

\begin{theorem}
    \label{thm:main}
    Let $T \in (0, \infty]$, $\Omega \subset \RR3$ satisfy Assumption \ref{ass:lipschitz}, and $\set{\Gamma _t} _{t \in (0, T)}$ be $d$-dimensional submanifolds satisfying Assupmtion \ref{ass:Gamma} with Lipschitz constant $L$. There exist universal constants $C _n > 0$ for each $n \ge 0$ and a constant $C _L > 0$ depending on $L$ such that the following is true. 
    Let $u$ be a classical solution to the incompressible Navier--Stokes equation \eqref{eqn:navier-stokes} with no-slip boundary condition \eqref{eqn:no-slip}. Denote $\omega = \curl u$. There exists a measurable function $s _1: (0, T) \times \Omega \to [0, \infty]$ with the following properties:
    \begin{align*}
            |\grad ^{n} \omega (t, x)| \le C _n s _1 (t, x) ^{-n - 2}, \qquad \forall (t, x) \in \Omega _T, n \ge 0.
    \end{align*}
    \begin{enumerate}[\upshape (a)]
        \item For any $0 \le d \le 3$, it holds that 
        \begin{align*}
            \nor{s _1 ^{-1} \ind*{s _1 < \rstar}} _{L ^{d + 1, \infty} (\Gamma _T)} ^{d + 1} \le C _L \nor{\grad u} _{L ^2 (\OmegaT)} ^2.
        \end{align*}
        \label{enu:main}

        \item 
        If $2 \le d \le 3$ then for every $t \in (0, T)$ it holds that  
        \begin{align*}
            \nor{s _1 ^{-1} (t) \ind*{s _1 < \rstar}} _{L ^{d - 1, \infty} (\Gamma _{t})} ^{d - 1} \le C _L \nor{\grad u} _{L ^2 (\OmegaT)} ^2.
        \end{align*}
        \label{enu:main-fixed-t}
    \end{enumerate}
\end{theorem}

The theorem shows that $\grad ^n \omega$ is locally controlled in the weak $L ^\frac{d + 1}{n + 2}$ space.
Since $d \le 3$, the Lebesgue index $\frac{d + 1}{n + 2}$ is too small to pass to the limit and construct weak solutions in most cases, with two notable exceptions. 
\begin{enumerate}
    \item 
    When $d = 3$ and $n = 1$, we get local $L ^{\frac43, \infty}$ estimate for the vorticity gradient $\grad \omega$, which has been proven in several works. Constantin \cite{constantin1990} constructed suitable weak solutions in $\Omega = \mathbb T ^3$ with $\grad ^2 u \in L ^\frac4{3 + \e}$ for any $\e > 0$, which was improved by Lions \cite{lions1996} to weak space $L ^{\frac43, \infty}$, and local in space for bounded domains. For the case $\Omega = \RR3$, Vasseur \cite{vasseur2010} obtained local integrability of $\grad ^2 u \in L ^\frac4{3 + \e} _{\loc}$ using the blow-up method, then Choi and Vasseur \cite{choi2014} improved it to iterative local weak norm $L ^{\frac43, \infty} (t _0, T; L ^{\frac43, \infty} _{\loc} (\RR3))$ for $t _0 > 0$. They also extended this to fractional higher derivatives $(-\La) ^\frac\alpha2 \grad ^n u \in L ^{p, \infty} (t _0, T; L ^{p, \infty} _\loc (\RR3))$ with $p = \frac4{n + \alpha + 1}$. Recently, Vasseur and Yang \cite{vasseur2021} improved the local integrability to spacetime Lorentz norm $\grad ^2 u \in L ^{\frac43, q} _\loc$ for any $q > \frac43$, and they also obtained the higher regularity for the vorticity $\grad ^n \omega \in L ^{\frac4{n + 2}, q} _\loc$ for any $q > \frac4{n + 2}$ and $n \ge 0$ for classical solutions.
    
    \item 
    When $d = 2$ and $n = 0$, we obtain local $L ^{\frac32, \infty}$ interior trace of the vorticity. A weaker version was obtained by Vasseur and Yang \cite{vasseur2023,vasseur2023b} on the boundary trace $\partial \Omega$ for an averaged vorticity $\tilde \omega$ instead of $\omega$ itself. They used the blow-up method on the boundary to obtain this estimate and studied the inviscid limit problem.
\end{enumerate}
These two estimates can thus be extended to suitable weak solutions. 
Recall that suitable weak solutions defined in \cite{caffarelli1982} refer to a divergence-free vector field $u \in C _{\mathrm w} (0, T; L ^2 (\Omega)) \cap L ^2 (0, T; \dot H ^1 _0 (\Omega))$ and a scalar field $P \in L ^\frac54 (\OmegaT)$, that satisfy \eqref{eqn:navier-stokes} and the following local energy inequalities in distributional sense:
\begin{align*}
    \partial _t \hfsq u + \div \pth{
        u \pth{
            \hfsq u + P
        }
    } + |\grad u| ^2 \le \La \hfsq u.
\end{align*}
For such suitable solutions, as well as more general Leray--Hopf weak solutions, the energy dissipation is controlled by the initial kinetic energy:
\begin{align*}
    \frac12 \nor{u (T)} _{L ^2 (\Omega)} ^2 + \nor{\grad u} _{L ^2 (\OmegaT)} ^2 \le \frac12 \nor{u (0)} _{L ^2 (\Omega)} ^2.
\end{align*}
Therefore, the above and following estimates are a priori in the sense that they only rely on the kinetic energy of the initial data $u (0)$.

\begin{corollary}
    \label{cor:main}
    Let $T > 0$, let $\Omega \subset \RR3$ be a bounded set satisfying Assumption \ref{ass:lipschitz}. For any suitable weak solution $u$ to \eqref{eqn:navier-stokes}-\eqref{eqn:no-slip}, it holds that 
    \begin{align*}
        \nor{\grad \omega \ind*{|\grad \omega| > C \rstar ^{-3}}} _{L ^{\frac43, \infty} (\OmegaT)} ^\frac43 \le C _L \nor{\grad u} _{L ^2 (\OmegaT)} ^2.
    \end{align*}
    Moreover, if $\set{\Gamma _t} _{t \in (0, T)}$ satisfy Assumption \ref{ass:Gamma} with $d = 2$, then  
    \begin{align*}
        \nor{\omega \ind*{|\omega| > C \rstar ^{-2}}} _{L ^{\frac{3}{2}, \infty} (\Gamma _T)} ^\frac{3}{2} \le C _L \nor{\grad u} _{L ^2 (\OmegaT)} ^2.
    \end{align*}
\end{corollary}

Although previous a priori estimates only apply to the vorticity, we can control the higher derivatives for the symmetric part of $\grad u$ as well using the Hessian of the pressure. 

\begin{theorem}
    \label{thm:main-with-pressure}
    Under the same setting as in Theorem \ref{thm:main}, there exists a measurable function $s _2: (0, T) \times \Omega \to [0, \infty]$ with the following properties:
    \begin{align*}
        |\grad ^{n} u (t, x)| \le C _n s _2 (t, x) ^{-n - 1}, \qquad \forall (t, x) \in \Omega _T, n \ge 1.
    \end{align*}
    \begin{enumerate}[\upshape (a)]
        \item For any $0 \le d \le 3$, it holds that 
        \begin{align*}
            \nor{s _2 ^{-1} \ind*{s _2 < \rstar}} _{L ^{d + 1, \infty} (\Gamma _T)} ^{d + 1} \le C _L \pth{
                \nor{\grad u} _{L ^2 (\OmegaT)} ^2 + \nor{\grad ^2 P} _{L ^1 (\OmegaT)}
            }.
        \end{align*}
        \label{enu:a-priori-u}

        \item 
        If $2 \le d \le 3$ then for every $t \in (0, T)$ it holds that  
        \begin{align*}
            \nor{s _2 ^{-1} (t) \ind*{s _2 < \rstar}} _{L ^{d - 1, \infty} (\Gamma _{t})} ^{d - 1} \le C _L \pth{
                \nor{\grad u} _{L ^2 (\OmegaT)} ^2 + \nor{\grad ^2 P} _{L ^1 (\OmegaT)}
            }.
        \end{align*}
    \end{enumerate}
\end{theorem}

When $\Omega = \RR3$ or $\mathbb T ^3$, the $L ^1$ norm of $\grad ^2 P$ is controlled by $L ^2$ norm of $\grad u$ due to compensated compactness (see Remark \ref{rmk:pressure}), so the right-hand side is bounded again by the initial kinetic energy. When $\Omega$ has a nontrivial $C ^2$ boundary, we can still control $\grad ^n u$ away from the parabolic boundary of $\Omega _T$ as in Proposition \ref{prop:gradu}.

The main idea of both theorems is inspired by the blow-up method developed by Vasseur in \cite{vasseur2010}. The key is to find the appropriate ``scale function'' $s$. At each $(t, x)$, we blow-up the equation to the scale $\rho = s (t, x)$, then use existing $\e$-regularity theorems for suitable weak solutions to obtain full regularity in the $\rho$-neighborhood. By scaling, any higher derivative is controlled by the appropriate negative power of the scale. In \cite{vasseur2010}, the scale is a constant, whereas in \cite{choi2014} the scale can be viewed as a function of time. Now by allowing scale to depend on both time and space, we can obtain finer estimates here and below, with considerably more efforts in quantifying the scale function. In fact, our results also apply to suitable weak solutions on the regular set, and the singular set coincides with $\set{s = 0}$.

\subsection{Anisotropic norm estimates and regularity criteria}

The trace estimates mentioned above can also provide mixed norms of derivatives of vorticity via interpolation. First, note that taking trace to $d = 0$ yields an $L ^\infty$ estimate in $x$:
$\grad ^n \omega$ is locally in $L ^{\frac1{n + 2}, \infty} _t L ^\infty _x$, but this is slightly weaker than the current knowledge of a priori estimate for higher derivatives. It was proven in \cite{foias1981, duff1990} that $\partial _t ^r \grad ^n u$ is $L ^\frac1{2r + n + 1} _t L ^\infty _x$ and $L ^\frac{2}{4 r + 2 n - 1} _t L ^2 _x$ in strong norms up to the parabolic boundary. However, their proofs are based on different strategies, and the dependence on the initial energy is less clear.

By interpolating the spatial trace $L ^{1, \infty} _t L ^\infty _x$, the isotropic norm $L ^{4, \infty} _{t, x}$, and the temporal trace $L ^\infty _t L ^{2, \infty} _x$, we get the following picture of the anisotropic integrability. 
We remark that even though the isotropic Lorentz norm $L ^{4, \infty} _{t, x}$ and the nested Lorentz norm $L ^{4, \infty} _t L ^{4, \infty} _x$ are not equivalent and not comparable, interpolation is still possible (see Appendix \ref{app:lorentz}). For simplicity, we restrict our attention to $\Omega = \mathbb T ^3$. The case of a bounded domain is studied in Proposition \ref{prop:gradu}.

\begin{corollary}
    \label{cor:anisotropic}
    When $\Omega = \mathbb T ^3$, we conclude the following a priori bounds: for $0 < p < q \le \infty$ with $\frac1p + \frac3q = n + 1$, $n \ge 1$, $0 < t < T$, it holds that 
    \begin{align*}
        \nor{\grad ^n u} _{L ^{p, \infty} (t _0, T; L ^{q, \infty} (\mathbb T ^3))} \le C \pth{
            \nor{\grad u} _{L ^2 ((0, T) \times \mathbb T ^3)} ^{\frac2p} + (T - t _0) ^\frac1p \max \{t _0^{-\frac{n + 1}2}, 1\}
        }.
    \end{align*}
    Here $C = C (L, p, q, n)$. 
    For $0 < q < p \le \infty$ with $\frac1p + \frac1q = \frac{n + 1}2$, $n \ge 1$, it holds that 
    \begin{align*}
        \nor{\grad ^n u} _{L ^{p, \infty} (t _0, T; L ^{q} (\mathbb T ^3))} \le C \pth{
            \nor{\grad u} _{L ^2 ((0, T) \times \mathbb T ^3)} ^{\frac2q} + (T - t _0) ^\frac1q \max \{t _0^{-\frac{n + 1}2}, 1\}
        }.
    \end{align*}
    In addition, for $p = q = \frac4{n + 1}$, 
    \begin{align*}
        \nor{\grad ^n u} _{L ^{\frac4{n + 1}, \infty} ((t _0, T) \times \mathbb T ^3)} \le C \pth{
            \nor{\grad u} _{L ^2 ((0, T) \times \mathbb T ^3)} ^{\frac{n + 1}2} + (T - t _0) ^\frac{n + 1}4 \max \{t _0^{-\frac{n + 1}2}, 1\}
        }.
    \end{align*}
\end{corollary}
These results are plotted in Figure \ref{fig:integrability}.
Solid lines represent the strong norm $L ^p _t L ^q _x$ connecting $q = \infty$ and $q = 2$, which are due to \cite{foias1981} and \cite{duff1990}, whereas $\grad u \in L ^{\infty} _t L ^1 _x$ is due to \cite{constantin1990} and \cite{lions1996}. Dashed lines represent the weak norm $L ^{p, \infty} _{t, \mathrm{loc}} L ^{q, \infty} _x$ if $p < q$, and $L ^{p, \infty} _{t, \mathrm{loc}} L ^{q} _x$ if $p > q$. Note that it passes through $\grad ^2 u \in L ^{2, \infty} _{t, \loc} L ^1 _x$, which strengthens the result of Lions \cite{lions1996} which showed $L ^p _{t} L ^1 _x$ for $p < 2$. It also passes through $\grad ^2 u \in L ^{1, \infty} _{t, \loc} L ^{\frac32, \infty} _x$, which is weaker than the result of Lions \cite{lions1996} which showed $L ^{1, \infty} _t L ^{\frac32, 1} _x$. The diagonal dotted line $p = q$ represents the isotropic Lorentz space $L ^{p, \infty} _{t, x}$.

\begin{figure}
    \centering
    \begin{tikzpicture}
        \draw[dotted] (0, 0) -- (-5, -5);
    
        \draw[->, -latex] (-10.5, 0) -- (0.5, 0) node [anchor=west] {$-\frac1p$};
    
        \draw[->, -latex] (0, -5) -- (0, 0.5) node [anchor=south] {$-\frac1q$};
        \draw[] (-1.5, 0) -- (-1.5, 0.1) node [anchor=south] {$2$};
        \draw[] (-2.25, 0) -- (-2.25, 0.1) node [anchor=south] {$\frac43$};
        \draw[] (-3, 0) -- (-3, 0.1) node [anchor=south] {$1$};
        \draw[] (-4.5, 0) -- (-4.5, 0.1) node [anchor=south] {$\frac23$};
        \draw[] (-6, 0) -- (-6, 0.1) node [anchor=south] {$\frac12$};
        \draw[] (-7.5, 0) -- (-7.5, 0.1) node [anchor=south] {$\frac25$};
        \draw[] (-9, 0) -- (-9, 0.1) node [anchor=south] {$\frac13$};
        \draw[] (0, -0.5) -- (0.1, -0.5) node [anchor=west] {$6$};
        \draw[] (0, -1.5) -- (0.1, -1.5) node [anchor=west] {$2$};
        \draw[] (0, -2.25) -- (0.1, -2.25) node [anchor=west] {$\frac43$};
        \draw[] (0, -3) -- (0.1, -3) node [anchor=west] {$1$};
        \draw[] (0, -4.5) -- (0.1, -4.5) node [anchor=west] {$\frac23$};
        
        \draw (0, 0) node [anchor=south west] {$\infty$};
        \draw[] (-3, 0) -- (-1.5, -0.5) node[anchor=south] {$u$} -- (0, -1.5);
        \draw[] (-6, 0) -- (-1.5, -1.5) node[anchor=south] {$\grad u$} -- (0, -3);
        \draw[] (-9, 0) -- (-4.5, -1.5);
        \draw[dashed] (-4.5, -1.5) -- (-2.25, -2.25) node[anchor=south] {$\grad ^2 u$} -- (0, -4.5);
        \draw[] (-10.5, -0.5) -- (-7.5, -1.5);
        \draw[dashed] (-7.5, -1.5) -- (-3, -3) node[anchor=south] {$\grad ^3 u$} -- (-1, -5);
        \draw[dashed] (-10.5, -1.5) -- (-3.75, -3.75) node[anchor=south] {$\grad ^4 u$} -- (-2.5, -5);
        \draw[dashed] (-10.5, -2.5) -- (-4.5, -4.5) node[anchor=south] {$\grad ^5 u$} -- (-4, -5);
        \draw[dashed] (-10.5, -3.5) -- (-6, -5) node[anchor=south] {$\grad ^6 u$};
        \draw[dashed] (-10.5, -4.5) -- (-9, -5) node[anchor=south] {$\grad ^7 u$};
    \end{tikzpicture}
    \caption{Higher derivatives in mixed norm $L ^p _t L ^q _x$ or weak in $\Omega = \mathbb T ^3$}
    \label{fig:integrability}
\end{figure}
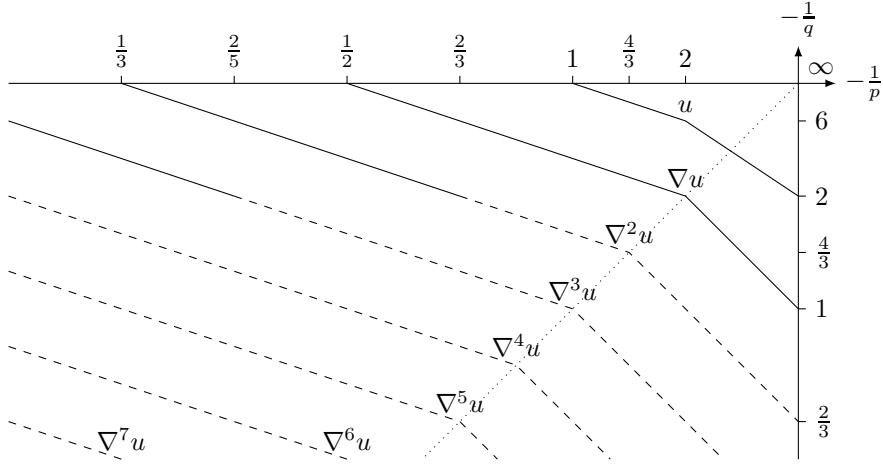

Unfortunately, none of the anisotropic a priori estimates could reach any regularity criteria due to the essence of the blow-up method. Indeed, since the energy norm is supercritical, all the a priori higher derivative estimates acquired via scaling are supercritical as well. However, all of the current energy criteria are either critical or subcritical.
Critical norms are invariant under the scaling of Navier--Stokes equation:
\begin{align*}
    u _\lambda (t, x) = \lambda u (\lambda ^2 t, \lambda x), \qquad P _\lambda (t, x) = \lambda ^2 P (\lambda ^2 t, \lambda x).
\end{align*}
Many of them also serve as regularity criteria: a weak solution is regular as long as one of these criteria norms is bounded. For instance, the Lady\v{z}enskaya--Prodi--Serrin criteria \cite{ladyzhenskaya1957,prodi1959,serrin1962,serrin1963,fabes1972} assert whenever $u$ is in $L ^p _t L ^q _x$ with $\frac2p + \frac3q = 1$, $2 \le p < \infty$, then $u$ is regular. The limit case $L ^\infty _t L ^3 _x$ was shown by Escauriaza--Ser\"{e}gin--\v{S}ver\'{a}k \cite{escauriaza2003}. There are some logarithmic improvements or in weak Lebesgue spaces \cite{chan2007,bjorland2011,bosia2014,montgomery2005,pineau2020}. The regularity criteria on velocity was extended to regularity criteria on velocity gradients $\grad u \in L ^p _t L ^q _x$ with $\frac2p + \frac3q = 2$ by Veiga \cite{veiga1995}. This also includes a special case, the Beal--Kato--Majda criterion \cite{beale1984}, i.e., $\nor{\omega} _{L ^1 _t L ^\infty _x}$, which was originally proposed as a criterion for the Euler equation. It is possible to require certain norms only on some components of the velocity or velocity gradient \cite{cao2008,cao2011,bae2016,chemin2017,chemin2017}, or in the direction of the velocity or vorticity \cite{vasseur2009,pineau2018,veiga2019,zhang2019,guo2022}. Below we give a comparison among some of these criteria in 3D periodic domain.

\begin{proposition}
    \label{prop:blow-up}
    Let $\Omega = \mathbb T ^3$, $T > 0$, and \renewcommand{\Omega}{\mathbb T ^3}
    let $u: [0, T) \times \Omega \to \RR3$ be a classical solution to \eqref{eqn:navier-stokes} with possible blowup at terminal time $T$. 
    Let $p, q, p', q'$ satisfy 
    \begin{align*}
        \frac2p + \frac3q = \frac2{p'} + \frac3{q'} &= 1, \qquad
        3 < q \le p < \infty, 2 \le p' \le \infty, 3 \le q' \le \infty.
    \end{align*}
    Then there exists $\gamma = \gamma (p, p') > 0$, such that for any $t \in (0, T)$ there exists $C = C (t, T, p, p')$ such that 
    \begin{align}
        \label{eqn:use-u}
        \nor{u} _{L ^{p'} (t, T; L ^{q'} (\Omega))} + \nor{\grad u} _{L ^{\frac{p'}2} (t, T; L ^{\frac{q'}2} (\Omega))} ^\frac12 &\le C \left(
            \nor{u} _{L ^p (0, T; L ^q (\Omega))} ^\gamma + 1
        \right). 
    \end{align}
    In addition, if $\int _{\Omega} u (0) \dx = 0$, and $3 < q \le \frac{15}4 < 10 \le p < \infty$ or $4 < q \le p < 8$, then
    \begin{align}
        \label{eqn:use-gradu}
        \nor{u} _{L ^{p'} (t, T; L ^{q'} (\Omega))} + \nor{\grad u} _{L ^{\frac{p'}2} (t, T; L ^{\frac{q'}2} (\Omega))} ^\frac12 &\le C \left(
            \nor{\grad u} _{L ^{\frac{p}2} (0, T; L ^{\frac{q}2} (\Omega))} ^\gamma + 1
        \right). 
    \end{align}
\end{proposition}

If any Lady\v zenskaya--Prodi--Serrin norm or Veiga norm blow up at $T ^*$, then $\nor{\grad u} _{L ^p _t L ^q _x}$ also blows up for $p, q$ in the range $2 \le q \le p \le 4$ and $\frac32 < q \le \frac{15}8 < 5 \le p < \infty$, and $\nor{u} _{L ^p _t L ^q _x}$ blows up for $p, q$ in the range $3 < q \le p < \infty$. Of course, all blow-up criteria should blow up at the same time. The purpose of this proposition is to provide a quantitative comparison of blow-up rates among different norms.

This paper is structured as the following. In Section \ref{sec:preliminary}, we introduce the notations and assumptions. We introduce the blow-up technique and the averaging operator in $\Rd$ in Section \ref{sec:space}, the time-dependent case in Section \ref{sec:spacetime}, and we include drift in Section \ref{sec:blowup}. We conclude in Section \ref{sec:proof} with the proofs of the main results. Some technical lemmas and auxiliary results are deferred to Appendix \ref{app:cantor} and \ref{app:lorentz}.

\section{Preliminary}
\label{sec:preliminary}

We will adopt the following assumptions on $\Omega$ and $\Gamma _t$.

\begin{assumption}
    \label{ass:lipschitz}
    Suppose $D \ge 0$, $\Omega \subset \Rd$ is open and nonempty, and $\partial\Omega$ is uniformly Lipschitz. That is, there exist $r _0 > 0, L > 0$ such that for every $y \in \partial \Omega$, $\partial \Omega \cap B _{r _0} (y)$ is an $L$-Lipschitz graph. Alternatively, we could also have the periodic cube $\Omega = \mathbb T ^D$; in this case, we denote $r _0 = 1$ to be the period of the set. 
\end{assumption}

\begin{assumption}
    \label{ass:Gamma}
    Let $\Gamma _t \subset \Omega$ be a time-dependent $L$-Lipschitz graph for every $t \in (0, T)$ with dimension $0 \le d \le D$. 
    More precisely, for every $t \in (0, T)$, up to choosing an orthonormal basis, there is a Lipschitz function $g _t: U _t \subset \R ^{d} \to \R ^{D - d}$ such that $\Gamma _t = \operatorname{Graph} (g _t)$. Note that $\Gamma _t \equiv \Omega$ is a $0$-Lipschitz graph with $d = D$.
\end{assumption}

We use $\meas (A)$ or $|A|$ to denote the Lebesgue measure of a set $A \subset \Rd$ or $\domain$. 
The symbol $\mathscr H ^{d}$ denotes the $d$-dimensional Hausdorff measure. 
Throughout the article, we equip $\Gamma _t$ with the Hausdorff measure $\mu _t = \mathscr H ^{d} \mres \Gamma _t$. 
For any $\mu _t$-measurable function $f$, $0 < q < \infty$, define 
\begin{align*}
    \nor{f} _{L ^q (\Gamma _t)} = \pth{
        \int _{\Gamma _t} |f| ^q \d \mu _t
    } ^\frac1q,
\end{align*}
and $\nor{f} _{L ^\infty (\Gamma _t)}$ denotes the $\mu _t$-$\esssup$ of $|f|$. For $0 < q _1, q _2 < \infty$, define the Lorentz norm 
\begin{align*}
    \nor{f} _{L ^{q _1, q _2} (\Gamma _t)} = \pth{
        q _1 \int _0 ^\infty \mu _t \pth{\set{|f| > \lambda}} ^{\frac{q _2}{q _1}} \lambda ^{q _2 - 1} \d \lambda
    } ^\frac1{q _2}.
\end{align*}
When $q _1 < q _2 = \infty$, define the weak norm 
\begin{align*}
    \nor{f} _{L ^{q _1, \infty} (\Gamma _t)} = \sup _{\lambda > 0} \lambda \mu _t \pth{\set{|f| > \lambda}} ^\frac1{q _1}.
\end{align*}
When $q _1 = \infty$, $L ^{\infty, q _2}$ coincides with $L ^\infty$ norm. When $q _1 = q _2$, $L ^{q _1, q _2}$ coincides with $L ^{q _1}$ norm.

We introduce the measure $\mu _T$ on $\Gamma _T$ by 
\begin{align*}
    \d \mu _T = \d \mu _t \d t.
\end{align*}
That is, for a set $A \subset \Gamma _T$, its $\mu _T$ measure is defined by 
\begin{align*}
    \mu _T (A) = \int _0 ^T \mu _t (A _t) \d t, \qquad A _t = \set{x \in \Gamma _t: (t, x) \in A}.
\end{align*}
For any $\mu _T$-measurable function $f$, $0 < p < \infty$, define 
\begin{align*}
    \nor{f} _{L ^p (\Gamma _T)} = \pth{
        \int _{\Gamma _T} |f| ^p \d \mu _T
    } ^\frac1q,
\end{align*}
and $\nor{f} _{L ^\infty (\Gamma _T)}$ denotes the $\mu _T$-ess sup of $|f|$. For $0 < p, q \le \infty$, define 
\begin{align*}
    \nor{f} _{L ^{p} _t L ^{q} _x (\Gamma _T)} = \begin{cases}
        \pth{
            \int _0 ^T \nor{f (t)} _{L ^q (\Gamma _t)} ^p \d t
        } ^\frac1p & p < \infty \\
        \esssup _{t \in (0, T)} \nor{f (t)} _{L ^q (\Gamma _t)} & p = \infty\;.
    \end{cases}
\end{align*}
The Lorentz norms $L ^{p _1, p _2} (\Gamma _T)$ and $L ^{p _1, p _2} _t L ^{q _1, q _2} _x (\Gamma _T)$ are defined similarly. In particular, the anisotropic weak norm is defined as the following. 

\begin{definition}
    Let $0 < p, q < \infty$. We define the anisotropic (nested) weak Lebesgue space $L ^{p, \infty} _t L ^{q, \infty} _x (\Gamma _T)$ by the space of all measurable functions $f: \Gamma _T \to \R$ such that 
    \begin{align*}
        \nor{f} _{L ^{p, \infty} _t L ^{q, \infty} _x (\Gamma _T)} = \sup _{\lambda > 0} \lambda \abs{\set{t \in (0, T): \nor{f (t)} _{L ^{q, \infty} (\Gamma _t)} > \lambda}} ^\frac1p < \infty. 
    \end{align*}
\end{definition}

In Appendix \ref{app:lorentz}, we see that the anisotropic weak Lebesgue space $L ^{p, \infty} _t L ^{p, \infty} _x (\Gamma _T)$ is not equivalent to the weak Lebesgue space $L ^{p, \infty} _{t, x} (\Gamma _T)$ even in the most simple setting, contrary to the strong Lebesgue spaces.

Similar to the strong norm, we have the following well-known interpolation and convolution theorems for the weak norm. Proofs can be found in \cite{grafakos2008}.

\begin{lemma}
    \label{lem:interpolation}
    Let $0 < p _0, p _1, q _0, q _1, r _0, r _1, s _0, s _1 \le \infty$, and $0 < \theta < 1$. Then we have 
    \begin{align*}
        \nor{f} _{L ^{p _\theta, r _\theta} _t L ^{q _\theta, s _\theta} _x (\Gamma _T)} \le \nor{f} _{L ^{p _0, r _0} _t L ^{q _0, s _0} _x (\Gamma _T)} ^{1 - \theta} \nor{f} _{L ^{p _1, r _1} _t L ^{q _1, s _1} _x (\Gamma _T)} ^{\theta}
    \end{align*}
    where
    \begin{align*}
        p _\theta &= (1 - \theta) p _0 + \theta p _1, &
        q _\theta &= (1 - \theta) q _0 + \theta q _1, 
        \\
        r _\theta &= (1 - \theta) r _0 + \theta r _1, &
        s _\theta &= (1 - \theta) s _0 + \theta s _1. 
    \end{align*}
\end{lemma}

\begin{lemma}
    \label{lem:convolution}
    Let $1 < p, q, r < \infty$ satisfy $\frac1p + \frac1q = \frac1r + 1$, $f \in L ^p (\R)$ and $g \in L ^{q, \infty} (\R)$. Then 
    \begin{align*}
        \nor{f * g} _{L ^r} \le C _{p, q} \nor{f} _{L ^p} \nor{g} _{L ^{q, \infty}}. 
    \end{align*}
    Moreover, in case $f \in L ^1 (\R)$, we have 
    \begin{align*}
        \nor{f * g} _{L ^{q, \infty}} \le C _q \nor{f} _{L ^1} \nor{g} _{L ^{q, \infty}}.
    \end{align*}
\end{lemma}

For $(t, x) \in \domain$, define the parabolic distance by
\begin{align*}
    \dist _\P ((t _1, x _1), (t _2, x _2)) = \pth{
        \abs{t _1 - t _2} + \abs{x _1 - x _2} ^2
    } ^\frac12, 
\end{align*}
and denote the parabolic boundary of $\OmegaT$ as 
\begin{align*}
    \partial _\P \OmegaT = \set{0} \times \Omega \cup [0, T) \times \partial \Omega.
\end{align*}
We say a function $f$ is essentially locally bounded at $(t, x)$ if there exist $\delta, M > 0$ such that $|f (s, y)| \le M$ for almost every $(s, y) \in \domain$ with $\dist _\P ((t, x), (s, y)) < \delta$.

Recall the following partial regularity result of Caffarelli, Kohn, and Nirenberg \cite{caffarelli1982}.

\begin{theorem}
    \label{thm:ckn}
    Let $u$ be a suitable weak solution to \eqref{eqn:navier-stokes} in $\OmegaT$. The singular set $\Sing (u)$ denotes the set where $u$ is not essentially locally bounded. Then the one-dimensional parabolic Hausdorff measure $\mathscr P ^1 (\Sing (u)) = 0$.
\end{theorem}

The estimates in this paper are over the complement of the singular set, i.e. the regular set $\Reg (u)$.

\vskip1em

We introduce a distance function $\rstar: \OmegaT \to \R$, which characterizes the parabolic distance to the parabolic boundary ($a \wedge b$ stands for $\mins{a,b}$):
\begin{align}
    \label{eqn:defn-rstar}
    \rstar (t, x) &:= \frac{\dist _\P ((t, x), \partial _\P \OmegaT) \wedge r _0}{L + 4} = \frac1{L + 4} \mins{
        \sqrt t, \dist(x, \partial \Omega), r _0
    }.
\end{align}

\section{Blow-up technique and trace estimate: space case}
\label{sec:space}

In the blow-up technique, we deal with an equation with scaling invariance, which contains the following situation. At the unit scale, two quantities $f$ and $g$ are related by the following one-scale, quantitative $\e$-regularity theorem: 

\begin{tcolorbox}
\begin{statement}[Template theorem]
    \label{thm:template}
    There exist two constants $\eta, C > 0$, such that 
    \begin{align*}
        \int _{B _1} f \dx \le \eta \qquad \text{ implies } \qquad \| g \| _{L ^\infty (B _\frac12)} \le C.
    \end{align*}
\end{statement}
\end{tcolorbox}

In the above statement, $f \ge 0$ is called a ``pivot quantity'', and $g$ is called the ``controlled quantity''. For instance, suppose $u: \Omega \to \R$ solves the elliptic equation 
\begin{align*}
    \div (A \grad u) = 0 \qquad \inn \Omega \subset \Rd    
\end{align*}
where $A: \Omega \to \mathcal M _{\text{sym}} ^{D \times D}(\R)$ is a uniformly elliptic matrix with $0 < \lambda \Id \le A \le \Lambda \Id$. Then the first De Giorgi lemma provides Statement \ref{thm:template} with $f = |u| ^2$ and $g = u$. There are also corresponding variants in the time-dependent setting, with ball $B _1$ replaced by cylinder $Q _1 = (-1, 0] \times B _1$. One example is the $\e$-regularity theorem for the Navier--Stokes in \cite{caffarelli1982} with $f = |u| ^3 + |p| ^\frac32$ and $g = u$.

Next, we rescale the above statement to $\e$-scale near $x \in \Omega$, using the scaling of the equation. Below $\alpha \ge 0$ and $\beta$ are two real numbers.

\begin{tcolorbox}
\begin{statement}[Rescaling of the template theorem]
    There exist two universal constants $\eta, C > 0$, such that for any $\e > 0$ sufficiently small, it holds that 
    \begin{align*}
        \fint _{B _\e (x)} f \dx \le \eta \e ^{-\alpha} 
        \qquad \text{ implies } \qquad 
        |g (x)| \le \| g \| _{L ^\infty (B _{\frac \e 2} (x))} \le C \e ^{-\beta}.
    \end{align*}
\end{statement}
\end{tcolorbox}

This can be used to prove the partial regularity of quantity $g$. Over the set at which $g$ is large, $\e$ is small, so we cover it by $5 B _{\e _i} (x _i)$ using a pairwise disjoint family of balls $B _{\e _i} (x _i)$. Since $\int _{B _{\e _i}} f \dx \sim \e _i ^{D - \alpha}$ is summable, the singular set $\{g = \infty\}$ has Hausdorff dimension at most $D - \alpha$. 

Since the singular set has codimension $\alpha$, intuitively it should be possible to take the trace of $g$ on a subset with codimension strictly less than $\alpha$. We will show that this is indeed the case.

\subsection{Scale operator and averaging operator}

In this subsection, we define the scale operator, which finds the threshold of scales below which the $\e$-regularity theorem can be applied. 

\begin{definition}
    \label{def:SASingReg}
    Let $f: \Rd \to [0, \infty]$ be a Lebesgue measurable function. Denote 
    \begin{align*}
        f _\rho (x) := \fint _{B _\rho (x)} f (y) \d y, \qquad 0 < \rho < \infty.
    \end{align*}
    For $\alpha > 0$, we define the \textbf{scale operator} $\Sa [f]: \Rd \to [0, \infty]$ by 
    \begin{align}
        \label{def:E-space}
        \Sa [f] (x) &:= \inf _{0 < \rho < \infty} \set{
            \rho: f _\rho (x) > \rho ^{-\alpha}
        }, \qquad x \in \Rd.
    \end{align}
    We define the \textbf{averaging operator} $\Aa [f]: \Rd \to [0, \infty]$ by 
    \begin{align}
        \label{def:A-space}
        \Aa [f] (x) &:= \Sa [f] (x) ^{-\alpha}, \qquad x \in \Rd.
    \end{align}
    We define the \textbf{singular set} and \textbf{regular set} as 
    \begin{align*}
        \Sing _\alpha (f) = \set{x \in \Rd: \Sa [f] (x) = 0}, \qquad \Reg _\alpha (f) = \set{x \in \Rd: \Sa [f] (x) > 0}.
    \end{align*}
\end{definition}

Note that $B _\rho (x)$ depends continuously on $\rho$ and $x$ in the following sense. We omit the proof since it is a standard exercise.

\begin{lemma}
    \label{lem:inner-outer-regularity}
    Fix $\rho > 0$ and $x \in \Rd$. 
    \begin{enumerate}
        \item (Inner regularity) For any compact subset $K \ssubset B _{\rho} (x)$, $B _\sigma (y)$ is also a superset of $K$ provided $(\sigma, y)$ is sufficiently close to $(\rho, x)$. 
        \item (Outer regularity) For any open superset $O \supset \supset \bar B _{\rho} (x)$, $B _\sigma (y)$ is also a subset of $O$ provided $(\sigma, y)$ is sufficiently close to $(\rho, x)$.
    \end{enumerate}    
\end{lemma}

Because $B _\sigma (y)$ depends on $(\sigma, y)$ continuously, we have the following (semi-) continuity on $f _\sigma (y)$.

\begin{lemma}
    \label{lem:continuity-space}
    The map $(\rho, x) \mapsto f _\rho (x)$ is lower semicontinuous in $\R _+ \times \Rd$. Moreover, if $f \in L ^1 _\loc (\domain)$, then this map is continuous in $\R _+ \times \Rd$.
\end{lemma}

\begin{proof}
    Note that 
    \begin{align*}
        f _\rho (x) = \frac1{|B _\rho (x)|} \int _{B _\rho (x)} f,
    \end{align*}
    where $|B _\rho (x)| = \rho ^D |B _1|$ is continuous in $\rho \in \R _+$. To show that the integral is lower semicontinuous, we fix $\rho > 0$ and $x \in \Rd$. For any compact subset $K \ssubset B _{\rho} (x)$,  Lemma \ref{lem:inner-outer-regularity} implies
    \begin{align*}
        \liminf _{(\sigma, y) \to (\rho, x)} f _\sigma (y) = \frac1{|B _{\rho}|} \liminf _{(\sigma, y) \to (\rho, x)} \int _{B _{\sigma} (y)} f \ge \frac1{|B _{\rho}|} \int _K f.
    \end{align*} 
    This is true for any compact subset $K$, so taking a sequence $K _n \uparrow B _{\rho} (x)$ yields $\int _{K _n} f \to \int _{B _\rho (x)} f = |B _\rho| f _\rho (x)$ and completes the proof of lower semicontinuity.

    To show continuity when $f$ is locally integrable, we may take an open superset $O$ of the closure of $B _{\rho} (x)$. Again by Lemma \ref{lem:inner-outer-regularity}, we obtain 
    \begin{align*}
        \limsup _{(\sigma, y) \to (\rho, x)} f _\sigma (y) = \frac1{|B _{\rho}|} \limsup _{(\sigma, y) \to (\rho, x)} \int _{B _{\sigma} (y)} f \le \frac1{|B _{\rho}|} \int _O f.
    \end{align*} 
    This is true for any superset $O$ of the closure of $B _{\rho} (x)$, so by taking a sequence of $O _n \downarrow B _{\rho} (x)$ completes the proof of upper semicontinuity, provided that $f$ is locally integrable.
\end{proof}

Moreover, we can also show the semicontinuity of $\Sa [f]$ and $\Aa [f]$.

\begin{lemma}
    $\Sa [f]$ is upper semicontinuous, and $\Aa [f]$ is lower semicontinuous.
\end{lemma}

\begin{proof}
    Fix $x \in \Rd$, we want to show that 
    \begin{align*}
        \limsup _{y \to x} \Sa [f] (y) \le \Sa [f] (x).
    \end{align*}
    If $\Sa [f] (x) = +\infty$, then there is nothing to prove. If $\Sa [f] (x) < +\infty$, then according to \eqref{def:E-space}, for any $\e > 0$, there exists $M \in [\Sa [f] (x), \Sa [f] (x) + \e]$ such that 
    \begin{align*}
        f _M (x) > M ^{-\alpha}.
    \end{align*}
    By the lower semicontinuity of $f _M$, there exists $\delta > 0$ such that
    \begin{align*}
        f _M (y) > M ^{-\alpha}, \qquad \forall y \in B _\delta (x).
    \end{align*}
    By definition, $\Sa [f] (y) \le M \le \Sa [f] (x) + \e$. This shows 
    \begin{align*}
        \limsup _{y \to x} \Sa [f] (y) \le \Sa [f] (x) + \e    
    \end{align*}
    for any $\e > 0$, thus $\Sa [f]$ is upper semicontinuous, and correspondingly $\Aa [f] = (\Sa [f]) ^{-\alpha}$ is lower semicontinuous.
\end{proof}

Now we can justify the name ``averaging operator'': $\Aa [f]$ is indeed the average over a ball.

\begin{lemma}
    \label{lem:average-name}
    Suppose $f \in L ^1 _{\loc} (\Rd)$, $\alpha > 0$, $x \in \Rd$. Denote $s = \Sa [f] (x)$. If $0 < s < \infty$, then $\Aa [f] (x) = f _s (x) = \fint _{B _s (x)} f$.
\end{lemma}

\begin{proof}
    On the one hand, from the definition \eqref{def:E-space}, we have $f _\rho (x) \le \rho ^{-\alpha}$ for any $\rho < s$. By the lower semicontinuity of $f _\rho (x)$ in $\rho$, we have $f _s (x) \le s ^{-\alpha} = \Aa [f] (x)$. On the other hand, if $f _s (x) < s ^{-\alpha}$ then the continuity of $f _\rho (x)$ and $\rho ^{-\alpha}$ in $\rho$ would imply that $f _\rho (x) < \rho ^{-\alpha}$ for $\rho$ sufficiently close to $s$, violating definition \eqref{def:E-space}. Therefore $f _s (x) = s ^{-\alpha} = \Aa [f] (x)$.
\end{proof}

Notice that $\Aa$ is a nonlinear operator, where $\alpha > 0$ is related to nonlinearity. The following property is a simple consequence of Jensen inequality.

\begin{lemma}
    \label{lem:nonlinear-scaling}
    Let $f: \Rd \to [0, \infty]$ be a Lebesgue measurable function. Let $p \in [1, \infty)$, then 
    \begin{align*}
        \Sa [f] \ge \mathscr S _{p\alpha} [f ^p], \qquad \Aa [f] \le \A _{p\alpha} [f ^p]
        ^\frac1p.
    \end{align*}
\end{lemma}

\begin{proof}
    Fix $x \in \Rd$, $\rho \in (0, \infty)$. If $f _\rho (x) > \rho ^{-\alpha}$ then 
    \begin{align*}
        \fint _{B _\rho (x)} f ^p \ge \pth{
            \fint _{B _\rho (x)} f
        } ^p > \rho ^{-p\alpha}.
    \end{align*}
    Hence 
    \begin{align*}
        \set{\rho \in (0, \infty): f _\rho (x) > \rho ^{-\alpha}} \subset \set{\rho \in (0, \infty): (f ^p) _\rho (x) > \rho ^{-p\alpha}}.
    \end{align*}
    Thus the infimum of the left is no less than the infimum of the right, which by definition \eqref{def:E-space} says $\Sa [f] (x) \ge \mathscr S _{p\alpha} [f ^p] (x)$. The second inequality is a direct consequence from the definition \eqref{def:A-space}.
\end{proof}

We conclude this subsection by showing $\Aa$ is quasiconvex. 

\begin{lemma}\label{lem:sublin-space}
    Fix $\alpha > 0$, for any $f, g \in  L ^1 _{\loc} (\Rd)$, $\lambda \in [0, 1]$, it holds that
    \begin{align*}
        \Aa \bkt{(1 - \lambda) f + \lambda g} (x) \le \maxs{
            \Aa [f] (x), \Aa (g) (x)
        }, \qquad \forall x \in \Rd.
    \end{align*}
\end{lemma}

\begin{proof}
    Denote $h = (1 - \lambda) f + \lambda g$, and take $x \in \Rd$. If $\Sa [h] (x) = \infty$, then $\Aa [h] (x) = 0$ and there is nothing to prove. Otherwise, there exists a sequence of $\rho _n > \Sa [h] (x)$ and $\rho _n \to \Sa [h] (x)$ such that 
        \begin{align*}
            \rho _n ^{-\alpha} < h _{\rho _n} (x) = (1 - \lambda) f _{\rho _n} (x) + \lambda g _{\rho _n} (x).
        \end{align*}
        Therefore, at least one of $f _{\rho _n} (x)$ and $g _{\rho _n} (x)$ is greater than $\rho _n ^{-\alpha}$. Suppose, up to a subsequence, that $f _{\rho _n} (x) > \rho _n ^{-\alpha}$. Then $\Sa [f] (x) \le \rho _n$ for any $n$, and taking $n \to \infty$ we find $\Sa [f] (x) \le \Sa [h] (x)$. Consequently 
        \begin{align*}
            \Aa [h] (x) = \Sa [h] (x) ^{-\alpha} \le \Sa [f] (x) ^{-\alpha} = \Aa [f] (x).
        \end{align*}
        The case $g _{\rho _n} (x) > \rho _n ^{-\alpha}$ would result in $\Aa [h] (x) \le \Aa [g] (x)$, so we conclude $\Aa [h]$ is smaller than the max of $\Aa [f]$ and $\Aa [g]$.
\end{proof}

\subsection{Partial regularity and trace estimate}

Recall in Definition \ref{def:SASingReg}, we decomposed $\Rd = \Sing _\alpha (f) \cup \Reg _\alpha (f)$. First, we show that for a locally $L ^p$ function, the singular set has codimension at least $p \alpha$.

\begin{proposition}
    \label{prop:partial-regularity}
    Suppose $f \in L ^p _{\loc} (\Rd)$ is a nonnegative function with $p \in [1, \infty)$. If\; $0 < p \alpha < D$ then the $(D - p \alpha)$-dimensional Hausdorff measure of $\Sing _\alpha (f)$ is zero. If $p \alpha \ge D$ then $\Sing _\alpha (f)$ is empty.
\end{proposition}

\begin{proof}
    We first prove it for $p = 1$.

    Define $\beta = \maxs{D - \alpha, 0}$.
    Let us assume $f \in L ^1 (\Rd)$ first. We will later see that local integrability is sufficient.
    We first fix a small number $\delta \in (0, 1)$. For any $x \in \Sing _\alpha (f)$, $\Sa [f] (x) = 0$. By \eqref{def:E-space}, there exists $\rho _x < \frac15 \delta$ such that 
    \begin{align*}
        \fint _{B _{\rho _x} (x)} f \dx > \rho _x ^{-\alpha} \qquad \implies \qquad \frac1{|B _1|}\int _{B _{\rho _x} (x)} f \d x > \rho _x ^{D - \alpha} \ge \rho _x ^\beta.
    \end{align*}
    The last inequality holds when $D - \alpha < 0$ because $\rho _x < 1$. 
    Now $\set{B _{\rho _x} (x)} _{x \in \Sing _\alpha (f)}$ forms an open cover of $\Sing _\alpha (f)$. Using Vitali covering lemma, we can select an at most countable disjoint subfamily $\set{B _{\rho _i} (x _i)} _{i}$ such that
    \begin{align*}
        \Sing _\alpha (f) \subset \bigcup _{i} B _{5 \rho _i} (x _i).
    \end{align*}
    Since $5 \rho _i < \delta$, we can bound 
    \begin{align*}
        \mathscr H _\delta ^{\beta} (\Sing _\alpha (f)) &:= \inf \set{
            \sum _j r _j ^{\beta}: \Sing _\alpha (f) \subset \bigcup _{j} B _{r _j} (x _j)
        } \\
        & \le \sum _i (5 \rho _i) ^{\beta} \\
        & \le C \sum _i \int _{B _{\rho _i} (x _i)} f \dx \\
        & = C \int _{\bigcup _i B _{\rho _i} (x _i)} f \dx \\
        & \le C \int _{\Rd} f \dx.
    \end{align*}
    It is finite, so 
    \begin{align*}
        \abs{\bigcup _i B _{\rho _i} (x _i)} = |B _1| \sum _i \rho _i ^D \le C \delta ^{D - \beta} \sum _i \rho _i ^\beta \le C \delta ^{D - \beta} \int _{\Rd} f \dx \to 0
    \end{align*}
    as $\delta \to 0$. Therefore 
    \begin{align*}
        \mathscr H ^{\beta} (\Sing _\alpha (f)) = \lim _{\delta \to 0} \mathscr H _\delta ^{\beta} (\Sing _\alpha (f)) \le \lim _{\delta \to 0} C \int _{\bigcup _i B _{\rho _i} (x _i)} f \dx = 0.
    \end{align*}
    In particular, when $\alpha \ge D$ and $\beta = 0$, we have $\Sing _\alpha (f) = \varnothing$.

    If $f$ is merely locally integrable but not $L ^1 (\Rd)$, we can pick any compact set $K \subset \Rd$ and apply the above argument to prove $\mathscr H ^{\beta} (\Sing _\alpha (f) \cap K) = 0$. Since $K$ is arbitrary, we conclude $\mathscr H ^{\beta} (\Sing _\alpha (f)) = 0$.

    For the general case $p > 1$, we apply the above argument to $\tilde f = f ^p \in L ^1 _\loc$ and $\tilde \alpha = p \alpha > 0$. If $0 < \tilde \alpha < D$, then the $(D - \tilde \alpha)$-dimensional Hausdorff measure of $\Sing _{\tilde \alpha} (\tilde f)$ is zero. If $\tilde \alpha \ge D$ then $\Sing _{\tilde \alpha} (\tilde f) = \varnothing$. By Lemma \ref{lem:nonlinear-scaling}, we know that 
    \begin{align*}
        \Sa [f] \ge \mathscr S _{\tilde \alpha} [\tilde f] \implies \Sing _\alpha (f) \subset \Sing _{\tilde \alpha} (\tilde f).
    \end{align*}
    This completes the proof.
\end{proof}

Assume $\Gamma \subset \Rd$ is a $d$-dimensional Lipschitz graph.
To establish trace estimates, we need to measure the level sets of $\Sa [f]$. 

\begin{lemma}
    \label{lem:ap}
    Let $\alpha > 0$, and $f \in L ^1 (\Rd)$ is nonnegative. Denote 
    \begin{align*}
        A (\rho) &:= \set{
            x' \in \Gamma:
            \rho \le \Sa [f] (x') < 2 \rho
        }, \qquad \rho > 0.
    \end{align*}
    Then 
    \begin{align*}
        \mathscr H ^d \pth{A (\rho)} \le C (D, d, L) \rho ^{-D + d + \alpha} \intRd f \dx.
    \end{align*}
\end{lemma}

\begin{proof}
    Recall that up to choosing a coordinate, $\Gamma$ is the graph of some $L$-Lipschitz function $g: U \subset \RR d \to \RR {D - d}$. Define $\phi: U \times \RR{D - d} \to \R ^D$ by 
    \begin{align*}
        \phi (x', z) = (x', g (x') + z), \qquad x' \in U, z \in \RR{D - d}.
    \end{align*}
    then 
    \begin{align*}
        A (\rho) &= \phi (U (\rho) \times \set0), \text{ for some } U (\rho) \subset U.
    \end{align*}
    Since $g$ is Lipschitz, the Hausdorff measure is bounded by 
    \begin{align}
        \label{eqn:Arho1}
        \mathscr H ^d (A (\rho)) &\le \nor{\phi} _{\Lip} ^{d} \mathscr L ^{d} \pth{U (\rho)} \le (1 + \nor{g} _{\Lip} ^2) ^{\frac{d}2} \mathscr L ^{d} \pth{U (\rho)}, 
    \end{align}
    where $\mathscr L ^{d}$ is the Lebesgue outer measure of dimension $d$. Recall that $\nor{g} _{\Lip} \le L$. Next, note that it is clear from the definition of $\phi$ that 
    \renewcommand{\delta}{\rho} 
    \begin{align*}
        \phi (U (\rho) \times B _\delta (0)) \subset \mathcal U _\delta (A (\rho)),
    \end{align*}
    where $B _\delta (0) \subset \RR{D - d}$ and $\mathcal U _\delta (A (\rho))$ is the $\delta$-tubular neighborhood of $A (\rho)$. As $\det \grad \phi = 1$, $\phi$ is measure preserving, so the measure of $U$ equals to 
    \begin{align}
        \mathscr L ^{d} \pth{U (\rho)} = \frac{\mathscr L ^D (\phi (U (\rho) \times B _\delta (0)))}{\mathscr L ^{D - d} (B _\delta (0))} \le \frac{\mathscr L ^D (\mathcal U _\delta (A (\rho)))}{c _{D - d} \delta ^{D - d}}.
    \end{align}
    Here $c _{D - d} = \mathscr L ^{D - d} (B _1)$ is the measure of a $(D - d)$-dimensional unit ball.

    For any $x' \in A (\rho)$, by definition \eqref{def:E-space}, there exists $\rho _{x'} \in [\rho, 2\rho)$ such that 
    \begin{align*}
        f _{\rho _{x'}} (x') > \rho _{x'} ^{-\alpha} \qquad \implies \qquad \rho _{x'} ^{D - \alpha} < \int _{B _{\rho _{x'}} (x')} f \dx.
    \end{align*}
    Note that $\rho _{x'} \ge \rho$ implies $\mathcal U _\rho (A (\rho))$ has an open cover
    \begin{align*}
        \mathcal U _\rho (A (\rho)) \subset 
        \bigcup _{x' \in A (\rho)} B _{\rho _{x'}} (x).
    \end{align*}
    By Vitali's covering lemma, we can find a disjoint subcollection $\set{B _{\rho _i} (x' _i)} _i$ such that $\mathcal U _\rho (A (\rho)) \subset \bigcup _i B _{5 \rho _i} (x' _i)$. So 
    \begin{align}
        \mathscr L ^{D} (\mathcal U _\rho (A (\rho)) ) \le \sum _{i} \mathscr L ^{D} (B _{5 \rho _i} (x _i)) \le |B _1| 5 ^{D} \sum _i \rho _i ^{D} \le C \rho ^{\alpha} \sum _i \rho _i ^{D-\alpha}.
    \end{align}
    Finally, since $B _{\rho _i} (x' _i)$ are pairwise disjoint, we have 
    \begin{align}
        \label{eqn:Arho2}
        \sum _i \rho _i ^{D-\alpha} \le \sum _i \int _{B _{\rho _i} (x' _i)} f \dx \le \int _{\Rd} f \dx.
    \end{align}
    Combine \eqref{eqn:Arho1}-\eqref{eqn:Arho2} finishes the proof of the lemma.
\end{proof}

With this, now we can prove the trace estimate for the averaging operator.

\begin{theorem}
    \label{thm:avg-space}
    Let $f \in L ^p _{\mathrm{loc}} (\Rd)$ be a nonnegative function with some $p \in [1, \infty]$. Let $\alpha > 0$ satisfy $p \alpha > D - d$.
    \begin{enumerate}[\upshape (a)]
        \item $\Aa [f]$ is $(\mathscr H ^d \mres \Gamma)$-measurable, and 
        \begin{align*}
            \mathscr H ^d (\set{\Aa [f] = 0} \cap \Gamma) = 0.
        \end{align*}
        \item If $p = 1$, $f \in L ^1 (\Rd)$, then 
        \label{enu:weaktype-space}
        \begin{align*}
            \nor{(\Aa [f]) ^{1 - \frac{D - d}{\alpha}}} _{L ^{1, \infty} (\Gamma)}\le C (\alpha, D, d, L) \nor{f} _{L ^1 (\Rd)}.
        \end{align*}
        \item If $p > 1$, $f \in L ^p (\Rd)$, then
        \label{enu:strongtype-space}
        \begin{align*}
            \nor{(\Aa [f]) ^{1 - \frac{D - d}{p \alpha}}} _{L ^p (\Gamma)} \le C (p, \alpha, D, d, L) \nor{f} _{L ^p (\Rd)}.
        \end{align*}
    \end{enumerate}
\end{theorem}

\begin{proof}
    Notice that for $\rho > 0$, $x \in \Rd$, 
    \begin{align*}
        (2 \rho) ^{-\alpha} < \Aa [f] (x) \le \rho ^{-\alpha} \iff 
        \rho \le \Sa [f] (x) < 2 \rho.
    \end{align*}
    This means 
    \begin{align*}
        \set{(2 \rho) ^{-\alpha} < \Aa [f] \le \rho ^{-\alpha}} \cap \Gamma = A (\rho),
    \end{align*}
    where $A (\rho)$ is defined in Lemma \ref{lem:ap}.

    \begin{enumerate}[\upshape (a)]
        \item The measurability of $\Aa [f]$ follows from semicontinuity. By Proposition \ref{prop:partial-regularity}, we know the singular set $\Sing _\alpha (f) = \set{\Aa [f] = \infty}$ has Hausdorff dimension at most $D - p \alpha$, so its $d$-dimensional measure is zero.
        \item Fix $\lambda > 0$, and pick $\rho _0 = \lambda ^{-\frac1\alpha}$, $\rho _k = 2 ^{-k} \rho _0$. Then 
        \begin{align*}
            \mathscr H ^d (\set{\Aa [f] > \lambda} \cap \Gamma) &= \sum _{k = 0} ^\infty \mathscr H ^d \pth{
                \set{\rho _k ^{-\alpha} < \Aa [f] \le \rho _{k + 1} ^{-\alpha}} \cap \Gamma
            } \\
            &= \sum _{k = 0} ^\infty \mathscr H ^d (A (\rho _{k + 1})) \\
            &\le C \sum _{k = 0} ^\infty \rho _{k + 1} ^{-D + d + \alpha} \nor{f} _{L ^1 (\Rd)} \\
            &\le C \rho _0 ^{\alpha - D + d} \nor{f} _{L ^1 (\Rd)} = \frac C{\lambda ^{1 - \frac{D - d}{\alpha}}} \nor{f} _{L ^1 (\Rd)}.
        \end{align*}
        The summation is a converging geometric series since the index $\alpha-D + d$ is positive.
        We thus conclude that $\nor{(\Aa [f]) ^{1 - \frac{D - d}{\alpha}}} _{L ^{1, \infty} (\Gamma)} \le C \nor{f} _{L ^1 (\Rd)}$.

        \item We prove the case $p = \infty$ first.
        Fix any $x \in \Rd$, we note from the definition of $f _\rho$ that $f _\rho (x) \le \nor{f} _{L ^\infty (\Rd)}$ for any $\rho > 0$. Therefore 
        \begin{align*}
            f _\rho (x) < \rho ^{-\alpha}, \qquad \forall \rho < \nor{f} _{L ^\infty (\Rd)} ^{-\frac1\alpha}.
        \end{align*}
        From the definition \eqref{def:E-space} we know 
        \begin{align*}
            \Sa [f] (x) \ge \nor{f} _{L ^\infty (\Rd)} ^{-\frac1\alpha} > 0,
        \end{align*}
        and correspondingly $\Aa [f] (x) \le \nor{f} _{L ^\infty (\Rd)}$.

        To show the boundedness for $p \in (1, \infty)$, we first assume in addition that $\alpha > D - d$. For any $\lambda > 0$ we define 
        \begin{align*}
            f _1 = 2 f \ind*{f \le \frac\lambda2}, \qquad f _2 = 2 f \ind*{f > \frac\lambda2} .
        \end{align*}
        Then $f = \half \pth{f _1 + f _2}$. By the quasiconvexity Lemma \ref{lem:sublin-space}, we conclude that 
        \begin{align*}
            \set{\Aa [f] > \lambda} \subset \set{\Aa [f _1] > \lambda} \cup \set{\Aa [f _2] > \lambda} = \set{\Aa [f _2] > \lambda},
        \end{align*}
        as $\Aa [f _1] \le \nor{f _1} _{L ^\infty (\Rd)} \le \lambda$. Thus
        \begin{align*}
            \mathscr H ^d \pth{
                \set{\Aa [f] > \lambda} \cap \Gamma
            } &\le \mathscr H ^d \pth{
                \set{\Aa [f _2] > \lambda} \cap \Gamma
            } \\
            &\le \frac C{\lambda ^{1 - \frac {D - d}\alpha}} \nor{f _2} _{L ^1 (\Rd)}.
        \end{align*}
        Here we used part \eqref{enu:weaktype-space} on $f _2$. And the remaining is analogous to the classical setting:
        \begin{align*}
            \nor{(\Aa [f]) ^{1 - \frac{D - d}{p \alpha}}} _{L ^{p} (\Gamma)} ^{p} &= 
            \int _{\Gamma} (\Aa [f]) ^{p - \frac{D - d}{\alpha}} \d \mathscr H ^d \\
            &= C \int _0 ^\infty \lambda ^{p - \frac{D - d}{\alpha} - 1} \mathscr H ^d \pth{
                \set{\Aa [f] > \lambda} \cap \Gamma
            } \d \lambda
            \\
            & \le C \int _0 ^\infty \lambda ^{p - 2}\int _{\Rd} f _2 \dx \d \lambda 
            \\
            &= C \int _{\Rd} \int _0 ^\infty f \ind*{f > \frac\lambda2} \lambda ^{p - 2} \d \lambda \dx  \\
            &\le C \int _{\Rd} f ^p \dx .
        \end{align*}

        We now use Lemma \ref{lem:nonlinear-scaling} to remove the extra assumption $\alpha > D - d$. Since $p \alpha > D - d$, we can fix $p' < p$ such that $\tilde \alpha := p' \alpha > D - d$ and $\tilde p := p / p' > 1$. Define $\tilde f = f ^{p'}$, then $\tilde f \in L ^{\tilde p} (\Rd)$. Apply the above proof for $\tilde f$, $\tilde \alpha$ and $\tilde p$ gives
        \begin{align*}
            \nor{(\A _{\tilde \alpha} [\tilde f]) ^{1 - \frac{D - d}{\tilde p \tilde \alpha}}} _{L ^{\tilde p} (\Gamma)} \le C \nor*{\tilde f} _{L ^{\tilde p} (\Rd)} = \nor{f} _{L ^p} ^{p'}.
        \end{align*}
        By Lemma \ref{lem:nonlinear-scaling}, we know that $\A _{\alpha} f \le (\A _{\tilde \alpha} [\tilde f]) ^\frac1{p'}$, so 
        \begin{align*}
            \nor{(\Aa [f]) ^{1 - \frac{D - d}{p \alpha}}} _{L ^p (\Gamma)} \le \nor{(\A _{\tilde \alpha} [\tilde f]) ^{\frac1{p'}(1 - \frac{D - d}{p \alpha})}} _{L ^p (\Gamma)} = \nor{(\A _{\tilde \alpha} [\tilde f]) ^{1 - \frac{D - d}{\tilde p \tilde \alpha}}} _{L ^{\tilde p} (\Gamma)} ^\frac1{p'} \le\nor{f} _{L ^p}.
        \end{align*}
    \end{enumerate}
    This completes the proof of the theorem.
\end{proof}

When $D = d$, $\Gamma = \Rd$, Theorem \ref{thm:avg-space} simply says 
\begin{align*}
    \nor{\Aa [f]} _{L ^{1, \infty} (\Rd)} \le C \nor{f} _{L ^{1} (\Rd)}, \qquad \nor{\Aa [f]} _{L ^{p} (\Rd)} \le C (p) \nor{f} _{L ^{p} (\Rd)}.
\end{align*}
This can also follow directly from the fact that $\Aa [f] \le \mm f$ pointwise by Lemma \ref{lem:average-name}, where $\mm f (x) = \sup _{\rho > 0} f _\rho (x)$ is the Hardy-Littlewood maximal function of $f$. 

\eqref{enu:strongtype-space} is false for $p = 1$. Indeed, unlike the maximal operator, when $f = |B _1| \delta _0$ is a Dirac measure, $\alpha \in (0, d)$, we can verify that 
\begin{align*}
    f _\rho (x) = \rho ^{-d} \ind*{|x| < \rho}, \qquad \Sa [f] (x) = |x| + \infty \ind*{|x| > 1}.
\end{align*}
So when $f \approx |B _1| \delta _0$ is an $L ^1$ function, $\Aa [f] (x) = \Sa [f] ^{-\alpha} (x) \approx |x| ^{-\alpha} \ind*{|x| < 1}$ is still integrable. However, in Appendix \ref{app:cantor} we construct an example where $f$ is a bounded Cantor-style measure and $\Aa [f]$ is $L ^{1, \infty} (\R)$ but not $L ^{1, q} (\R)$ for any $q < \infty$. 

When $D > d$, Theorem \ref{thm:avg-space} cannot hold for the maximal function anymore, as the maximal function does not have a trace on hypersurfaces.

\section{Blow-up technique and trace estimate: spacetime case}
\label{sec:spacetime}

We want to extend the theory in Section \ref{sec:space} to the spacetime setting. Due to our particular interest in equations with heat diffusion, we will focus on the parabolic scaling between time and space. That is, we define a cylinder $Q _\rho (t, x)$ to be the following: 
\begin{align}
    \label{eqn:Qe-spacetime}
    Q _\rho (t, x) = \set{(s, y): s \in (t - \rho ^2, t], y \in B _\rho (x)}.
\end{align}
We remark that other scalings are possible and similar results can be obtained.
$Q _\rho (t, x)$ will replace the role of $B _\rho (x)$ in the space setting. To be precise, we define the following. 

\begin{definition}
    \label{def:SASingReg-spacetime}
    Let $f: \R \times \Rd \to [0, \infty]$ be a Lebesgue measurable function. Denote 
    \begin{align*}
        f _\rho (t, x) := \fint _{Q _\rho (t, x)} f (s, y) \d y \d s, \qquad 0 < \rho < \infty.
    \end{align*}
    For $\alpha > 0$, we define the \textbf{scale operator} $\Sa [f]: \R \times \Rd \to [0, \infty]$ by 
    \begin{align}
        \label{def:E-spacetime}
        \Sa [f] (t, x) &:= \inf _{0 < \rho < \infty} \set{
            \rho: f _\rho (t, x) > \rho ^{-\alpha}
        }, \qquad (t, x) \in \R \times \Rd.
    \end{align}
    We define the \textbf{averaging operator} $\Aa [f]:\R \times \Rd \to [0, \infty]$ by 
    \begin{align}
        \label{def:A-spacetime}
        \Aa [f] (t, x) &:= \Sa [f] (t, x) ^{-\alpha}, \qquad (t, x) \in \R \times \Rd.
    \end{align}
    We define the \textbf{singular set} and \textbf{regular set} as 
    \begin{align*}
        \Sing _\alpha (f) = \set{(t, x) \in \R \times \Rd: \Sa [f] (t, x) = 0}, \\ \Reg _\alpha (f) = \set{(t, x) \in \R \times \Rd: \Sa [f] (t, x) > 0}.
    \end{align*}
\end{definition}
For simplicity we used the same notation as in the last section, but its meaning should be detected from the context based on the domain of $f$.

Clearly, Lemma \ref{lem:inner-outer-regularity} holds for cylinders as well. Thus we can parallel the semicontinuity, continuity, nonlinearity, and quasiconvexity results in Section \ref{sec:space}. The proof is identical to Lemma \ref{lem:continuity-space}-\ref{lem:sublin-space}, so we omit it here.

\begin{lemma}
    \label{lem:Aa-basic-spacetime}
    Let $\alpha > 0$ and let $f : \R \times \Rd \to [0, \infty]$ be measurable. Then
    $\Sa [f]$ is upper semicontinuous, and $\Aa [f]$ is lower semicontinuous. For any $p \in [1, \infty)$, it holds that 
    \begin{align*}
        \Sa [f] \ge \mathscr S _{p\alpha} [f ^p], \qquad \Aa [f] \le \A _{p\alpha} [f ^p] ^\frac1p.
    \end{align*}
    Moreover, if $f \in L ^1 _\loc (\R \times \Rd)$, then 
    \begin{align*}
        \Aa [f] (t, x) = f _{\Sa [f] (t, x)} (t, x), \qquad \text{ when } 0 < \Sa [f] (t, x) < \infty.
    \end{align*} 
    For any nonnegative $f, g \in  L ^1 _{\loc} (\R \times \Rd)$, $\lambda \in [0, 1]$, it holds that
    \begin{align*}
        \Aa \bkt{(1 - \lambda) f + \lambda g} (t, x) \le \maxs{
            \Aa [f] (t, x), \Aa [g] (t, x)
        }, \qquad \forall (t, x) \in \R \times \Rd.
    \end{align*} 
\end{lemma}

\subsection{Partial regularity and trace estimate}

The partial regularity theory for spacetime should be modified in order to fit the parabolic scaling. For $\beta \ge 0$, we introduce parabolic Hausdorff measure as 
\begin{align*}
    \mathscr P ^\beta (A) := \lim _{\delta \to 0} \mathscr P ^\beta _\delta (A) = \lim _{\delta \to 0} \inf \set{
        \sum _j r _j ^\beta: A \subset \bigcup _j Q _{r _j} (t _j, x _j)
    }, \qquad A \subset \R \times \Rd,
\end{align*} 
where infimum is taken over all possible covering of the set $A$ by cylinders. Accordingly, we obtain the following partial regularity theorem.

\begin{proposition}
    \label{prop:partial-regularity-spacetime}
    Suppose $f \in L ^p _{\loc} (\R \times \Rd)$ is nonnegative for some $p \in [1, \infty)$. If $0 < p \alpha < D + 2$ then the $(D + 2 - p \alpha)$-dimensional parabolic Hausdorff measure of $\Sing _\alpha (f)$ is zero:
    \begin{align*}
        \mathscr P ^{D + 2 - p \alpha} (\Sing _\alpha (f)) = 0.
    \end{align*}
    If $p \alpha \ge D + 2$ then $\Sing _\alpha (f)$ is empty.
\end{proposition}

The proof is identical to the proof of Proposition \ref{prop:partial-regularity}, except $|Q _r| = C r ^{D + 2}$ comparing to $|B _r| = C r ^D$.

It becomes slightly more delicate to estimate the trace in the parabolic setting. Recall that $\Gamma _T$, which satisfies Assumption \ref{ass:Gamma}, is not necessarily a manifold in $\R \times \Rd$, but can be discontinuous in time. The purpose of allowing this generality will be clear when we would like to replace $L ^\infty$ norms in $x$ by taking traces, but let us defer this discussion to the next subsection.

First, we revise Lemma \ref{lem:ap} to the following time-dependent form.

\begin{lemma}
    \label{lem:atp-spacetime}
    Let $t \in \R$, $\alpha > 0$, $p \in [1, \infty)$, and $f \in L ^p _{\loc, t} L ^p _x (\domain)$. Denote 
    \begin{align*}
        A _t (\rho) &:= \set{
            x' \in \Gamma _t:
            \rho \le \Sa [f] (t, x') < 2 \rho
        }, \qquad \rho > 0.
    \end{align*}
    Then for every $t \in (0, T)$, $\rho > 0$, 
    \begin{align*}
        \mathscr H ^d \pth{A _t (\rho)} \le C \rho ^{-D + d - 2 + p \alpha} \int _{t - 4 \rho ^2} ^{t} \intRd f (s, x) ^p \dx \d s.
    \end{align*}
\end{lemma}

\begin{proof}
    For any $x' \in A _t (\rho)$, by definition \eqref{def:E-spacetime}, there exists $\rho _{x'} \in [\rho, 2\rho)$ such that 
    \begin{align*}
        f _{\rho _{x'}} (t, x') > \rho _{x'} ^{-\alpha} \qquad \implies \qquad \rho _{x'} ^{D + 2 - p \alpha} < \int _{Q _{\rho _{x'}} (t, x')} f (s, y) ^p \d y \d s.
    \end{align*}
    Here we used Jensen's inequality as in Lemma \ref{lem:nonlinear-scaling}.
    Same as Lemma \ref{lem:ap}, Lipschitzness of $\Gamma _t$ implies 
    \begin{align}
        \label{eqn:Arho1-spacetime}
        \mathscr H ^{d} (A _t (\rho)) \le (1 + \nor{g _t} _{\Lip} ^2) ^{\frac{d}2}  \frac{\mathscr L ^D (\mathcal U _\rho (A _t (\rho)))}{c _{D - d} \rho ^{D - d}}.
    \end{align}
    Again, $\mathcal U _\rho (A _t (\rho))$ is covered by $\bigcup _{x' \in A _t (\rho)} B _{\rho _{x'}} (x')$, so Vitali covering lemma gives a disjoint subcollection $\set{B _{\rho _i} (x' _i)} _i$ with
    \begin{align}
        \mathscr L ^{D} (\mathcal U _\rho (A _t (\rho))) \le |B _1| 5 ^{D} \sum _i \rho _i ^{D} = C \rho ^{p \alpha - 2} \sum _i \rho _i ^{D + 2 - p \alpha}.
    \end{align}
    Finally, since $B _{\rho _i} (x' _i)$ are pairwise disjoint, $Q _{\rho _i} (t, x' _i)$ are also pairwise disjoint, so 
    \begin{align}
        \label{eqn:Arho2-spacetime}
        \sum _i \rho _i ^{D + 2 - p \alpha} \le C \sum _i \int _{Q _{\rho _i} (t, x' _i)} f (s, y) ^p \dy \d s \le C \int _{t - 4 \rho ^2} ^t \int _{\Rd} f ^p \dx \d t.
    \end{align}
    The last inequality used the fact that each $Q _{\rho _i} (t, x _i')$ has at most a timespan of $4 \rho ^2$. Combining \eqref{eqn:Arho1-spacetime}-\eqref{eqn:Arho2-spacetime} finishes the proof of the lemma.
\end{proof}

Using a similar idea, we can also introduce the following finer version of partial regularity comparing to Proposition \ref{prop:partial-regularity-spacetime}, which deals with anisotropic norms on $f$ and separates dimensions of singularity in space and time.

\begin{proposition}
    \label{lem:at0-spacetime}
    Let $\alpha > 0$, and $f \in L ^p _{\loc, t} L ^q _{\loc, x} (\domain)$, $1 \le q \le p < \infty$. Denote the singular section by 
    \begin{align*}
        S _t := \set{
            x' \in \Gamma _t: (t, x') \in \Sing _\alpha (f)
        }, \qquad t \in \R,
    \end{align*}
    and denote the set of singular time by 
    \begin{align*}
        \mathcal T = \set{t \in \R: \mathscr H ^d (S _t) > 0}.
    \end{align*}
    \begin{enumerate}[\upshape (a)]
        \item \label{enu:at0-spacetime} If $\frac{D - d}q < \alpha \le \frac 2p + \frac{D - d}q$, then $\mathcal T$ has Hausdorff dimension no greater than
        \begin{align*}
            \dim _{\mathscr H} (\mathcal T) \le 1 - \frac p2 \pth{\alpha - \frac{D - d}{q}}.
        \end{align*}

        \item \label{enu:at0-spacetime-reg} If $\alpha > \frac 2p + \frac{D - d}q$, then $\mathcal T = \varnothing$.
    \end{enumerate}
    
\end{proposition}

\begin{proof}
    The proposition is equivalent to the following statement: for any $\beta \in [0, 1]$ with $\frac{2 \beta}p + \frac{D - d}q < \alpha$, it holds that $\mathscr H ^d (S _t) = 0$ for $\mathscr H ^{1-\beta}$-a.e. $t \in \R$. 

    Same as in the proof of Proposition \ref{prop:partial-regularity}, we can assume $f \in L ^p _{t} L ^q _x (\domain)$ without loss of generality.
    For any $\eta > 0$, we define 
    \begin{align*}
        \mathcal T (\eta) = \set{
            t \in \R: \limsup _{\rho \to 0} \rho ^{2\beta} \fint _{t - (2 \rho) ^2} ^t \nor{f (s)} _{L ^q (\R ^D)} ^p \d s \ge \eta 
        }.
    \end{align*}
    We claim that $\mathcal T \subset \mathcal T (\eta)$ for any $\eta > 0$. Indeed, if $t \notin \mathcal T (\eta)$, then for $\e > 0$ sufficiently small, for all $\rho \in (0, \e)$, it holds that 
    \begin{align*}
        \rho ^{2\beta} \fint _{t - (2 \rho) ^2} ^t \nor{f (s)} _{L ^q (\R ^D)} ^p \d s < \eta.
    \end{align*}
    If $x' \in S _t$, there exists $\rho _{x'} \in (0, \e)$ such that $f _{\rho _{x'}} (t, x') > \rho _{x'} ^{-\alpha}$. Denote 
    \begin{align*}
        S _t (\rho) = \set{
            x' \in S _t: \rho \le \rho _{x'} < 2 \rho
        }, \qquad \rho > 0.
    \end{align*}
    Note that in the proof of Lemma \ref{lem:atp-spacetime}, the only feature we need from $A _t (\rho)$ is to be able to find a $\rho _{x'} \in [\rho, 2\rho)$ with $f _{\rho _{x'}} (t, x') > \rho _{x'} ^{-\alpha}$. The same proof yields
    \begin{align*}
        \mathscr H ^d (S _t (\rho)) &\le C \rho ^{-D + d - 2 + q \alpha} \int _{(t - (2 \rho) ^2, t) \times \R ^D} f (s, y) ^q \d y \d s \\
        & = C \rho ^{-D + d + q \alpha} \fint _{t - (2 \rho) ^2} ^t \nor{f (s)} _{L ^q (\Rd)} ^q \d s \\
        &\le C \rho ^{-D + d + q \alpha} \pth{\fint _{t - (2 \rho) ^2} ^t \nor{f (s)} _{L ^q (\Rd)} ^p \d s} ^\frac qp \\ 
        & \le C \rho ^{-D + d + q \alpha} (\eta \rho ^{-2 \beta}) ^\frac qp \int _{t - (2 \rho) ^2} ^t \nor{f (s)} _{L ^q (\Rd)} ^p \d s \\
        & \le C \eta ^\frac qp \rho ^{-D + d + q (\alpha - 2 \beta / p)} \int _{t - (2 \rho) ^2} ^t \nor{f (s)} _{L ^q (\Rd)} ^p \d s.
    \end{align*}
    Recall that $-D + d + q (\alpha - 2 \beta / p) > 0$.
    Since $\set{S _t (2 ^{-k} \e)} _{k \ge 1}$ forms a partition of $S _t$, we can control the measure of $S _t$ by 
    \begin{align*}
        \mathscr H ^d (S _t) \le \sum _{k = 1} ^\infty \mathscr H ^d (S _t (2 ^{-k} \e)) \le C \eta ^\frac qp \e ^{-D + d + q (\alpha - 2 \beta / p)}.
    \end{align*}
    This holds for all $\e$ sufficiently small, so $\mathscr H ^d (S _t) = 0$, $t \notin \mathcal T$.

    It remains to show that $\mathcal T$ has $\mathscr H ^{1-\beta}$-measure zero. Fix $\eta > 0$ and $\e > 0$. For every $t \in \mathcal T (\eta)$, there exists $\rho _t > 0$ such that $20 \rho _t ^2 < \e$ and
    \begin{align*}
        \rho _t ^{2\beta} \fint _{t - (2 \rho _t) ^2} ^t \nor{f (s)} _{L ^q (\Rd)} ^p \d s \ge \eta.
    \end{align*}
    Intervals $\set{(t - (2 \rho _t) ^2, t + (2 \rho _t) ^2)} _{t \in \mathcal T (\eta)}$ forms an open cover of $\mathcal T (\eta)$, so we can find a disjoint subcollection $\set{(t _i - (2 \rho _i) ^2, t _i + (2 \rho _i) ^2)} _i$ by Vitali covering lemma such that 
    \begin{align*}
        \mathcal T \subset \mathcal T (\eta) \subset \bigcup _i (t _i - 5 (2 \rho _i) ^2, t _i + 5 (2 \rho _i) ^2).
    \end{align*}
    So the $\mathscr H ^{1 - \beta}$-Hausdorff measure of $\mathcal T$ is bounded by 
    \begin{align*}
        \mathscr H ^{1 - \beta} (\mathcal T) \le C \sum _i \rho _i ^{2 (1 - \beta)} &\le \frac C\eta \int _{t _i - (2 \rho _i) ^2} ^{t _i} \nor{f (s)} _{L ^q (\Rd)} ^p \d s \\
        &\le \frac C\eta \int _{\mathcal T ^* (\eta)}\nor{f (s)} _{L ^q (\Rd)} ^p \d s \\
        & \le \frac C\eta \nor{f} _{L ^p _t L ^q _x (\domain)} ^p.
    \end{align*}
    The constant $C$ comes from the covering lemma and is independent of $\eta$. Since this is true for arbitrary large $\eta$, we must have $\mathscr H ^{1 - \beta} (\mathcal T) = 0$.
\end{proof}

Now we are ready to show the main theorem for trace estimates. We first work on isotropic norms. We show that when $p \alpha > D - d$ we can take trace in space, and if in addition $p \alpha > D - d + 2$ then we can take trace in both space and time.

\begin{theorem}
    \label{thm:avg-spacetime}
    Let $f \in L ^p _{\mathrm{loc}} (\domain)$ be a nonnegative function with some $p \in [1, \infty]$. Let $\alpha > 0$ satisfy $p \alpha > D - d$. Then
    \begin{enumerate}[\upshape (a)]
        \item $\Aa [f]$ is $\mu _T$-measurable, and 
        \begin{align*}
            \mu _T (\set{\Aa [f] = \infty}) = 0.
        \end{align*}
        \item If $p = 1$, $f \in L ^1 (\domain)$, then 
        \label{enu:weaktype-spacetime}
        \begin{align*}
            \nor{(\Aa [f]) ^{1 - \frac{D - d}{\alpha}}} _{L ^{1, \infty} _{t, x} (\Gamma _T)}\le C (\alpha, D, d, L) \nor{f} _{L ^1 (\domain)}.
        \end{align*}
        \item If $p > 1$, $f \in L ^p (\domain)$, then
        \label{enu:strongtype-spacetime}
        \begin{align*}
            \nor{(\Aa [f]) ^{1 - \frac{D - d}{p \alpha}}} _{L ^p _{t, x} (\Gamma _T)} \le C (p, \alpha, D, d, L) \nor{f} _{L ^p (\domain)}.
        \end{align*}
    \end{enumerate}
    Moreover, if $p \alpha > D - d + 2$, then 
    \begin{enumerate}[\upshape (a)]
        \setcounter{enumi}{3}
        \item For every $t \in (0, T)$, $\Aa [f] (t, \cdot)$ is $\mu _t$-measurable, and 
        \begin{align*}
            \mu _t (\set{\Aa [f] (t, \cdot) = \infty}) = 0.            
        \end{align*}

        \item If $f \in L ^1 (\domain)$, then for every $t \in (0, T)$, 
        \label{enu:weaktype-fixedt-spacetime}
        \begin{align*}
            \nor{[\Aa [f] (t)] ^{1 - \frac{D - d + 2}{\alpha}}} _{L ^{1, \infty} (\Gamma _t)}\le C (\alpha, D, d, L) \nor{f} _{L ^1 (\domain)}.
        \end{align*}
        \item If $f \in L ^p (\domain)$ for some $p \in (1, \infty]$, then for every $t \in (0, T)$, 
        \label{enu:strongtype-fixedt-spacetime}
        \begin{align*}
            \nor{[\Aa [f] (t)] ^{1 - \frac{D - d + 2}{p \alpha}}} _{L ^p (\Gamma _t)} \le C (p, \alpha, D, d, L) \nor{f} _{L ^p (\domain)}.
        \end{align*}
    \end{enumerate}
\end{theorem}

\begin{proof}
    Notice that for $\rho > 0$, $(t, x) \in \domain$, 
    \begin{align*}
        (2 \rho) ^{-\alpha} < \Aa [f] (t, x) \le \rho ^{-\alpha} &\iff 
        \rho \le \Sa [f] (t, x) < 2 \rho,
        \\
        \Aa [f] (t, x) = \infty &\iff \Sa [f] (t, x) = 0.
    \end{align*}
    Thus 
    \begin{align*}
        \set{(2 \rho) ^{-\alpha} < \Aa [f] (t) \le \rho ^{-\alpha}} \cap \Gamma &= A _t (\rho), \\
        \set{\Aa [f] (t) = \infty} \cap \Gamma &= S _t,
    \end{align*}
    where $A _t (\rho)$ is defined in Lemma \ref{lem:atp-spacetime}, and $S _t$ is defined in Lemma \ref{lem:at0-spacetime}. 

    \begin{enumerate}[\upshape (a)]
        \item The measurability of $\Aa [f]$ comes from semicontinuity. Lemma \ref{lem:at0-spacetime} \eqref{enu:at0-spacetime} implies $\mu _t (S _t) = 0$ for $\mathscr L ^1$-almost every $t \in [0, 1]$. Therefore 
        \begin{align*}
            \mu _T (\set{\Aa [f] = \infty}) = \int _0 ^T \mu _t (S _t) \d t = 0.
        \end{align*}

        \item Fix $\lambda > 0$, and pick $\rho _0 = \lambda ^{-\frac1\alpha}$, $\rho _k = 2 ^{-k} \rho _0$. Then 
        \begin{align*}
            \mu _T (\set{\Aa [f] > \lambda}) &= \sum _{k = 0} ^\infty \mu _T \pth{
                \set{\rho _k ^{-\alpha} < \Aa [f] \le \rho _{k + 1} ^{-\alpha}}
            } \\
            & \le \sum _{k = 0} ^\infty \int _0 ^T \mu _t (A _t (\rho _{k + 1})) \d t.
        \end{align*}
        By Lemma \ref{lem:atp-spacetime}, we know that  
        \begin{align*}
            \int _0 ^T \mu _t (A _t (\rho _{k + 1})) \d t & = C \rho _{k + 1} ^{-D + d - 2 + \alpha} \int _0 ^T \int _{t - (2 \rho _{k + 1}) ^2} ^t \int _{\R ^D} f (s, y) \d y \d s \d t \\
            & \le C \rho _{k + 1} ^{-D + d + \alpha} \int _{-\infty} ^T \int _{\R ^D} f (s, y) \dy \d s.
        \end{align*}
        Therefore, we conclude 
        \begin{align*}
            \mu _T (\set{\Aa [f] > \lambda}) &\le C \sum _{k = 0} ^\infty \rho _{k + 1} ^{-D + d + \alpha} \nor{f} _{L ^1 (\domain)} \\
            &\le C \rho _0 ^{\alpha - D + d} \nor{f} _{L ^1 (\domain)} = \frac C{\lambda ^{1 - \frac{D - d}{\alpha}}} \nor{f} _{L ^1 (\domain)}.
        \end{align*}
        Here we use $\alpha > D - d$ to compute the summation. We thus conclude that $\nor{(\Aa [f]) ^{1 - \frac{D - d}{\alpha}}} _{L ^{1, \infty} (\Gamma _T)} \le C \nor{f} _{L ^1 (\domain)}$.

        \item Since we again have semicontinuity and quasiconvexity as the spatial case, we can prove it using an identical Marcinkovich interpolation proof as in Theorem \ref{thm:avg-space} \eqref{enu:strongtype-space}. We repeat it here for completeness.
        
        For $p = \infty$, we fix any $x \in \Rd$, then 
        \begin{align*}
            f _\rho (t, x) < \nor{f} _{L ^\infty (\domain)} \le \rho ^{-\alpha}, \qquad \forall \rho < \nor{f} _{L ^\infty (\domain)} ^{-\frac1\alpha}.
        \end{align*}
        From the definition \eqref{def:E-spacetime} we know 
        \begin{align*}
            \Sa [f] (t, x) \ge \nor{f} _{L ^\infty (\domain)} ^{-\frac1\alpha} > 0,
        \end{align*}
        and correspondingly $\Aa [f] (t, x) \le \nor{f} _{L ^\infty (\domain)}$. 

        To show the boundedness for $p \in (1, \infty)$, we assume in addition that $\alpha > D - d$. This requirement can be lifted in the end by making a nonlinear transform using Lemma \ref{lem:nonlinear-scaling} as in the proof of Theorem \ref{thm:avg-space} \eqref{enu:strongtype-space}.
        
        For any $\lambda > 0$ we define 
        \begin{align*}
            f _1 = 2 f \ind*{f \le \frac\lambda2}, \qquad f _2 = 2 f \ind*{f > \frac\lambda2} .
        \end{align*}
        Then $f = \half \pth{f _1 + f _2}$. By the quasiconvexity in Lemma \ref{lem:Aa-basic-spacetime}, we conclude that 
        \begin{align*}
            \set{\Aa [f] > \lambda} \subset \set{\Aa [f _1] > \lambda} \cup \set{\Aa [f _2] > \lambda} = \set{\Aa [f _2] > \lambda},
        \end{align*}
        as $\Aa [f _1] \le \nor{f _1} _{L ^\infty (\domain)} \le \lambda$. Thus
        \begin{align*}
            \mu _T \pth{
                \set{\Aa [f] > \lambda}
            } &\le \mu _T \pth{
                \set{\Aa [f _2] > \lambda}
            } 
            \le \frac C{\lambda ^{1 - \frac {D - d}\alpha}} \nor{f _2} _{L ^1 (\domain)}.
        \end{align*}
        Here we used part \eqref{enu:weaktype-spacetime} on $f _2$. And the remaining is analogous to the classical setting:
        \begin{align*}
            \nor{(\Aa [f]) ^{1 - \frac{D - d}{p \alpha}}} _{L ^{p} (\Gamma _T)} ^{p} &= 
            \int _{\Gamma _T} (\Aa [f]) ^{p - \frac{D - d}{\alpha}} \d \mu _T \\
            &= C \int _0 ^\infty \lambda ^{p - \frac{D - d}{\alpha} - 1} \mu _T \pth{
                \set{\Aa [f] > \lambda}
            } \d \lambda
            \\
            & \le C \int _0 ^\infty \lambda ^{p - 2}\int _{\domain} f _2 \dx \dt \d \lambda 
            \\
            &= C \int _{\domain} \int _0 ^\infty f \ind*{f > \frac\lambda2} \lambda ^{p - 2} \d \lambda \dx \dt \\
            &\le C \int _{\domain} f ^p \dx .
        \end{align*}

        \item Again, the measurability comes from semicontinuity. Lemma \ref{lem:at0-spacetime} \eqref{enu:at0-spacetime-reg} implies $\mu _t (S _t) = 0$ for every $t \in (0, T)$. 
        
        \item The proof is analogous to part \eqref{enu:weaktype-spacetime}. Fix $\lambda > 0$, and pick $\rho _0 = \lambda ^{-\frac1\alpha}$, $\rho _k = 2 ^{-k} \rho _0$. Then 
        \begin{align*}
            \mu _t (\set{\Aa [f] (t) > \lambda}) &= \sum _{k = 0} ^\infty \mu _t \pth{
                \set{\rho _k ^{-\alpha} < \Aa [f] (t) \le \rho _{k + 1} ^{-\alpha}}
            } \\
            & \le \sum _{k = 0} ^\infty \mu _t (A _t (\rho _{k + 1})) \\
            & \le C \sum _{k = 0} ^\infty \rho _{k + 1} ^{-D + d - 2 + \alpha} \int _{t - (2 \rho _{k + 1}) ^2} ^t \intRd f (s, y) \d y \d s \\
            & \le C \rho _{0} ^{\alpha - D + d - 2} \nor{f} _{L ^1 (\domain)} \\
            & = \frac1{\lambda ^{1 - \frac{D - d + 2}{\alpha}}} \nor{f} _{L ^1 (\domain)}.
        \end{align*}
        Here we used $\alpha > D - d + 2$ to compute the summation.

        \item The proof is identical to part \eqref{enu:strongtype-spacetime} except the codimension is now $D - d + 2$ instead of $D - d$.
    \end{enumerate}
\end{proof}

\subsection{Anisotropic estimates}

By interpolation, we can extend the estimates in Theorem \ref{thm:avg-spacetime} to the following anisotropic bounds. 

\begin{proposition}
    \label{prop:anisotropic-spacetime}
    Let $f \in L ^p (\domain)$, $p \ge 1$. If $p \alpha > D - d + 2$, then for $0 < \gamma \le 1$
    \begin{align}
        \label{eqn:t>x-spacetime}
        \nor{(\Aa [f]) ^{1 - \frac{D - d + 2 \gamma}{\alpha}}} _{L ^{\frac1{1 - \gamma}, \infty} _t L ^{1, \infty} _x (\Gamma _T)} &\le C (\gamma, \alpha, D, d, L) \nor{f} _{L ^1 (\domain)}, & p = 1. \\
        \label{eqn:t>x-spacetime-p}
        \nor{(\Aa [f]) ^{1 - \frac{D - d + 2 \gamma}{p\alpha}}} _{L ^{\frac p{1 - \gamma}} _t L ^{p} _x (\Gamma _T)} &\le C (\gamma, p, \alpha, D, d, L) \nor{f} _{L ^p (\domain)}, & p > 1.
    \end{align}
    If $p\alpha > D$, then for $0 \le \beta < 1$ 
    \begin{align}
        \label{eqn:t<x-spacetime}
        \nor{(\Aa [f]) ^{1 - \frac{D - \beta d}{\alpha}}} _{L ^{1, \infty} _t L ^{\frac1{\beta}, \infty} _x (\Gamma _T)} &\le C (\gamma, \alpha, D, d, L) \nor{f} _{L ^1 (\domain)}, & p = 1. \\
        \label{eqn:t<x-spacetime-p}
        \nor{(\Aa [f]) ^{1 - \frac{D - \beta d}{p\alpha}}} _{L ^{p} _t L ^{\frac p{\beta}} _x (\Gamma _T)} &\le C (\gamma, p, \alpha, D, d, L) \nor{f} _{L ^p (\domain)}, & p > 1.
    \end{align}
    If $p\alpha > D + 2$, then for $0 < \gamma \le 1$, $0 \le \beta \le 1$ 
    \begin{align}
        \label{eqn:t?x-spacetime}
        \nor{(\Aa [f]) ^{1 - \frac{D - \beta d + 2 \gamma}{\alpha}}} _{L ^{\frac1{1 - \gamma}, \infty} _t L ^{\frac1{\beta}, \infty} _x (\Gamma _T)} \le C (\gamma, \beta, \alpha, D, d, L) \nor{f} _{L ^1 (\domain)}, p = 1. \\
        \label{eqn:t?x-spacetime-p}
        \nor{(\Aa [f]) ^{1 - \frac{D - \beta d + 2 \gamma}{p\alpha}}} _{L ^{\frac p{1 - \gamma}} _t L ^{\frac p{\beta}} _x (\Gamma _T)} \le C (\gamma, \beta, p, \alpha, D, d, L) \nor{f} _{L ^p (\domain)}, p > 1.
    \end{align}
\end{proposition}

\begin{proof}
    The case $\gamma = 1$ and $\beta = 0$ can be derived from Theorem \ref{thm:avg-spacetime} directly. Indeed, \eqref{eqn:t>x-spacetime} and \eqref{eqn:t>x-spacetime-p} with $\gamma = 1$ are direct consequences of \eqref{enu:weaktype-fixedt-spacetime} and \eqref{enu:strongtype-fixedt-spacetime}. Moreover, \eqref{eqn:t<x-spacetime} and \eqref{eqn:t<x-spacetime-p} with $\beta = 0$ are direct consequences of \eqref{enu:weaktype-spacetime} and \eqref{enu:strongtype-spacetime} with $d = 0$: we can choose $\Gamma _t$ to be a singleton $\set{x (t)}$ such that $\Aa [f] (t, x (t)) > \nor{\Aa [f] (t, \cdot)} _{L ^\infty (\Rd)} - \e$ for arbitrary small $\e$ (recall that there is no continuity requirement on the map $t \mapsto \Gamma _t$). Existence of such a choice is guaranteed by the measurable selection theorem.

    We easily derive \eqref{eqn:t>x-spacetime-p} and \eqref{eqn:t<x-spacetime-p} for $\gamma, \beta \in (0, 1)$ by interpolation. Indeed, when $\gamma = 0$ and $\beta = 1$, both degenerates to Theorem \ref{thm:avg-spacetime} \eqref{enu:strongtype-spacetime}.

    To derive \eqref{eqn:t>x-spacetime} and \eqref{eqn:t<x-spacetime} for $\gamma, \beta \in (0, 1)$, we can also use interpolation. 
    As explained in Appendix \ref{app:lorentz}, iterated weak norm and joint weak norm are not comparable. However, interpolation is possible between them. Denote $g = (\Aa [f]) ^{1 - \frac{D - d}{\alpha}}$. Since $\alpha > D - d$ in both cases, Theorem \ref{thm:avg-spacetime} \eqref{enu:weaktype-spacetime} shows 
    \begin{align*}
        \nor{g} _{L ^{1, \infty} _{t, x} (\Gamma _T)} \le C \nor{f} _{L ^1 (\domain)}.    
    \end{align*}
    \eqref{eqn:t>x-spacetime} with $\gamma = 1$ and \eqref{eqn:t<x-spacetime} with $\beta = 0$ translate to 
    \begin{align*}
        \nor{g} _{L ^\infty _t L ^{\frac{\alpha - D + d - 2}{\alpha - D + d}, \infty} _x (\Gamma _T)} &\le C \nor{f} _{L ^1 (\domain)} ^{\frac{\alpha - D + d}{\alpha - D + d - 2}}, \\
        \nor{g} _{L ^{\frac{\alpha - D}{\alpha - D + d}, \infty} _t L ^\infty _x (\Gamma _T)} &\le C \nor{f} _{L ^1 (\domain)} ^{\frac{\alpha - D + d}{\alpha - D}}.
    \end{align*}
    We now interpolate using Lemma \ref{lem:lorentz-interpolation}. Here $p _0 = \frac{\alpha - D}{\alpha - D + d}$, $q _0 = \frac{\alpha - D + d - 2}{\alpha - D + d}$. For some $p, q$ satisfying \eqref{eqn:pq-range} to be determined, we have 
    \begin{align*}
        \nor{g} _{L ^{p, \infty} _t L ^{q, \infty} _x (\Gamma _T)} &\le C \nor{f} _{L ^1 (\domain)} ^{\frac1p} \nor{f} _{L ^1 (\domain)} ^{\frac{\alpha - D + d}{\alpha - D + d - 2} (1 - \frac1p)} \\
        &= C \nor{f} _{L ^1 (\domain)} ^{\frac1p + \frac1{q _0} (1 - \frac1p)} = C \nor{f} _{L ^1 (\domain)} ^\frac1q.
    \end{align*}
    Therefore, 
    \begin{align*}
        \nor{g ^q} _{L ^{\frac pq, \infty} _t L ^{1, \infty} _x (\Gamma _T)} \le C \nor{f} _{L ^1 (\domain)}.
    \end{align*}
    Set $q = \frac{\alpha - D + d- 2 \gamma}{\alpha - D + d}$, then $g ^q = (\Aa [f]) ^{1 - \frac{D - d + 2 \gamma}{\alpha}}$, and from \eqref{eqn:pq-range} we know 
    \begin{align*}
        \frac{1 - q _0}{p} + \frac{q _0}{q} = 1 \implies \frac pq = \frac{1 - q _0}{q - q _0} = \frac1{1 - \gamma}.
    \end{align*}

    Similarly, for some $p, q$ satisfying \eqref{eqn:pq-range-2} to be determined, we have
    \begin{align*}
        \nor{g} _{L ^{p, \infty} _t L ^{q, \infty} _x (\Gamma _T)} &\le C \nor{f} _{L ^1 (\domain)} ^{\frac1q} \nor{f} _{L ^1 (\domain)} ^{\frac{\alpha - D + d}{\alpha - D} (1 - \frac1q)} \\
        &= C \nor{f} _{L ^1 (\domain)} ^{\frac1q + \frac1{p _0} (1 - \frac1q)} = C \nor{f} _{L ^1 (\domain)} ^\frac1p.
    \end{align*}
    Therefore, 
    \begin{align*}
        \nor{g ^p} _{L ^{1, \infty} _t L ^{\frac qp, \infty} _x (\Gamma _T)} \le C \nor{f} _{L ^1 (\domain)}.
    \end{align*}
    Set $p = \frac{\alpha - D + (1 - \gamma) d}{\alpha - D + d}$, then $g ^p = (\Aa [f]) ^{1 - \frac{D - (1 - \gamma) d}{\alpha}}$, and from \eqref{eqn:pq-range-2} we know
    \begin{align*}
        \frac{p _0}{p} + \frac{1 - p _0}{q} = 1 \implies \frac qp = \frac{1 - p _0}{p - p _0} = \frac1{\beta}.
    \end{align*}
    This proves \eqref{eqn:t>x-spacetime} and \eqref{eqn:t<x-spacetime}.

    Finally, \eqref{eqn:t?x-spacetime} is an interpolation of \eqref{eqn:t>x-spacetime} and \eqref{eqn:t>x-spacetime} with $d = 0$, while \eqref{eqn:t?x-spacetime-p} is an interpolation of \eqref{eqn:t>x-spacetime-p} and \eqref{eqn:t>x-spacetime-p} with $d = 0$, 
\end{proof}

\eqref{eqn:t>x-spacetime} can be slightly improved by the following proposition to \eqref{eqn:t>x-strong-spacetime}, where we achieve strong spatial integrability instead of weak integrability, and enlarge the range of $\alpha$. We could also work with anisotropic norms on $f$. 
\begin{proposition}
    \label{prop:anisotropic-2-spacetime}
    Let $0 < \gamma < \theta < 1 \le p \le \infty$ satisfy $p \alpha > {D - d + 2 \gamma}{}$. Then it holds that 
    \begin{align}
        \label{eqn:t>x-strong-spacetime}
        \nor{(\Aa [f]) ^{1 - \frac{D - d + 2 \gamma}{p \alpha}}} _{L ^{\frac p{1 - \gamma}, \infty} _t L ^p _x (\Gamma _T)} &\le C (p, \gamma, \alpha, D, d, L) \nor{f} _{L ^p (\domain)} , \\
        \label{eqn:anisotropic-f-spacetime}
        \nor{(\Aa [f]) ^{1 - \frac{D - d + 2 \gamma}{p \alpha}}} _{L ^{\frac p{\theta - \gamma}} _t L ^p _x (\Gamma _T)} &\le C (\theta, p, \gamma, \alpha, D, d, L) \nor{f} _{L ^\frac p\theta _t L ^p _x (\domain)}
        ,\\
        \label{eqn:anisotropic-f-2-spacetime}
        \nor{(\Aa [f]) ^{1 - \frac{D - d + 2 \gamma}{p \alpha}}} _{L ^{\infty} _t L ^{p, \infty} _x (\Gamma _T)} &\le C (p, \gamma, \alpha, D, d, L) \nor{f} _{L ^{\frac p\gamma} _t L ^p _x (\domain)}.
    \end{align}
    \eqref{eqn:anisotropic-f-2-spacetime} also holds for $\gamma = 0, 1$.
    In addition, if $p \alpha > {D + 2 \gamma}$, then for $\beta \in [0, 1]$ it holds that 
    \begin{align}
        \label{eqn:t>x-strong-2-spacetime}
        \nor{(\Aa [f]) ^{1 - \frac{D - \beta d + 2 \gamma}{p \alpha}}} _{L ^{\frac p{1 - \gamma}, \infty} _t L ^\frac p\beta _x (\Gamma _T)} &\le C (p, \gamma, \alpha, D, d, L) \nor{f} _{L ^p (\domain)} , \\
        \label{eqn:anisotropic-f-3-spacetime}
        \nor{(\Aa [f]) ^{1 - \frac{D - \beta d + 2 \gamma}{p \alpha}}} _{L ^{\frac p{\theta - \gamma}} _t L ^{\frac p \beta} _x (\Gamma _T)} &\le C (\theta, p, \gamma, \alpha, D, d, L) \nor{f} _{L ^\frac p\theta _t L ^p _x (\domain)}, \\
        \label{eqn:anisotropic-f-4-spacetime}
        \nor{(\Aa [f]) ^{1 - \frac{D - \beta d + 2 \gamma}{p \alpha}}} _{L ^{\infty} _t L ^{\frac p\beta, \infty} _x (\Gamma _T)} &\le C (p, \gamma, \alpha, D, d, L) \nor{f} _{L ^{\frac p\gamma} _t L ^p _x (\domain)}.
    \end{align}
    \eqref{eqn:anisotropic-f-4-spacetime} also holds for $\gamma = 0, 1$.
\end{proposition}

\begin{proof}

    The case $p = \infty$ reduces to Theorem \ref{thm:avg-spacetime} \eqref{enu:strongtype-spacetime}, so we assume $p < \infty$ from now on. We will prove \eqref{eqn:t>x-strong-spacetime}-\eqref{eqn:anisotropic-f-2-spacetime} first, and prove \eqref{eqn:t>x-strong-2-spacetime}-\eqref{eqn:anisotropic-f-4-spacetime} using interpolation.
    
    \noindent {\bf Proof of \eqref{eqn:t>x-strong-spacetime}-\eqref{eqn:anisotropic-f-2-spacetime}.}  
    Denote $q = p (1 - \frac{D - d + 2 \gamma}{p\alpha}) > 0$. Then
    
    \begin{align*}
        \int _{\Gamma _t} (\Aa [f]) ^q (t, x) \d \mu _t &= \int _0 ^\infty q \lambda ^{q - 1} \mu _t (\set{\Aa [f] > \lambda}) \d \lambda \\
        &= 2 ^{-\alpha q} \int _0 ^\infty q \lambda ^{q - 1} \mu _t (\set{\Aa [f] > 2 ^{-\alpha} \lambda}) \d \lambda.
    \end{align*} 
    Taking the difference of two integrals yields (recall $A _t (\rho)$ defined in Lemma \ref{lem:atp-spacetime})
    \begin{align*}
        \int _{\Gamma _t} (\Aa [f]) ^q (t, x) \mu _t (\dx) &= \frac1{2 ^{\alpha q} - 1} \int _0 ^\infty q \lambda ^{q - 1} \mu _t (\set{2 ^{-\alpha} \lambda < \Aa [f] \le \lambda}) \d \lambda 
        \\
        &\le \frac1{2 ^{\alpha q} - 1} \int _0
        ^\infty q \lambda ^{q - 1} \mu _t (A _t (\lambda ^{-\frac1\alpha})) \d \lambda 
        \\
        &\le C \int _0
        ^\infty \lambda ^{q - 1} \lambda ^{\frac{D - d + 2 - p \alpha}\alpha} \int _{t - \lambda ^{-\frac2\alpha}} ^{t} \intRd f (s, x) ^p \dx \d s \d \lambda 
        \\
        &= C \int _0 ^\infty \int _0 
        ^{s ^{-\frac\alpha2}} \lambda ^{\frac{2}\alpha  (1 - \gamma) - 1} \d \lambda \int _{\Rd} f (t - s, x) ^p \d x \d s 
        \\
        &= \frac{C \alpha}{2 (1 - \gamma)} \int _0 ^\infty s ^{-1 + \gamma} \nor{f (t - s)} _{L ^p (\Rd)} ^p \d s.
    \end{align*}
    Since $t \mapsto t ^{-1 + \gamma}$ is in $L ^{\frac1{1 - \gamma}, \infty} (\R _+)$, the generalized Young's convolution inequality in Lemma \ref{lem:convolution} directly implies
    \begin{align*}
        \nor{(\Aa [f]) ^q} _{L ^{\frac1{1 - \gamma}, \infty} _t L ^1 _{x} (\Gamma _T)} &\le C \nor{f} _{L ^p _t L ^p _x (\domain)} ^p, \\
        \nor{(\Aa [f]) ^q} _{L ^{\frac1{\theta - \gamma}} _t L ^1 _x (\Gamma _T)} &\le C \nor{f} _{L ^\frac p\theta _t L ^p _x (\domain)} ^p.
    \end{align*}
    This completes the proof for \eqref{eqn:t>x-strong-spacetime}-\eqref{eqn:anisotropic-f-spacetime}.

    To prove \eqref{eqn:anisotropic-f-2-spacetime}, note that for any $\lambda > 0$, 
    \begin{align*}
        \mu _t (\set{2 ^{-\alpha} \lambda < \Aa [f] (t) \le \lambda}) 
        &\le C \lambda ^{\frac{D - d + 2 - p \alpha}\alpha} \int _{t - \lambda ^{-\frac2\alpha}} ^{t} \intRd f (s, x) ^p \dx \d s \d \lambda \\
        &= C \lambda ^{\frac{D - d - p \alpha}\alpha} \fint _{t - \lambda ^{-\frac2\alpha}} ^{t} \nor{f (s)} _{L ^p (\Rd)} ^p \d s.
    \end{align*}
    If $\gamma \in (0, 1]$, then by Jensen's inequality, 
    \begin{align*}
        \mu _t (\set{2 ^{-\alpha} \lambda < \Aa [f] (t) \le \lambda}) 
        &\le C \lambda ^{\frac{D - d - p \alpha}\alpha} \pth{\fint _{t - \lambda ^{-\frac2\alpha}} ^{t} \nor{f (s)} _{L ^p (\Rd)} ^{\frac p\gamma} \d s } ^\gamma \\
        &= C \lambda ^{\frac{D - d + 2 \gamma - p \alpha}\alpha} \pth{\int _{t - \lambda ^{-\frac2\alpha}} ^{t} \nor{f (s)} _{L ^p (\Rd)} ^{\frac p\gamma} \d s } ^\gamma \\
        &\le C \lambda ^{-q} \nor{f} _{L ^{\frac p \gamma} _t L ^p _x (\domain)} ^p.
    \end{align*}
    If $\gamma = 0$, then 
    \begin{align*}
        \mu _t (\set{2 ^{-\alpha} \lambda < \Aa [f] (t) \le \lambda}) 
        &\le C \lambda ^{\frac{D - d - p \alpha}\alpha} \nor{f} _{L ^\infty _t L ^p _x (\domain)} ^p \\
        &= C \lambda ^{-q} \nor{f} _{L ^{\frac p \gamma} _t L ^p _x (\domain)} ^p.
    \end{align*}
    Therefore, for every $t \in \R$,
    \begin{align*}
        \nor{\Aa [f] (t)} _{L ^{q, \infty} (\Gamma _t)} ^q \le C \nor{f} _{L ^{\frac p \gamma} _t L ^p _x (\domain)} ^p, \\
        \nor{\Aa [f] (t) ^\frac qp} _{L ^{p, \infty} (\Gamma _t)} \le C \nor{f} _{L ^{\frac p \gamma} _t L ^p _x (\domain)}.
    \end{align*}

    \noindent {\bf Proof of \eqref{eqn:t>x-strong-2-spacetime}-\eqref{eqn:anisotropic-f-4-spacetime}.}  
    If $p \alpha > {D + 2 \gamma}$, by setting $d = 0$ in \eqref{eqn:anisotropic-f-spacetime}, we have 
    \begin{align*}
        \nor{(\Aa [f]) ^{1 - \frac{D + 2 \gamma}{p \alpha}}} _{L ^{\frac p{\theta - \gamma}} _t L ^\infty _x (\Gamma _T)} &\le C (\theta, p, \gamma, \alpha, D, d, L) \nor{f} _{L ^\frac p\theta _t L ^p _x (\domain)}.
    \end{align*}
    This is proven using the same strategy for \eqref{eqn:t<x-spacetime}-\eqref{eqn:t<x-spacetime-p}.
    \eqref{eqn:anisotropic-f-3-spacetime} is an interpolation between this and \eqref{eqn:anisotropic-f-spacetime}. \eqref{eqn:t>x-strong-2-spacetime} and \eqref{eqn:anisotropic-f-4-spacetime} can also be proven by interpolation.
\end{proof}

We summarize Theorem \ref{thm:avg-spacetime}, Proposition \ref{prop:anisotropic-spacetime}, and Proposition \ref{prop:anisotropic-2-spacetime} in the following unified form.

\begin{corollary}
    \label{cor:p1q1p2q2}
    Let $0 < p _1, q _1, p _2, q _2 \le \infty$ and $\alpha > 0$. Define $r _1, r _2 \in \R \setminus \{0\} \cup \{\infty\}$ by 
    \begin{align*}
        \frac1{r _1} = \alpha - \frac{2}{p _1} - \frac{D}{q _1}, \qquad \frac1{r _2} = \alpha - \frac{2}{p _2} - \frac{d}{q _2}.
    \end{align*}
    Suppose $r _2 = \lambda r _1$ for some $\lambda \in (0, \infty)$, $p _2 \ge \lambda p _1$, $q _2 \ge \lambda q _1$, $1 \le q _1 \le p _1$, and $f \in L ^{p _1} _t L ^{q _1} _x (\domain)$. 
    
    \begin{enumerate}[\upshape (A)]
        \item \label{enu:weakweak} Suppose $p _1 = q _1 = p \in [1, \infty]$, and assume $\frac 1{r _1} + \frac2{p _1} \cdot \ind*{p _2 = \lambda p _1} + \frac d{q _1} \cdot \ind*{q _2 = \lambda q _1} > 0$.
        
        \begin{enumerate}[\upshape (a)]
            \item \label{enu:weakweak1} If $p = 1$, then 
            \begin{align*}
                \nor{\Aa [f]} _{L ^{\lambda, \infty} _{t, x} (\Gamma _T)} ^{\lambda} &\le C \nor{f} _{L ^1 _{t, x} (\domain)} \qquad \text{ if } p _2 = q _2 = \lambda, \\
                \nor{\Aa [f]} _{L ^{p _2, \infty} _t L ^{q _2, \infty} _x (\Gamma _T)} ^{\lambda} &\le C \nor{f} _{L ^1 _{t, x} (\domain)} \qquad \text{ if } p _2 > \lambda \text{ or }  q _2 > \lambda.
            \end{align*}
    
            \item \label{enu:p=q>1} If $p > 1$, then 
            \begin{align*}
                \nor{\Aa [f]} _{L ^{p _2} _t L ^{q _2} _x (\Gamma _T)} ^{\lambda} \le C \nor{f} _{L ^{p} _{t, x} (\domain)}.
            \end{align*}
        \end{enumerate}

        \item \label{enu:weakstrong} Suppose $\lambda p _1 < p _2 < \infty$. When $\lambda q _1 < q _2$, we further require $\alpha > \frac{d}{q _2}$.
        
        \begin{enumerate}[\upshape (a)]
            \item If $p _1 = q _1$, then 
            \begin{align*}
                \nor{\Aa [f]} _{L ^{p _2, \infty} _t L ^{q _2} _x (\Gamma _T)} ^{\lambda} \le C \nor{f} _{L ^{p _1} _t L ^{q _1} _x (\domain)}.
            \end{align*}

            \item \label{enu:Bb-spacetime} If $p _1 > q _1$, then 
            \begin{align*}
                \nor{\Aa [f]} _{L ^{p _2} _t L ^{q _2} _x (\Gamma _T)} ^{\lambda} \le C \nor{f} _{L ^{p _1} _t L ^{q _1} _x (\domain)}.
            \end{align*}
        \end{enumerate}

        \item \label{enu:C}
        Suppose $p _2 = \infty$. When $\lambda q _1 < q _2$, we further require $\alpha > \frac{d}{q _2}$. Then 
        \begin{align*}
            \nor{\Aa [f]} _{L ^{\infty} _t L ^{q _2, \infty} _x (\Gamma _T)} ^{\lambda} \le C \nor{f} _{L ^{p _1} _t L ^{q _1} _x (\domain)}.
        \end{align*}
    \end{enumerate}
    In all of the above, the constant $C = C (p _1, q _1, p _2, q _2, \alpha, D, d, L)$.
\end{corollary}

\begin{proof}
    When $q _1 = \infty$, we have $p _1 = p _2 = q _2 = \infty$, $r _1 = r _2 = \alpha ^{-1}$ so $\lambda = 1$. \eqref{enu:p=q>1} and \eqref{enu:C} now are simply Theorem \ref{thm:avg-spacetime} \eqref{enu:strongtype-spacetime} with $p = \infty$. We now focus on the case $q _1 < \infty$. \eqref{enu:weakweak}-\eqref{enu:weakstrong} are translations of Proposition \ref{prop:anisotropic-spacetime}-\ref{prop:anisotropic-2-spacetime} under the following change of parameters:
    \begin{align*}
        \begin{cases}
            p _1 = \frac p\theta \\
            q _1 = p \\
            p _2 = \lambda \cdot \frac p{\theta - \gamma} \\
            q _2 = \lambda \cdot \frac p\beta 
        \end{cases}
        \impliedby
        \begin{cases}
            p = q _1 \\ 
            \theta = \frac{q _1}{p _1} \\
            \beta = \lambda \cdot \frac{q _1}{q _2} \\
            \gamma = \frac{q _1}{p _1} - \lambda \cdot \frac{q _1}{p _2} 
        \end{cases}
    \end{align*}
    with $\lambda$ satisfying $r _2 = \lambda r _1$. Note that 
    \begin{align*}
        \alpha \pth{1 - \frac{D - \beta d + 2 \gamma}{p \alpha}} &= \alpha - \frac Dp + \frac {\beta d} p - \frac{2 \gamma} p = \alpha - \frac D{q _1} + \frac{\lambda d}{q _2} - 2 \pth{\frac1{p _1} - \frac{\lambda}{p _2}} \\
        &= \alpha - \frac D{q _1} - \frac2{p _1} + \lambda \pth{\frac d{q _2} + \frac{2}{p _2}} = \frac1{r _1} + \lambda \pth{\alpha - \frac1{r _2}} = \lambda \alpha.
    \end{align*}
    So 
    \begin{align*}
        \lambda = 1 - \frac{D - \beta d + 2 \gamma}{p \alpha}.
    \end{align*}

    First, we verify that $p, \theta, \beta, \gamma$ are within the range prescribed by Proposition \ref{prop:anisotropic-spacetime}-\ref{prop:anisotropic-2-spacetime} provided
    $q _1 < \infty$. 
    \begin{itemize}
        \item $p = q _1 \in [1, \infty]$.
        \item $\theta = \frac{q _1}{p _1} \in [0, 1]$ because $q _1 \le p _1 < \infty$. $\theta = 1$ iff $p _1 = q _1$. $\theta = 0$ iff $p _1 = \infty$.
        \item $\beta = \lambda \cdot \frac{q _1}{q _2} \in [0, 1]$ because $q _2 \ge \lambda q _1$ and $q _1 < \infty$. $\beta = 0$ iff $q _2 = \infty$. $\beta = 1$ iff $q _2 = \lambda q _1$.
        \item $\gamma > 0$ if $p _2 > \lambda p _1$ because 
        \begin{itemize}
            \item when $p _2 < \infty$, $\gamma = \frac{q _1}{p _1 p _2} (p _2 - \lambda p _1) > 0$.
            \item when $p _2 = \infty$, $\gamma = \frac{q _1}{p _1} > 0$.
        \end{itemize}
        \item $\gamma = 0$ if $p _2 = \lambda p _1$.
        \item $\gamma \le \frac{q _1}{p _1} = \theta$. Equality holds if and only if $p _2 = \infty$.
    \end{itemize}
    We delay the verification of assumptions on parameters to the end of the proof.

    \begin{enumerate}[\upshape (A)]
        \item This will be proven using Theorem \ref{thm:avg-spacetime} and Proposition \ref{prop:anisotropic-spacetime}. 
        We separate four cases.
        
        \begin{enumerate}[\upshape (1)]
            \item When $p _2 = \lambda p _1$ and $q _2 = \lambda q _1$, we have $\gamma = 0$ and $\beta = 1$, thus $\lambda = 1 - \frac{D - d}{p \alpha}$. When $p _1 = q _1 = 1$, by Theorem \ref{thm:avg-spacetime} \eqref{enu:weaktype-spacetime}, we have 
            \begin{align*}
                & \nor{(\Aa [f]) ^{1 - \frac{D - d}{\alpha}}} _{L ^{1, \infty} _{t, x} (\Gamma _T)}\le C \nor{f} _{L ^1 (\domain)} \\
                \implies & \nor{\Aa [f]} ^\lambda _{L ^{\lambda, \infty} _{t, x} (\Gamma _T)}\le C \nor{f} _{L ^1 (\domain)}. 
            \end{align*}
            When $p _1 = q _1 > 1$, by Theorem \ref{thm:avg-spacetime} \eqref{enu:strongtype-spacetime}, we have 
            \begin{align*}
                & \nor{(\Aa [f]) ^{1 - \frac{D - d}{p \alpha}}} _{L ^p _{t, x} (\Gamma _T)} \le C \nor{f} _{L ^p (\domain)} \\
                \implies & \nor{\Aa [f]} ^\lambda _{L ^{p _2} _{t} L ^{q _2} _{x} (\Gamma _T)}\le C \nor{f} _{L ^p (\domain)} .
            \end{align*} 

            \item When $p _2 > \lambda p _1$ and $q _2 = \lambda q _1$, we have $0 < \gamma \le 1$ and $\beta = 1$, thus $\lambda = 1 - \frac{D + 2 \gamma}{p \alpha}$. When $p _1 = q _1 = 1$, by Proposition \ref{prop:anisotropic-spacetime} \eqref{eqn:t>x-spacetime}, we have 
            \begin{align*}
                & \nor{(\Aa [f]) ^{1 - \frac{D - d + 2 \gamma}{\alpha}}} _{L ^{\frac1{1 - \gamma}, \infty} _t L ^{1, \infty} _x (\Gamma _T)} \le C \nor{f} _{L ^1 (\domain)}, \\
                \implies & \nor{\Aa [f]} ^\lambda _{L ^{p _2, \infty} _{t}L ^{q _2, \infty} _{x} (\Gamma _T)}\le C \nor{f} _{L ^1 (\domain)}. 
            \end{align*}
            When $p _1 = q _1 > 1$, by Proposition \ref{prop:anisotropic-spacetime} \eqref{eqn:t>x-spacetime-p}, we have 
            \begin{align*}
                & \nor{(\Aa [f]) ^{1 - \frac{D - d + 2 \gamma}{p\alpha}}} _{L ^{\frac p{1 - \gamma}} _t L ^{p} _x (\Gamma _T)} \le C \nor{f} _{L ^p (\domain)}. \\
                \implies & \nor{\Aa [f]} ^\lambda _{L ^{p _2} _{t} L ^{q _2} _{x} (\Gamma _T)}\le C \nor{f} _{L ^p (\domain)}.
            \end{align*} 

            \item When $p _2 = \lambda p _1$ and $q _2 > \lambda q _1$, we have $\gamma = 0$ and $0 \le \beta < 1$, thus $\lambda = 1 - \frac{D - \beta d}{p \alpha}$. We use Proposition \ref{prop:anisotropic-spacetime} \eqref{eqn:t<x-spacetime}-\eqref{eqn:t<x-spacetime-p}.
            \item When $p _2 > \lambda p _1$ and $q _2 > \lambda q _1$, we have $0 < \gamma \le 1$ and $0 \le \beta < 1$. We use Proposition \ref{prop:anisotropic-spacetime} \eqref{eqn:t?x-spacetime}-\eqref{eqn:t?x-spacetime-p}.
        \end{enumerate}

        \item $\lambda p _1 < p _2 < \infty$ implies $0 < \gamma < \theta$, so we have verified $0 < \gamma < \theta \le 1 \le p < \infty$. If $\lambda q _1 < q _2$, then $\beta \in [0, 1)$, we have 

        \begin{enumerate}[\upshape (a)]
            \item If $p _1 = q _1$, then $\theta = 1$. By \eqref{eqn:t>x-strong-2-spacetime}, we have
            \begin{align*}
                & \nor{(\Aa [f]) ^{1 - \frac{D - \beta d + 2 \gamma}{p \alpha}}} _{L ^{\frac p{1 - \gamma}, \infty} _t L ^{\frac p \beta} _x (\Gamma _T)} \le C \nor{f} _{L ^p _t L ^p _x (\domain)} \\
                \implies & \nor{(\Aa [f]) ^{\lambda}} _{L ^{\frac p{1 - \gamma}, \infty} _t L ^{\frac p \beta} _x (\Gamma _T)} \le C \nor{f} _{L ^{p _1} _t L ^{q _1} _x (\domain)} \\
                \implies & \nor{\Aa [f]} _{L ^{p _2, \infty} _t L ^{q _2} _x (\Gamma _T)} ^{\lambda} \le C \nor{f} _{L ^{p _1} _t L ^{q _1} _x (\domain)}.
            \end{align*}

            \item If $p _1 > q _1$, then $\theta < 1$. By \eqref{eqn:anisotropic-f-3-spacetime}, we have 
            \begin{align*}
                & \nor{(\Aa [f]) ^{1 - \frac{D - \beta d + 2 \gamma}{p \alpha}}} _{L ^{\frac p{\theta - \gamma}} _t L ^{\frac p \beta} _x (\Gamma _T)} \le C \nor{f} _{L ^\frac p\theta _t L ^p _x (\domain)} \\
                \implies & \nor{(\Aa [f]) ^{\lambda}} _{L ^{\frac p{\theta - \gamma}} _t L ^{\frac p \beta} _x (\Gamma _T)} \le C \nor{f} _{L ^{p _1} _t L ^{q _1} _x (\domain)} \\
                \implies & \nor{\Aa [f]} _{L ^{p _2} _t L ^{q _2} _x (\Gamma _T)} ^{\lambda} \le C \nor{f} _{L ^{p _1} _t L ^{q _1} _x (\domain)}.
            \end{align*}

        \end{enumerate}
        If $\lambda q _1 = q _2$, $\beta = 1$, we use \eqref{eqn:t>x-strong-spacetime}-\eqref{eqn:anisotropic-f-spacetime} instead of \eqref{eqn:t>x-strong-2-spacetime}-\eqref{eqn:anisotropic-f-3-spacetime}.
    
        \item When $p _2 = \infty$, we use \eqref{eqn:anisotropic-f-2-spacetime} if $\lambda q _1 = q _2$ and \eqref{eqn:anisotropic-f-4-spacetime} if $\lambda q _1 < q _2$. Now $\gamma = \theta$ is allowed to take value in $[0, 1]$.
    \end{enumerate}

    The requirements of Theorem \ref{thm:avg-spacetime} and Proposition \ref{prop:anisotropic-spacetime} can be summarized to 
    \begin{align*}
        p \alpha > D - d \cdot \ind*{q _2 = \lambda q _1} + 2 \cdot \ind*{p _2 > \lambda p _1},
    \end{align*}
    which is equivalent to $\frac p{r _1} + d \cdot \ind*{q _2 = \lambda q _1} + 2 \cdot \ind*{p _2 = \lambda p _1} > 0$ using 
    \begin{align*}
        \frac p {r _1} = p \alpha - 2 - D
    \end{align*}
    when $p _1 = q _1 = p$.

    The requirement $p \alpha > D - d + 2 \gamma$ of Proposition \ref{prop:anisotropic-2-spacetime} can be simplified as the following: 
    \begin{align*}
        & \alpha - \frac D{q _1} + \frac{d}{q _1} > \frac{2\gamma}{q _1} = \frac2{p _1} - \lambda \frac{2}{p _2} \\
        \iff & \frac1{r _1} + \frac d{q _1} > - \lambda \frac2{p _2} = \lambda \pth{
            \frac1{r _2} - \alpha + \frac d{q _2}
        } \\
        \iff & \frac d{q _1} + \lambda \pth{
            \alpha - \frac d{q _2}
        } > 0
        \\
        \iff & \lambda \alpha + d \pth{\frac1{q _1} - \frac\lambda{q _2}} > 0.
    \end{align*}
    This always hold since $\lambda \alpha > 0$, and $q _2 \ge \lambda q _1$.

    Similarly, the requirement $p \alpha > D + 2 \gamma$ of Proposition \ref{prop:anisotropic-2-spacetime} can be simplified as $\lambda \pth{
        \alpha - \frac d{q _2}
    } > 0$, that is, $\alpha \ge \frac d{q _2}$.
\end{proof}

\begin{remark}
    Recall that $\Aa [f] = (\Sa [f]) ^{-\alpha}$. The above corollary provides a bound on the inverse scale function, for instance \eqref{enu:Bb-spacetime} implies 
    \begin{align*}
        \nor{(\Sa [f]) ^{-1}} _{L ^{p _2 \alpha} _t L ^{q _2 \alpha} _x (\Gamma _T)} ^{\lambda \alpha} \le C \nor{f} _{L ^{p _1} _t L ^{q _1} _x (\domain)}.
    \end{align*}
\end{remark}

\section{Blow-up technique and trace estimate: Lagrangian case}
\label{sec:blowup}

\subsection{Adapting the drift and boundary}

For equations with a transport term, such as transport equations, active scalar equations, kinetic equations, or fluid equations, additional difficulties appear when the pivot quantity $f$ does not control the drift $b$. It might be more convenient to consider the skewed cylinders that ``follow the trajectories'', and interpret the cylinders in the Lagrangian coordinate. The author studied these objects in \cite{yang2022}, but similar construction has been introduced by Isett et al. \cite{isett2016,isett2017,isett2018,isett2022} in the study of the Euler equation. Given a locally integrable drift $b: \R \times \R ^D \to \R ^D$, we fix $\varphi \in \cci (\R ^D)$ to be a nonnegative smooth mollifier supported in $B _1$. For any $\rho > 0$, define $\varphi _\rho (x) = \rho ^{-D}\varphi(x/\rho)$. Let $b _\rho (t) = b (t) * \varphi _\rho$ be the spatially mollified drift, and the flow $X _\rho$ is defined to be the unique solution to the following ODE:
\begin{align}
    \label{eqn:flow}
    \begin{cases}
        \dot X _\rho (s; t, x) = b _\rho (s, X _\rho (s; t, x)) \\
        X _\rho (t; t, x) = x.
    \end{cases}
\end{align}
Suppose $b \in L ^p (\OmegaT)$ for some $p \in [1, \infty]$, then we can make a zero extension so that $b \in L ^p (\R \times \R ^D)$, thus the ODE can have a unique solution for all time. We will assume this from now on. Define for $(t, x) \in \domain$ the skewed parabolic cylinder
\begin{align}
    \label{eqn:Qe}
    Q _\rho (t, x) := \set{(s, X _\rho (s; t, x) + y): s \in (t - \rho ^2, t], y \in B _\rho}.
\end{align}
Although \eqref{eqn:Qe} is duplicate notation with \eqref{eqn:Qe-spacetime}, in this section we will always assume an implied background drift $b$ and {\bf use \eqref{eqn:Qe} to define $Q _\rho (t, x)$ from now on}. In the previous section, the standard parabolic cylinder $(t - \rho ^2, t] \times B _\rho (x)$ can be regarded as a skewed parabolic cylinder $Q _\rho (t, x)$ when the underlying drift is $b = 0$. 

\begin{definition}[\cite{yang2022}, Definition 2]
    We say $Q _\rho (t, x)$ is \textbf{admissible} if 
    \begin{align}
        \label{eqn:admissible}
        \fint _{Q _\rho (t, x)} \mathcal M ((\grad b) (s)) (y) \d y \d s \le \eta _0 \rho ^{-2}.
    \end{align}
    Here $\eta _0 > 0$ is a small universal constant depending only on $d$ and $L$ that will be specified later, and recall $\mm$ is the maximal function defined by 
    \begin{align*}
        (\mm f) (x) = \sup _{r > 0} \fint _{B _r (x)} |f (y)| \d y, \qquad f \in L ^1 _{\loc} (\R ^D), x \in \R ^D.
    \end{align*}
\end{definition}

If there is no background flow (when $b = 0$), then $Q _\rho (t, x) = (t - \rho ^2, t] \times B _\rho (x)$ is just a standard parabolic cylinder, and it is admissible. 
This admissibility ensures a covering lemma for the skewed cylinders.
\begin{lemma}[\cite{yang2022}, Proposition 12]
    There exists a universal constant $C$ depending on the dimension $D$ such that the following holds. Let $\set{Q _{\e _\alpha} (t ^\alpha, x ^\alpha)} _{\alpha}$ be a collection of admissible cylinders with uniformly bounded radii. Then there exists a pairwise disjoint subcollection $\set*{Q _{\e _{\alpha _j}} (t ^{\alpha _j}, x ^{\alpha _j})} _j$ such that 
    \begin{align*}
        \sum _j |Q _{\e _{\alpha _j}} (t ^{\alpha _j}, x ^{\alpha _j})| \ge \frac1C \abs{\bigcup _\alpha Q _{\e _\alpha} (t ^\alpha, x ^\alpha)}.
    \end{align*}
\end{lemma} 

Now for a measurable function $f: \domain \to [0, \infty]$, we define $f _\rho$ for $\rho > 0$, define the scale function $\Sa [f]$ and the averaging function $\Aa [f]$ in the same way as Definition \ref{def:SASingReg-spacetime}, except we now use skewed parabolic cylinders. We do not repeat the definition here. 

In order to extend results from the previous section to the setting with drift, we verify that the inner and outer regularity still hold for the skewed parabolic cylinders.

\begin{lemma}
    \label{lem:inner-outer-regularity-skewed}
    Fix $\rho > 0$ and $(t, x) \in \domain$. 
    \begin{enumerate}
        \item (Inner regularity) For any compact subset $K \ssubset Q _{\rho} (t, x)$, $Q _\sigma (s, y)$ is also a superset of $K$ provided $(\sigma, s, y)$ is sufficiently close to $(\rho, t, x)$. 
        \item (Outer regularity) For any open superset $O \supset \supset \bar Q _{\rho} (t, x)$, $Q _\sigma (s, y)$ is also a subset of $O$ provided $(\sigma, s, y)$ is sufficiently close to $(\rho, t, x)$. 
    \end{enumerate}
\end{lemma}

Since the drift is locally integrable, this is a simple ODE exercise and we omit the proof. This implies the continuity and semicontinuity for the scale function and average function. In other words, Lemma \ref{lem:Aa-basic-spacetime} still hold in our now generalized definition.

Below we give two important examples of the usage of the scale operator, and introduce a capped scale operator $\Saw$. 

\begin{definition}
    Recall $\eta _0 > 0$ defined in \eqref{eqn:admissible} is the admissibility threshold. Define 
    \begin{align*}
        \rint (t, x) &= \mathscr S _{2} [\infty \ind{\OmegaT ^c}] (t, x) = \inf \set{\rho > 0: Q _\rho (t, x) \nsubseteq \OmegaT},
        \\
        \radm (t, x) &= \mathscr S _{2} \bkt{\frac1{\eta _0}\mm (\grad b)} (t, x) = \inf \set{\rho > 0: Q _\rho (t, x) \text{ is not admissible}}.
    \end{align*}
    Denote $\rbar = \radm \wedge \rint$. For $f \in L ^1 _{\loc} (\R ^{1 + d})$, $\alpha > 0$, define $\Sa ^\wedge [f] = \Sa [f] \wedge \rbar$. 
\end{definition}

From the definition, we see that for any $\rho \le \radm (t, x)$, $Q _\rho (t, x)$ is admissible; for any $\rho \le \rint (t, x)$, $Q _\rho (t, x) \subseteq \OmegaT$. Therefore, $\rho \le \Saw [f] (t, x) < \infty$ ensures that the cube $Q _\rho (t, x)$ is an admissible cube, fully contained in $\OmegaT$, with $f _\rho (t, x) \le \rho ^{-\alpha}$. The benefits of using the capped scale operator $\Saw$ are two folds: the cut-off $\rint$ makes sure that the epsilon regularity theorem is applicable in $Q _\rho (t, x)$ in the presence of the boundary, whereas the cut-off $\radm$ enables the covering lemma and harmonic analysis which will be introduced in the next subsection. Moreover, note that $\Saw [f]$ is also upper semicontinuous since it is the minimum of three upper semicontinuous functions.

Provided $f \in L ^1 _\loc$, let us make the following observations. Given $(t, x) \in \domain$, denote $s (t, x) = \Saw [f] (t, x)$, then 
\begin{enumerate}[(a)]
    \item \label{enu:observation1}if $s (t, x) = 0$, it means either 
    \begin{enumerate}
        \item [(a\textsubscript{1})] $\rbar (t, x) = 0$.
        \item [(a\textsubscript{2})] $\rbar (t, x) > 0$ and $\Sa [f] (t, x) = 0$. There exists a sequence of $\e _i \to 0$ such that $f _{\e _i} (t, x) \ge \e _i ^{-\alpha}$, so $f _0 (t, x) = \limsup _{\e \to 0} f _\e (t, x) = +\infty$.
    \end{enumerate}
    
    \item if $s (t, x) > 0$, then $\rbar (t, x) > 0$ and $f _\rho (t, x) \le \rho ^{-\alpha}$ for every $\rho \le s (t, x)$. 
    In particular, 
    $f _{s (t, x)} (t, x) \le s (t, x) ^{-\alpha}$. More precisely:
    \label{enu:observation2}
    \begin{enumerate}
        \item [(b\textsubscript{1})] if $s (t, x) = \rbar (t, x)$, then $f _{s (t, x)} (t, x) = f _{\rbar (t, x)} (t, x) \le \rbar (t, x) ^{-\alpha} = s (t, x) ^{-\alpha}$. 
        \label{enu:observationb1}
        \item [(b\textsubscript{2})] if $s (t, x) < \rbar (t, x)$, then $f _{s (t, x)} (t, x) = s (t, x) ^{-\alpha} > \rbar (t, x) ^{-\alpha}$ by continuity of $f _\rho (t, x)$ in $\rho$.\label{enu:observationb2}
    \end{enumerate}
\end{enumerate}
Note that \eqref{enu:observationb1} also holds when $s (t, x) = \rbar (t, x) = +\infty$.
We divide the singular set and regular set by  
\begin{align*}
    \Sing _\alpha ^= (f) &= \set{(t, x) \in \domain: s (t, x) = 0 = \rbar (t, x)}, 
    \\
    \Sing _\alpha ^< (f) &= \set{(t, x) \in \domain: s (t, x) = 0 < \rbar (t, x)}, 
    \\
    \Reg _\alpha ^= (f) &= \set{(t, x) \in \domain: 0 < s (t, x) = \rbar (t, x)},
    \\
    \Reg _\alpha ^< (f) &= \set{(t, x) \in \domain: 0 < s (t, x) < \rbar (t, x)}.
\end{align*}
They form a partition of $\domain$. Note that $\Sing _\alpha (f) = \Sing _\alpha ^< (f) \cup \Sing _\alpha ^= (f)$, and $\Reg _\alpha (f) = \Reg _\alpha ^< (f) \cup \Reg _\alpha ^= (f)$, so they are consistent with the definition of singular sets and regular sets from the previous section.

To conclude this subsection, we remark that the interior threshold $\rint$ can be replaced by the parabolic distance in $\OmegaT$. This will be accomplished in Lemma \ref{lem:rstar}, but we need some auxiliary results first. 

\begin{lemma}
    \label{lem:c1c2}
    Suppose $c _2 > c _1 > 0$ satisfy
    \begin{align*}
        \nor{\varphi} _{L ^\infty} (c _2 + 2) ^D \eta _0 < \log c _2 - \log c _1.
    \end{align*}
    Let $Q _\e (t, x)$ be admissible, and let $t _* \in (t - \e ^2, t)$. If $x _* \in \R ^D$ satisfies 
    \begin{align*}
        |x _* - X _\e (t _*; t, x)| < c _1 \e,
    \end{align*}
    then 
    \begin{align*}
        |X _\e (s; t _*, x _*) - X _\e (s; t, x)| < c _2 \e, \qquad \text{ for all } s \in (t - \e ^2, t].
    \end{align*}
    Recall that $X _\e$ is defined in \eqref{eqn:flow}.
\end{lemma}

\begin{proof}
    The proof is the same as Lemma 5 of \cite{yang2022} which is a special case of $c _1 = 1$ and $c _2 = 2$. We prove it here for the sake of completeness.

    For simplicity, denote the two trajectories by $\gamma (s) = X _\e (s; t, x)$ and $\gamma _* (s) = X _\e (s; t _*, x _*)$. Then $|\gamma _* (t _*) - \gamma (t _*)| < c _1 \e$. 
    We have 
    \begin{align*}
        \abs{\dfr{}s \log |\gamma (s) - \gamma _* (s)|} 
        &= \frac1{|\gamma (s) - \gamma _* (s)|} \abs{\dfr{}s |\gamma (s) - \gamma _* (s)|}
        \\
        &\le \frac1{|\gamma (s) - \gamma _* (s)|} \abs{\dot \gamma (s) - \dot \gamma _* (s)}
        \\
        &= \frac1{|\gamma (s) - \gamma _* (s)|} \abs{b _\e (s, \gamma (s)) - b _\e (s, \gamma _* (s))} 
        \\
        &\le |\grad b _\e (s, \xi)|
    \end{align*}
    for some $\xi$ between $\gamma (s)$ and $\gamma _* (s)$. Then provided $|\gamma (s) - \gamma _* (s)| \le c _2 \e$ at time $s$, we would have $|\gamma (s) - \xi| \le c _2 \e$ and
    \begin{align*}
        |\grad b _\e (s, \xi)| \le \int _{B _\e (\xi)} |\grad b (x)| |\varphi (x - \xi)| \d x \le \nor{\varphi} _{L ^\infty} \int _{B _{(c _2 + 2) \e} (y)} |\grad b (x)| \d x 
    \end{align*}
    for any $y \in B _\e (\gamma (s))$. Hence 
    \begin{align*}
        \abs{\dfr{}s \log |\gamma (s) - \gamma _* (s)| }
        \le \nor{\varphi} _{L ^\infty} (c _2 + 2) ^D \int _{B _{\e} (\gamma (s))} \mm (\grad b) (x) \d x .
    \end{align*}

    We now prove the lemma using a continuity argument. Suppose there exists $t _0 \in (t - \e ^2, t _*)$ such that $|\gamma (s) - \gamma _* (s)| \le c _2 \e$ for all $t _0 < s < t _*$, and equality is reached at $s = t _0$. Then 
    \begin{align*}
        \log c _2 - \log c _1 &= \log (c _2 \e) - \log (c _1 \e) \\
        &\le \log |\gamma (t _0) - \gamma _* (t _0)| - \log |\gamma (t _*) - \gamma _* (t _*)| \\
        &\le \nor{\varphi} _{L ^\infty} (c _2 + 2) ^D \int _{t _0} ^{t _*} \int _{B _{\e} (\gamma (s))} \mm (\grad b) (x) \d x \\
        &\le \nor{\varphi} _{L ^\infty} (c _2 + 2) ^D \eta _0.
    \end{align*}
    Hence we reach a contradiction. The same argument holds for $t _0 \in (t _*, t)$. By continuity, $|\gamma (s) - \gamma _* (s)|$ cannot reach $c _2 \e$ at any $s \in (t - \e ^2, t]$.
\end{proof}

\begin{corollary}
    \label{cor:dist-from-boundary}
    Suppose $\Omega$ satisfies Assumption \ref{ass:lipschitz} with constant $L$ and $r _0$. Let $(t, x) \in \OmegaT$, and suppose $Q _\e (t, x)$ is admissible with $(L + 4) \e < r _0$. If 
    \begin{align*}
        \dist _\P (Q _\e (t, x), \partial _\P \OmegaT) < \e,    
    \end{align*}
    then 
    \begin{align*}
        \dist _\P ((t, x), \partial _\P \OmegaT) \le (L + 4) \e.
    \end{align*}
\end{corollary}

\begin{proof}
    We argue by contradiction and assume for some $(s, y) \in Q _\e (t, x)$, 
    \begin{align*}
        \dist _\P ((s, y), \partial _\P \OmegaT) &< \e, \\
        \dist _\P ((t, x), \partial _\P \OmegaT) &> (L + 4) \e.
    \end{align*}
    Since $(t, x) \in \OmegaT$, it is clear that 
    \begin{align*}
        \dist _\P ((t, x), \partial _\P \OmegaT) = \mins{\sqrt t, \dist (x, \partial \Omega)},
    \end{align*}
    so $\sqrt t > (L + 4) \e$ and $\dist (x, \partial \Omega) > (L + 4) \e$. Since $(s, y) \in Q _\e (t, x)$, we know that
    \begin{align*}
        0 < 9 \e ^2 < ((L + 4) ^2 - 1) \e ^2 < t - \e ^2 < s \le t < T.
    \end{align*}
    As $s > 9 \e ^2$, $\dist _\P ((s, y), \set{0} \times \Omega) > 3 \e$, so it must hold that 
    \begin{align*}
        \e > \dist _\P ((s, y), \partial _\P \OmegaT) = \dist _\P ((s, y), [0, T) \times \partial \Omega) = \dist (y, \partial \Omega).
    \end{align*}
    Moreover, since $|y - X _\e (s; t, x)| < \e$, triangular inequality implies
    \begin{align*}
        \dist (X _\e (s; t, x), \Omega ^c) < 2 \e.
    \end{align*}
    Together with $\dist (X _\e (t; t, x), \Omega ^c) = \dist (x, \Omega ^c) > (L + 4) \e > 2 \e$, by continuity, we know there exists $t _* \in (s, t)$ such that 
    \begin{align*}
        \dist (X _\e (t _*; t, x), \Omega ^c) = 2 \e.
    \end{align*}
    Therefore, we can find $y _* \in \partial \Omega$ such that $|y _* - X _\e (t _*; t, x)| = 2 \e$. Since $\partial \Omega \cap B _{r _0} (y _*)$ is an $L$-Lipschitz graph, we can find a ball 
    \begin{align*}
        B _\e (x _*) \subset B _{r _0} (y _*) \setminus \Omega
    \end{align*}
    with $|x _* - y _*| = (L + 1) \e$, and 
    \begin{align*}
        |x _* - X _\e (t _*; t, x)| \le |x _* - y _*| + |y _* - X _\e (t _*; t, x)| \le (L + 3) \e.
    \end{align*}
    By Lemma \ref{lem:c1c2}, provided $\eta _0$ is sufficiently small depending on $L$, we have
    \begin{align*}
        |X _\e (t; t _*, x _*) - X _\e (t; t, x)| \le (L + 4) \e .
    \end{align*}
    Since $B _\e (x _*)$ is disjoint from $\Omega$, the drift is zero, so $X _\e (t; t _*, x _*) = x _*$. Thus $|x _* - x| \le (L + 4) \e$. Because $x ^* \notin \Omega$, we know that $\dist (x, \partial \Omega) \le (L + 4) \e$, which is a contradiction.
\end{proof}

\begin{lemma}
    \label{lem:rstar}
    For any $(t, x) \in \OmegaT$, denote 
    \begin{align*}
        \rstar (t, x) = \frac{\dist _\P ((t, x), \partial _\P \OmegaT)}{L + 4} \wedge r _0.
    \end{align*}
    Then it holds 
    \begin{align*}
        \rint (t, x) \wedge \radm (t, x) \ge \rstar (t, x) \wedge \radm (t, x).
    \end{align*}
\end{lemma}

\begin{proof}
    If $\radm (t, x) \le \rint (t, x)$, then there is nothing to prove. If $\radm (t, x) > \rint (t, x)$, then $Q _{\rint (t, x)} (t, x)$ is an admissible cylinder that touches the parabolic boundary $\partial _\P \OmegaT$. Corollary \ref{cor:dist-from-boundary} then implies $\dist _\P ((t, x), \partial _\P \OmegaT) \le (L + 4) \rint (t, x)$.
\end{proof}

\subsection{The cutoff averaging operator}

We define the cutoff averaging operator $\Aa ^\wedge [f]$ by the following. 

\begin{definition}
    For $f \in L ^1 _{\mathrm{loc}} (\domain)$, $(t, x) \in \domain$, define
    \begin{align*}
        \Aa ^\wedge [f] (t, x) &:= f _{\Saw [f] (t, x)} (t, x), \\
        \Aa ^< [f] (t, x) &:= \Aa ^\wedge [f] (t, x) \ind{\Sing _\alpha ^< (f) \cup \Reg _\alpha ^< (f)} (t, x) = \Sa [f] (t, x) ^{-\alpha} \ind*{\Sa [f] (t, x) < \rbar (t, x)}, \\
        \Aa ^= [f] (t, x) &:= \Aa ^\wedge [f] (t, x) \ind{\Sing _\alpha ^= (f) \cup \Reg _\alpha ^= (f)} (t, x) = f _{\rbar (t, x)} (t, x) \ind*{\Sa [f] (t, x) = \rbar (t, x)}.
    \end{align*}
\end{definition}

\begin{remark}
    $\Aa ^< [f]$ and $\Aa ^= [f]$ are disjointly supported, and $\Aa ^\wedge [f] = \Aa ^< [f] + \Aa ^= [f]$.
    By observation \eqref{enu:observation2}, we note that
    \begin{itemize}
        \item $\Aa ^< [f] (t, x) \le \Aa ^\wedge [f] (t, x) \le \Sa [f] (t, x) ^{-\alpha}$ for every $(t, x)$.
        
        \item $\Aa ^= [f] (t, x) = \Aa ^\wedge [f] (t, x) \le \rbar (t, x) ^{-\alpha}$ if $(t, x) \in \Reg _\alpha ^= (f)$.
        
        \item $\Aa ^< [f] (t, x) = \Aa ^\wedge [f] (t, x) > \rbar (t, x) ^{-\alpha}$ if $(t, x) \in \Reg _\alpha ^< (f)$.
    \end{itemize}
    Moreover, the partition of $\R^{1 + d}$ can now be written as 
    \begin{align*}
        \Sing _\alpha ^< (f) &= \set{\Aa ^< [f] =  +\infty}, \\ 
        \Sing _\alpha ^= (f) &= \set{\Aa ^= [f] =  +\infty}, \\ 
        \Reg _\alpha ^< (f) &= \set{0 < \Aa ^< [f] <  +\infty}, \\
        \Reg _\alpha ^= (f) &= \set{0 < \Aa ^= [f] <  +\infty} \cup \set{\Aa ^< [f] = \Aa ^= [f] = 0}.
    \end{align*}
\end{remark}

Since $\Saw [f]$ is upper semicontinuous, we know that $\Aa ^\wedge [f]$, $\Aa ^< [f]$ and $\Aa ^= [f]$ are Borel measurable. Note that this average is pointwise bounded from above by the $\mathcal Q$-maximal function:
\begin{align*}
    \mathcal M _{\mathcal Q} f (t, x):= \set{
        \sup _{r > 0} \fint _{Q _r (t, x)} f \d y \d s: Q _r (t, x) \text{ is admissible}
    },
\end{align*}
the averaging operator $\Aa ^\wedge$ inherits the boundedness of the maximal function $\mathcal M _{\mathcal Q}$: weak type-$(1, 1)$ and strong type-$(p,p)$ for $1 < p \le \infty$ \cite[theorem 1]{yang2022}. 

\begin{proposition}
    \label{thm:maximal-weak-strong}
    Let $f \in L ^1 _{\mathrm{loc}} (\domain)$. Suppose $b \in L ^p (0, T; W ^{1, p} _0 (\Omega))$ is divergence free, and $\mm (\grad b) \in L ^p (\OmegaT)$ for some $1 \le p \le \infty$. Then 
    \begin{enumerate}[\upshape (1)]
        \item $\meas (\set{\Aa ^\wedge [f] = \infty}) = 0$.
        \item If $f \in L ^1 (\domain)$, then $\Aa ^\wedge [f] \in L ^{1, \infty} (\domain)$ with estimate 
        \begin{align*}
            \nor{\Aa ^\wedge [f]} _{L ^{1, \infty} (\domain)} \le C \nor{f} _{L ^1 (\domain)}.
        \end{align*}
        \item If $f \in L ^q (\domain)$ for some $q \in (1, \infty]$, then $\Aa ^\wedge [f] \in L ^q (\domain)$ with 
        \begin{align*}
            \nor{\Aa ^\wedge [f]} _{L ^q (\domain)}\le C (p) \nor{f} _{L ^q (\domain)} .
        \end{align*}
    \end{enumerate}
\end{proposition}

Next, we explore the finer structure of the averaging operator $\Aa$. 
We note that $\Aa ^\wedge$ is also quasiconvex in the following sense. 

\begin{lemma}\label{lem:sublin}
    Fix $\alpha > 0$.
    For any nonnegative functions $f, g \in  L ^1 _{\loc} (\domain)$, $\lambda \in [0, 1]$, it holds that
    \begin{align*}
        \Aa \bkt{(1 - \lambda) f + \lambda g} (t, x) \le \maxs{
            \Aa [f] (t, x), \Aa [g] (t, x)
        }, \qquad \forall (t, x) \in \domain.
    \end{align*}
\end{lemma}

\begin{proof}
    Denote $h = (1 - \lambda) f + \lambda g$, and take $(t, x) \in \domain$. We separate the following cases.
    \begin{enumerate}[(a)]
        \item If $(t, x) \in \Sing _\alpha (f) \cup \Sing _\alpha (g)$, then the right hand side is $+\infty$ and there is nothing to prove. Otherwise, $(t, x) \in \Reg _\alpha (f) \cap \Reg _\alpha (g)$, $\rbar (t, x) > 0$, so $(t, x) \notin \Sing _\alpha ^= (h)$. 

        \item If $(t, x) \in \Sing _\alpha ^< (h) \cup \Reg _\alpha ^< (h)$, then there exists a sequence of $\rho _n > 0$ within $\Sa [h] (t, x) < \rho _n < \rbar (t, x)$ and $\rho _n \to \Sa [h] (t, x)$ such that 
        \begin{align*}
            \rho _n ^{-\alpha} < h _{\rho _n} (t, x) = (1 - \lambda) f _{\rho _n} (t, x) + \lambda g _{\rho _n} (t, x).
        \end{align*}
        Therefore, at least one of $f _{\rho _n} (t, x)$ and $g _{\rho _n} (t, x)$ is greater than $\rho _n ^{-\alpha}$. Suppose, up to a subsequence, that $f _{\rho _n} (t, x) > \rho _n ^{-\alpha}$. Then $(t, x) \in \Reg _\alpha ^< (f)$, and $0 < \Sa [f] (t, x) \le \rho _n$. Taking $n \to \infty$ we find $\Sa [f] (t, x) \le \Sa [h] (t, x)$, and consequently 
        \begin{align*}
            \Aa ^\wedge [h] (t, x) = \Sa [h] (t, x) ^{-\alpha} \le \Sa [f] (t, x) ^{-\alpha} = \Aa ^\wedge [f] (t, x).            
        \end{align*}

        \item If $(t, x) \in \Reg _\alpha ^= (h) \cap \Reg _\alpha ^< (f)$, then 
        \begin{align*}
            \Aa ^\wedge [h] (t, x) \le \rbar (t, x) ^{-\alpha} < \Aa ^\wedge [f] (t, x).
        \end{align*}
        Similarly, if $(t, x) \in \Reg _\alpha ^= (h) \cap \Reg _\alpha ^< (g)$, then $\Aa ^\wedge [h] (t, x) < \Aa ^\wedge [g] (t, x)$.

        \item If $(t, x) \in \Reg _\alpha ^= (h) \cap \Reg _\alpha ^= (f) \cap \Reg _\alpha ^= (g)$, then 
        \begin{align*}
            \Aa ^\wedge [h] (t, x) &= \fint _{Q _{\rbar (t, x)} (t, x)} h \d y \d s
            \\
            &= (1 - \lambda) \fint _{Q _{\rbar (t, x)} (t, x)} f \d y \d s + \lambda \fint _{Q _{\rbar (t, x)} (t, x)} g \d y \d s 
            \\
            &= (1 - \lambda) \Aa ^\wedge [f] (t, x) + \lambda \Aa ^\wedge [g] (t, x).
        \end{align*}
    \end{enumerate}
    These are all the possible cases, so it always holds that $\Aa ^\wedge [h] \le \max \set{\Aa ^\wedge [f], \Aa ^\wedge [g]}$.
\end{proof}

Finally, we remark that $\rbar \ge \rstar$ in $\Reg _\alpha ^= (f)$ if $f$ is dominates the drift gradient. 

\begin{lemma}
    \label{lem:radmrbar}
    Suppose $\alpha \ge 2$, $\mm (\grad b) \in L ^1 _\loc (\OmegaT)$, and $f \ge [\frac1{\eta _0} \mathcal M (\grad b)] ^\frac\alpha2$ in $\OmegaT$. Then $\rbar (t, x) \ge \rstar (t, x)$ for any $(t, x) \in \Reg _\alpha ^= (f)$. Moreover, if $f$ is bounded in a neighborhood of $(t, x)$, then $(t, x) \in \Reg _\alpha (f)$. 
\end{lemma}

\begin{proof}
    First, we use Jensen Lemma \ref{lem:nonlinear-scaling}:
    \begin{align*}
        \radm = \mathscr S _2 \left[
            \frac{\mm (\grad b)}{\eta _0}
        \right] \ge \mathscr S _{\frac\alpha2 \cdot 2} \left[
            \left(
                \frac1{\eta _0} \mathcal M (\grad b)
            \right) ^\frac\alpha2
        \right] \ge \Sa [f].
    \end{align*}
    By Lemma \ref{lem:rstar}, we know that 
    \begin{align*}
        \rbar (t, x) \ge \radm (t, x) \wedge \rstar (t, x).
    \end{align*}
    Moreover, for any $(t, x) \in \Sing _\alpha ^= (f) \cup \Reg _\alpha ^= (f)$ it holds that 
    \begin{align*}
        \Sa [f] (t, x) \ge \Saw [f] (t, x) = \rbar (t, x).
    \end{align*}
    Combining the above three inequalities, we find that $\rbar (t, x) \ge \rstar (t, x)$ for any $(t, x) \in \Reg _\alpha ^= (f)$. 

    Next, suppose $f$ is bounded in a neighborhood $U \subset \OmegaT$ of $(t, x)$. Then $\mm (\grad b)$ is also bounded in $U$. For $\e$ sufficiently small, $Q _\e (t, x) \subset U$ and $\fint _{Q _\e (t, x)} \mm (\grad b) \le \eta _0 \e ^{-2}$ imply $Q _\e (t, x)$ is admissible and $\rbar (t, x) > 0$. This shows $(t, x) \notin \Sing _\alpha ^= (f)$. Moreover, $\fint _{Q _\e (t, x)} f \le \e ^{-\alpha}$ for sufficiently small $\e$, so $(t, x) \notin \Sing _\alpha ^< (f)$. Combined, they prove $(t, x) \in \Reg _\alpha (f)$. 
\end{proof}

In the next subsection, we will show some trace estimates on $\Aa ^<$, which extends results from Section \ref{sec:spacetime}.

\subsection{Trace estimate for the averaging operator}

To establish trace estimates, we need to measure the level sets of $\Sa [f]$. We start with the section of a regular set. 

\begin{lemma}
    \label{lem:atp-drift}

    Let $t \in \R$, $\alpha > 0$, $p \in [1, \infty)$, and $f \in L ^p _{\loc, t} L ^p _x (\domain)$. Denote 
    \begin{align*}
        A _t (\rho) &:= \set{
            x' \in \Gamma _t:
            \rho \le \Sa [f] (t, x') < 2 \rho \wedge \radm (t, x')
        }, \qquad \rho > 0.
    \end{align*}
    Then for every $t \in (0, T)$, $\rho > 0$, 
    \begin{align*}
        \mathscr H ^d \pth{A _t (\rho)} \le C \rho ^{-D + d - 2 + p \alpha} \int _{t - 4 \rho ^2} ^{t} \intRd f (s, x) ^p \dx \d s.
    \end{align*}
\end{lemma}

\begin{proof}
    We can borrow almost entirely the proof of Lemma \ref{lem:atp-spacetime}, with one exception that pairwise disjoint $B _{\rho _i} (x' _i)$ do not guarantee disjoint $Q _{\rho _i} (t, x' _i)$. We will patch this up by the following alternative argument instead. For simplicity, suppose $p = 1$ and $f \in L ^1 _{t, x} (\domain)$.

    Same as Lemma \ref{lem:atp-spacetime}, we know 
    \begin{align*}
        \mathscr H ^{d} (A _t (\rho)) \le (1 + \nor{g _t} _{\Lip} ^2) ^{\frac{d}2}  \frac{\mathscr L ^D (\mathcal U _\rho (A _t (\rho)))}{c _{D - d} \rho ^{D - d}},
    \end{align*}
    with $\mathcal U _\rho (A _t (\rho))$ covered by $\bigcup _{x' \in A _t (\rho)} B _{\rho _{x'}} (x')$ for some $\rho ' \in [\rho, 2 \rho \wedge \radm (t, x'))$ satisfying
    \begin{align*}
        f _{\rho _{x'}} (t, x') > \rho _{x'} ^{-\alpha}.
    \end{align*}
    Now we choose a bigger covering: 
    \begin{align*}
        \mathcal U _\rho (A _t (\rho)) \subset \bigcup _{x' \in A _t (\rho)} B _{9 \rho _{x'}} (x').
    \end{align*}
    By Vitali's covering lemma, we can find a disjoint subcollection $\set{B _{9 \rho _i} (x _i)} _i$ with 
    \begin{align*}
        \mathscr L ^{D} (\mathcal U _\rho (A _t (\rho))) \le C \sum _{i} \mathscr L ^{D} (B _{9 \rho _i} (x _i)) \le C \rho ^{-2} \sum _i \mathscr L ^{D + 1} (Q _{\rho _i} (t, x _i)).
    \end{align*}
    Here $C = C (d)$ depends only on the dimension.
    Since $\rho _i \le \radm (x _i)$, $Q _{\rho _i} (t, x _i)$ is an admissible cylinder. 
    By \cite[Proposition 7]{yang2022}, we know $B _{9 \rho _i} (x _i) \cap B _{9 \rho _j} (x _j) = \varnothing$ implies $Q _{\rho _i} (t, x_ i)$ and $Q _{\rho _j} (t, x _j)$ are disjoint. Hence $\set{Q _{\rho _i} (t, x _i)} _i$ are also pairwise disjoint. Therefore, the total volume is
    \begin{align*}
        \sum _i \mathscr L ^{D + 1} (Q _{\rho _i} (t, x _i)) &= \sum _i \frac1{f _{\rho _i} (t, x _i)} \int _{Q _{\rho _i} (t, x _i)} f (s, y) \d y \d s \\
        &\le \sum _i \rho _i ^\alpha \int _{Q _{\rho _i} (t, x _i)} f (s, y) \d y \d s \\
        &\le \sum _i (2\rho) ^\alpha \int _{Q _{\rho _i} (t, x _i)} f (s, y) \d y \d s \\
        & \le (2\rho) ^\alpha \int _{(t - (2 \rho) ^2, t) \times \R ^D} f (s, y) \d y \d s.
    \end{align*}
    This is because $Q _{\rho _i} (t, x _i) \subset (t - (2 \rho) ^2, t) \times \R ^D$. Combine these estimates, we have 
    \begin{align*}
        \mu _t (A _t (\rho)) \le C \rho ^{-D + d - 2 + \alpha} \int _{(t - (2 \rho) ^2, t) \times \Rd} f (s, y) \d y \d s.
    \end{align*}
    This finishes the proof of the lemma.
\end{proof}

Next, we measure the singular set. We obtain the same estimate on the dimension of the singular set as in Proposition \ref{lem:at0-spacetime}, and anisotropic estimates as in Corollary \ref{cor:p1q1p2q2}.

\begin{proposition}
    \label{lem:at0-drift}
    Let $\alpha > 0$, and $f \in L ^p _{\loc, t} L ^q _{\loc, x} (\domain)$, $1 \le q \le p < \infty$. Denote the singular section by 
    \begin{align*}
        S _t := \set{
            x' \in \Gamma _t: (t, x') \in \Sing _\alpha (f), \radm (t, x') > 0
        }, \qquad t \in \R,
    \end{align*}
    and denote the set of singular time by 
    \begin{align*}
        \mathcal T = \set{t \in \R: \mathscr H ^d (S _t) > 0}.
    \end{align*}
    \begin{enumerate}[\upshape (a)]
        \item \label{enu:at0} If $\frac{D - d}q < \alpha \le \frac 2p + \frac{D - d}q$, then $\mathcal T$ has Hausdorff dimension no greater than
        \begin{align*}
            \dim _{\mathscr H} (\mathcal T) \le 1 - \frac p2 \pth{\alpha - \frac{D - d}{q}}.
        \end{align*}

        \item \label{enu:at0-reg} If $\alpha > \frac 2p + \frac{D - d}q$, then $\mathcal T = \varnothing$.
    \end{enumerate}
    
\end{proposition}

\begin{proof}
    If $\radm (t, x') > 0$, then for every $\e > 0$, there exists $\rho = \rho _{x'} < \e$ such that $f _\rho (t, x') > \rho ^{-\alpha}$ and $Q _\rho (t, x')$ is admissible, so the covering lemma applies. The rest of the proof is identical to that of Proposition \ref{lem:at0-spacetime}.
\end{proof}

\begin{theorem}
    \label{thm:avg}
    \renewcommand{\Aa}{\mathscr A _\alpha ^<}
    \renewcommand{\domain}{\OmegaT}
    Let $0 < p _1, q _1, p _2, q _2 \le \infty$ and $\alpha > 0$. Define $r _1, r _2 \in \R \setminus \{0\} \cup \{\infty\}$ by 
    \begin{align*}
        \frac1{r _1} = \alpha - \frac{2}{p _1} - \frac{D}{q _1}, \qquad \frac1{r _2} = \alpha - \frac{2}{p _2} - \frac{d}{q _2}.
    \end{align*}
    Suppose $r _2 = \lambda r _1$ for some $\lambda \in (0, \infty)$, $p _2 \ge \lambda p _1$, $q _2 \ge \lambda q _1$, $1 \le q _1 \le p _1$, and $f \in L ^{p _1} _t L ^{q _1} _x (\domain)$. 
    
    \begin{enumerate}[\upshape (A)]
        \item \label{enu:weakweak-drift} Suppose $p _1 = q _1 = p \in [1, \infty]$, and assume $\frac 1{r _1} + \frac2{p _1} \cdot \ind*{p _2 = \lambda p _1} + \frac d{q _1} \cdot \ind*{q _2 = \lambda q _1} > 0$.
        
        \begin{enumerate}[\upshape (a)]
            \item \label{enu:weakweak1-drift} If $p = 1$, then 
            \begin{align*}
                \nor{\Aa [f]} _{L ^{\lambda, \infty} _{t, x} (\Gamma _T)} ^{\lambda} &\le C \nor{f} _{L ^1 _{t, x} (\domain)} \qquad \text{ if } p _2 = q _2 = \lambda, \\
                \nor{\Aa [f]} _{L ^{p _2, \infty} _t L ^{q _2, \infty} _x (\Gamma _T)} ^{\lambda} &\le C \nor{f} _{L ^1 _{t, x} (\domain)} \qquad \text{ if } p _2 > \lambda \text{ or }  q _2 > \lambda.
            \end{align*}
    
            \item \label{enu:p=q>1-drift} If $p > 1$, then 
            \begin{align*}
                \nor{\Aa [f]} _{L ^{p _2} _t L ^{q _2} _x (\Gamma _T)} ^{\lambda} \le C \nor{f} _{L ^{p} _{t, x} (\domain)}.
            \end{align*}
        \end{enumerate}

        \item \label{enu:weakstrong-drift} Suppose $\lambda p _1 < p _2 < \infty$. When $\lambda q _1 < q _2$, we further require $\alpha > \frac{d}{q _2}$.
        
        \begin{enumerate}[\upshape (a)]
            \item \label{enu:Ba} If $p _1 = q _1$, then 
            \begin{align*}
                \nor{\Aa [f]} _{L ^{p _2, \infty} _t L ^{q _2} _x (\Gamma _T)} ^{\lambda} \le C \nor{f} _{L ^{p _1} _t L ^{q _1} _x (\domain)}.
            \end{align*}

            \item \label{enu:Bb} If $p _1 > q _1$, then 
            \begin{align*}
                \nor{\Aa [f]} _{L ^{p _2} _t L ^{q _2} _x (\Gamma _T)} ^{\lambda} \le C \nor{f} _{L ^{p _1} _t L ^{q _1} _x (\domain)}.
            \end{align*}
        \end{enumerate}

        \item \label{enu:C-drift}
        Suppose $p _2 = \infty$. When $\lambda q _1 < q _2$, we further require $\alpha > \frac{d}{q _2}$. Then 
        \begin{align*}
            \nor{\Aa [f]} _{L ^{\infty} _t L ^{q _2, \infty} _x (\Gamma _T)} ^{\lambda} \le C \nor{f} _{L ^{p _1} _t L ^{q _1} _x (\domain)}.
        \end{align*}
    \end{enumerate}
    In all of the above, the constant $C = C (p _1, q _1, p _2, q _2, \alpha, D, d, L)$.
\end{theorem}

\begin{proof}
    Recall that $f$ is extended to be zero outside $\OmegaT$.
    The measurability of $\Aa [f]$ comes from semicontinuity. Since $\bar r$ is also semicontinuous, sets $\Sing _\alpha ^< (f)$ and $\Reg _\alpha ^< (f)$ are also measurable. Same as before, in both cases $q _1 \alpha > D - d$. Lemma \ref{lem:at0-drift} implies $\mu _t (S _t) = 0$ for $\mathscr L ^1$-almost every $t \in [0, 1]$. Therefore 
    \begin{align*}
        \mu _T (\set{\Aa ^< [f] = \infty}) \le  \int _0 ^T \mu _t (S _t) \d t = 0.
    \end{align*}
    Notice that for $\rho > 0$, $(t, x) \in \Gamma _T$, 
    \begin{align*}
        (2 \rho) ^{-\alpha} < \Aa ^< [f] (t, x) \le \rho ^{-\alpha} \implies 
        \rho \le \Sa [f] (t, x) < 2 \rho \wedge \radm (t, x) &\implies x \in A _t (\rho),
        \\
        \Aa ^< [f] (t, x) = \infty \implies \Sa [f] (t, x) = 0 < \bar r (t, x) \le \radm (t, x) &\implies x \in S _t.
    \end{align*}
    where $A _t (\rho)$ is defined in Lemma \ref{lem:atp-drift}, and $S _t$ is defined in Lemma \ref{lem:at0-drift}. All the estimates are based on these two lemmas, same as in Corollary \ref{cor:p1q1p2q2} from the previous section, so we omit the proofs. 
\end{proof}

\section{Trace of vorticity}
\label{sec:proof}

In this section, we apply the blow-up method on vorticity and higher derivative estimates. First, we recall some local estimates for the Navier--Stokes equation. Then we apply the averaging operator to prove the main results.

\subsection{$\e$-regularity theory for the Navier--Stokes equation}

\begin{lemma}[Caffarelli, Kohn, and Nirenberg, \cite{caffarelli1982}; Lin \cite{lin1998}]
    There exists $\epsilon _0 > 0$ such that if $u, P$ is a suitable weak solution to \eqref{eqn:navier-stokes} in $(-4, 0) \times B _2$, satisfying 
    \begin{align*}
        \int _{-4} ^0 \int _{B _2} |u| ^3 + |P| ^\frac32 \dx \dt \le \epsilon _0, 
    \end{align*}
    then for $n \ge 0$, $\nor{\grad ^n u} _{L ^\infty ((-1, 0) \times B _1)} \le C _n$.
\end{lemma}

By scaling, we obtain the following corollary.

\begin{corollary}
    Let $u$ be a suitable weak solution to \eqref{eqn:navier-stokes} in $\OmegaT$. For $\e > 0$ and $(t, x) \in \OmegaT$, define $Q _\e (t, x)$ as in \eqref{eqn:Qe-spacetime}. If $Q _\e (t, x) \subset \OmegaT$ and 
    \begin{align*}
        \fint _{Q _\e (t, x)} |u| ^3 + |P| ^\frac32 \dx \dt \le \epsilon _0 \e ^{-3}, 
    \end{align*}
    then $|\grad ^n u (t, x)| \le C _n \e ^{-n - 1}$.
\end{corollary}

The downside of this estimate is that the scaling of $u$ is worse than the scaling of $\grad u$. If we want to have the same scaling as $\grad u \in L ^2 _{t, x}$, $u$ needs to have $L ^4 _{t, x}$ integrability. This is not known for suitable weak solutions. Therefore, we use the following lemma instead.

\begin{lemma}[Choi and Vasseur, \cite{choi2014}]
    There exists $\bar \eta$ and a sequence of constants $C _n$ with the following property. If $u, P$ is a classical solution to \eqref{eqn:navier-stokes} in $(-4, 0) \times B _2$,satisfying 
    \begin{align*}
        \int _{B _2} \varphi (x) u (t, x) \d x = 0, \qquad \text{ for all } t \in (-2, 0), 
    \end{align*}
    and 
    \begin{align*}
        \int _{-4} ^0 \int _{B _2} |\mathcal M (\grad u)| ^2 + |\grad ^2 P| \dx \dt \le \bar \eta, 
    \end{align*}
    then for $n \ge 1$, $\nor{\grad ^n u} _{L ^\infty ((-\frac19, 0) \times B _\frac13)} \le C _n$. 
\end{lemma}

\begin{proof}
    This is the first part of Proposition 2.2 of \cite{choi2014} with $r = 0$ (see also Remark 2.7). Although the theorem stated in (18) that the domain of the solution is $\Omega = \RR3$, the integer-order derivatives part of the proof is based on De Giorgi iteration and hence is purely local. Therefore, it suffices to require $u$ to solve the Navier--Stokes equation in $(-2, 0) \times B _2$. 
\end{proof}

By scaling, we have the following corollary at the $\e$ scale. 
\begin{corollary}
    Let $u$ be a classical solution to \eqref{eqn:navier-stokes} in $\OmegaT$. For $\e > 0$ and $(t, x) \in \OmegaT$, define $Q _\e (t, x) \subset \OmegaT$ as in \eqref{eqn:flow}-\eqref{eqn:Qe} with $b = u$. If $Q _\e (t, x) \subset \OmegaT$ and 
    \begin{align*}
        \fint _{Q _\e (t, x)} |\mm (\grad u)| ^2 + |\grad ^2 P| \dx \dt \le \bar \eta \e ^{-4}, 
    \end{align*}
    then $|\grad ^n u (t, x)| \le C _n \e ^{-n - 1}$.
\end{corollary}

\begin{proof}
    The proof of the corollary uses the Galilean invariance of the Navier--Stokes equation. See \cite[Corollary 18]{yang2022} for instance. Given a solution $(u, P)$ to \eqref{eqn:navier-stokes} in $Q _\e (t, x)$ for some $(t, x) \in \OmegaT$, for $s \in (-1, 0]$ and $y \in B _1 (0) \subset \RR3$ define
    \begin{align*}
        r (s) &= t + \e ^2 s, \\
        z (s) &= X _\e (r (s); t, x), \\
        v (s, y) &= \e u (r (s), z (s) + \e y) - \e u _\e (r (s), z (s)), \\ 
        p (s, y) &= \e ^2 P (r (s), z (s) + \e y) + \e y \partial _s [u _\e (r (s), z (s))].
    \end{align*}
    Recall that $X _\e$ is defined by \eqref{eqn:flow}. $(v, p)$ then solves \eqref{eqn:navier-stokes} in $(-1, 0) \times B _1$, and 
    \begin{align*}
        \fint _{(-1, 0) \times B _1} \abs{\mm (\grad v)} ^2 + |\grad ^2 p| \d x \d t = \e ^4 \fint _{Q _\e (t, x)} |\mm (\grad u)| ^2 + |\grad ^2 P| \dx \dt.
    \end{align*}
    Also for any $n \ge 1$, $\grad ^n u (t, x) = \e ^{-n - 1} \grad ^n v (0, 0)$.
\end{proof}

Here we need to use the second derivative of the pressure. If $\Omega = \RR3$ or $\mathbb T ^3$, then $\nor{\grad ^2 P} _{L ^1 _t \mathcal H ^1 _x} \le C \nor{\grad u} _{L ^2 _{t, x} } ^2$ in Hardy space due to compensated compactness \cite{coifman1993,lions1996,vasseur2010}. However, if $\partial \Omega \neq \varnothing$, we can separate $P = P _0 + P _1$ with $\nor{\grad ^2 P _1} _{L ^1 _t \mathcal H ^1 _x} \le C \nor{\grad u} _{L ^2 _{t, x} } ^2$, $P _1 = 0$ on $\partial \Omega$, and $P _0$ is harmonic in $\Omega$. 

Next, we remove the dependence on the pressure. We will have a similar result for the control of vorticity.

\begin{lemma}[Vasseur and Yang, \cite{vasseur2021}]
    Let $\frac{11}6 < p < 2$,
    $\frac{4 - 2p}{3 - p} < \theta \le \frac{14 p - 24}{13 p - 18}$.
    There exist universal constants $\eta > 0$ and $C _n > 0$ for $n \ge 0$, such that if a suitable weak solution $u$ satisfies
    \begin{align}
        \label{eqn:local-interpolate}
        \pth{
            \fint _{Q _\e (t, x)} \abs{\grad u} ^p \dx \dt 
        } ^\frac{1 - \theta}p \pth{
            \fint _{Q _\e (t, x)} \abs{\grad u} ^2 \dx \dt
        } ^\frac\theta2 \le \eta \e ^{-2},
    \end{align}
    then $(t, x)$ is a regular point, and $|\grad ^n \omega (t, x)| \le C _n \e ^{-n -2}$. 
\end{lemma}

\begin{proof}
    Pick $\nu = \frac12(1 - \frac1\theta)$. According to \cite{vasseur2021}, there exists $\eta _2, \eta _3 > 0$, such that for any $\delta \le \eta _2$, if $u$ is a suitable weak solution to \eqref{eqn:navier-stokes} in $Q _\e (t, x)$ satisfying 
    \begin{align}
        \label{eqn:yang2021}
        \delta ^{-2 \nu} \pth{
            \fint _{Q _\e (t, x)} \abs{\grad u} ^p \dx \dt 
        } ^\frac2p + \delta \fint _{Q _\e (t, x)} \abs{\grad u} ^2 \dx \dt \le \eta _3 \e ^{-4},
    \end{align}
    then $\abs{\grad ^n \omega (t, x)} \le C _n \e ^{-n - 2}$. Now we prove \eqref{eqn:local-interpolate} implies this condition with a suitable choice of $\eta$.

    Define $\eta _4, \eta _5$ by 
    \begin{align*}
        \pth{\fint _{Q _\e (t, x)} \abs{\grad u} ^2 \dx \dt} ^\frac12 = \eta _4 \e ^{-2}, \qquad 
        \pth{\fint _{Q _\e (t, x)} \abs{\grad u} ^p \dx \dt} ^\frac1p = \eta _5 \e ^{-2}.
    \end{align*}
    Then the assumption implies $\eta _4 ^{1 - \theta} \eta _5 ^\theta \le \eta$. Moreover, since $p < 2$, Jensen's inequality implies $\eta _4 \ge \eta _5$, and thus $\eta _5 \le \eta$. Define $\delta = \mins{\eta _2, \eta _4 ^{-2 \theta} \eta _5 ^{2 \theta}}$, then 
    \begin{align*}
        \delta ^{-2 \nu} \pth{
            \fint _{Q _\e (t, x)} \abs{\grad u} ^p \dx \dt 
        } ^\frac2p + \delta \fint _{Q _\e (t, x)} \abs{\grad u} ^2 \dx \dt
        & = 
        \delta ^{-2 \nu} \eta _5 ^2 \e ^{-4} + \delta \eta _4 ^2 \e ^{-4} .
    \end{align*}
    If $\delta = \eta _4 ^{-2 \theta} \eta _5 ^{2 \theta}$, then $$\delta ^{-2 \nu} \eta _5 ^2 + \delta \eta _4 ^2 = 2 \eta _4 ^{2 - 2 \theta} \eta _5 ^{2 \theta} \le 2 \eta ^2.$$ Here we used that $\delta ^{-2\nu} = \eta _4 ^{2 - 2\theta} \eta _5 ^{2 \theta-2}$. If $\delta = \eta _2 < \eta _4 ^{-2 \theta} \eta _5 ^{2 \theta}$, then $\eta _4 < \eta _5 \eta _2 ^{-1/2\theta}$, so 
    \begin{align*}
        \delta ^{-2 \nu} \eta _5 ^2 + \delta \eta _4 ^2 \le \eta _2 ^{-2 \nu} \eta _5 ^2 + \eta _2 \eta _5 ^2 \eta _2 ^{-1/\theta} = 2 \eta _2 ^{-2 \nu} \eta ^2.
    \end{align*}
    Therefore, if we choose $\eta = (2 + 2 \eta _2 ^{-2 \nu}) ^{-1/2} \eta _3 ^{1/2}$, then for both possible values of $\delta$, we always have $\delta ^{-2 \nu} \eta _5 ^2 + \delta \eta _4 ^2 \le \eta _3$. Therefore \eqref{eqn:yang2021} is satisfied and $\abs{\grad ^n \omega (t, x)} \le C _n \e ^{-n - 2}$.
\end{proof}

Using Jensen's inequality, we can deduce the following corollary.

\begin{corollary}
    \label{cor:vorticity-local}
    There exists $\eta > 0$, such that for any suitable weak solution $u$ to the Navier--Stokes equation in $\OmegaT$, if 
    \begin{align*}
        \fint _{Q _\e (t, x)} \abs{\grad u} ^2 \dx \dt \le \eta\e ^{-4},
    \end{align*}
    then 
    \begin{align*}
        \abs{\grad ^n \omega (t, x)} \le C _n \e ^{-n - 2}.
    \end{align*}
\end{corollary}

\subsection{Proof of the main results}

Now we prove the vorticity trace estimates for the Navier--Stokes equation. The previous subsection shows the following pointwise estimate on the regular part. 

\begin{lemma}
    \label{lem:pointwise}
    Let $(u, P)$ be a suitable weak solution to \eqref{eqn:navier-stokes} in $\OmegaT$, and suppose $u$ satisfies no-slip boundary condition \eqref{eqn:no-slip} if the boundary is non-empty. Define $f _1, f _2, f _3: \OmegaT \to \R$ by
    \begin{align*}
        f _1 &= \frac{\mm(\grad u) ^2}\eta, 
        & 
        f _2 &= \frac{\mm(\grad u) ^2 + |\grad ^2 P|}{\bar \eta},
        &
        f _3 &= \frac{|u| ^3 + |P| ^\frac32}{\epsilon _0}.
    \end{align*}
    Choose drift $b = u$ and define $s _1 = \mathscr S _4 ^\wedge [f _1]$, $s _2 = \mathscr S _4 ^\wedge [f _2]$. Choose drift $b = 0$ and define $s _3 = \mathscr S _3 ^\wedge [f _3]$.
    Then there exists $C _n > 0$ for every $n \ge 0$, such that for any {\em regular} point $(t, x) \in \OmegaT$, we have 
    \begin{align*}
        |\grad ^n \omega (t, x)| &\le C _n s _1 (t, x) ^{-n - 2} \qquad n \ge 0, \\
        |\grad ^n u (t, x)| &\le C _n s _2 (t, x) ^{-n - 1} \qquad n \ge 1, \\
        |\grad ^n u (t, x)| &\le C _n s _3 (t, x) ^{-n - 1} \qquad n \ge 0.
    \end{align*}
    Moreover, 
    \begin{align*}
        s _i (t, x) ^{-1} &\le \A _4 ^< [f _i] (t, x) ^{\frac14} \vee r _* (t, x) ^{-1}, \qquad i = 1, 2 \\
        s _3 (t, x) ^{-1} &\le \A _3 ^< [f _3] (t, x) ^{\frac13} \vee r _* (t, x) ^{-1}.
    \end{align*}
\end{lemma}

\begin{proof}
    Fix a regular point $(t, x)$, and let $\rho = s _1 (t, x)$. Since $\rho = \mathscr S _4 ^\wedge [f _1] (t, x) \le \mathscr S _4 [f _1]$, by definition we know that 
    \begin{align*}
        \A _4 ^\wedge [f _1] (t, x) = \fint _{Q _\rho (t, x)} \frac{\mm(\grad u) ^2}\eta \le \rho ^{-4}.
    \end{align*}
    Since $\mm (\grad u) \ge |\grad u|$, Corollary \ref{cor:vorticity-local} implies that $|\grad ^n \omega (t, x)| \le C _n \rho ^{-n - 2}$. 

    Note that the conditions in Lemma \ref{lem:radmrbar} are satisfied with $\alpha = 4$, $b = u$, and $f = f _1$.
    Since $(t, x)$ is a regular point, $(t, x) \in \Reg _4 ^< (f _1) \cup \Reg _4 ^= (f _1)$. If $(t, x) \in \Reg _4 ^= (f _1)$, then $\rho = \rbar (t, x)$. By Lemma \ref{lem:radmrbar}, we have $\rho \ge \rstar (t, x)$ and thus $s _1 (t, x) ^{-1} = \rho ^{-1} \le \rstar ^{-1} (t, x)$. If $(t, x) \in \Reg _4 ^< (f _1)$, then $\rho = \mathscr S _4 [f _1] (t, x) = \A _4 ^< [f _1] (t, x) ^{-\frac14}$.
    This completes the proof for $s _1$, and the proofs for $s _2, s _3$ are similar.
\end{proof}

With this lemma, we prove the main results listed in Section \ref{sec:introduction}.

\begin{proof}[Proof of Theorem \ref{thm:main}]
    Since $u$ is a classical solution, the singular set is empty, and any derivative of the solution is locally bounded everywhere. Note that $\mm (\grad u)$ is also locally bounded. 
    
    By Lemma \ref{lem:pointwise}, we know that $s _1 ^{-1} \ind*{s _1 < \rstar} \le \A _4 ^< [f _1] ^\frac14$. Apply Theorem \ref{thm:avg} \eqref{enu:weakweak1-drift} with $D = d = 3$, $f = f _1 = \frac1\eta \mm (\grad u) ^2$, $\alpha = 4$, $p _1 = q _1 = 1$, $p _2 = q _2 = \frac{d + 1}4$, we have $r _1 = -1$, $r _2 = -\frac{d + 1}4$, $\lambda = \frac{d + 1}4$, and 
    \begin{align*}
        \nor{\A _4 ^< [f _1]} _{L ^{\frac{d + 1}4, \infty} (\Gamma _T)} ^\frac{d + 1}4 &\le C \nor{\frac1\eta\mm(\grad u) ^2} _{L ^1 (\OmegaT)} \le \frac C\eta \nor{\grad u} _{L ^2 (\Omega _T)} ^2.
    \end{align*}
    Apply Theorem \ref{thm:avg} \eqref{enu:weakweak1-drift} again with $d \ge 2$, $p _1 = q _1 = 1$, $p _2 = \infty$, $q _2 = \frac{d - 1}4$, we have $r _1 = -1$, $r _2 = -\frac{d - 1}4$, $\lambda = \frac{d - 1}4$, and 
    \begin{align*}
        \nor{\A _4 ^< [f _1]} _{L ^\infty _t L ^{\frac{d - 1}4, \infty} (\Gamma _T)} ^{\frac{d - 1}4} &\le C \nor{\frac1\eta\mm(\grad u) ^2} _{L ^1 (\OmegaT)} \le \frac C\eta \nor{\grad u} _{L ^2 (\Omega _T)} ^2.
    \end{align*}
    The theorem is thus proven by using $s _1 ^{-1} \ind*{s _1 < \rstar} \le \A _4 ^< [f _1] ^\frac14$.
\end{proof}

\begin{proof}[Proof of Corollary \ref{cor:main}]
    For a classical solution, the first estimate is a direct consequence of Theorem \ref{thm:main} \eqref{enu:main} with $\Gamma _t = \Omega$ and $D = d = 3$, $|\grad \omega| \le C s _1 ^{-3}$, so 
    \begin{align*}
        \nor{\grad \omega \ind*{|\grad \omega| > C \rstar ^{-3}}} _{L ^{\frac43, \infty} (\OmegaT)} ^\frac43 
        & \le \nor{s _1 ^{-3} \ind*{s _1 < \rstar}} _{L ^{\frac43, \infty} (\OmegaT)} ^\frac43 \le C \nor{\grad u} _{L ^2 (\OmegaT)} ^2.
    \end{align*}
    And the second estimate is the case $d = 2$ with $|\omega| \le C s _1 ^{-2}$:
    \begin{align*}
        \nor{\omega \ind*{|\grad \omega| > C \rstar ^{-2}}} _{L ^{\frac32, \infty} (\Gamma _T)} ^\frac32
        & \le \nor{s _1 ^{-3} \ind*{s _1 < \rstar}} _{L ^{\frac43, \infty} (\OmegaT)} ^\frac43 \le C \nor{\grad u} _{L ^2 (\OmegaT)} ^2.
    \end{align*}
    For a suitable weak solution, we have the same bounds on the regular set, which is the complement of the singular set $\Sing (u)$ defined in Theorem \ref{thm:ckn}. 
    
    Note that the Navier--Stokes equation \eqref{eqn:navier-stokes} can be understood as an evolutionary Stokes equation with a forcing term $u \cdot \grad u$, which can be bounded by 
    \begin{align*}
        \nor{u \cdot \grad u} _{L ^\frac54 (\OmegaT)} &\le \nor{u} _{L ^\frac{10}3 (\OmegaT)} \nor{\grad u} _{L ^2(\OmegaT)} \\
        &\le \nor{u} _{L ^\infty _t L ^2 _x (\OmegaT)} ^\frac25 \nor{u} _{L ^2 _t L ^6 _x (\OmegaT)} ^\frac35 \nor{\grad u} _{L ^2(\OmegaT)} \le \nor{u _0} _{L ^2 (\Omega)} ^2.
    \end{align*}
    As $u, \grad u, P \in L ^\frac54 _{\loc} (\OmegaT)$, by linear Stokes theory, we know that $\grad ^2 u \in L ^\frac54 _{\loc} (\OmegaT)$ and $\grad P \in L ^\frac54 _{\loc} (\OmegaT)$. Therefore $\grad \omega$ is a measurable function in $\OmegaT$, and thus for almost every $t \in (0, T)$, the trace of $\omega (t)$ on a codimension $1$ hypersurface $\Gamma _t$ is a well-defined function. Since the singular set $\Sing (u)$ is a $\mathscr P ^1$-nullset, it vanishes at almost every $t \in (0, T)$, so $\Sing (u)$ is also a $\mu _T$-nullset. 
    
    As Lemma \ref{lem:pointwise} holds over the regular set, and the singular set is negligible, we finish the proof.
\end{proof}

\begin{proof}[Proof of Theorem \ref{thm:main-with-pressure}]
    The proof is analogous to Theorem \ref{thm:main} and Corollary \ref{cor:main}, where we use the estimate on $s _2$ in Lemma \ref{lem:pointwise} instead of $s _1$.
\end{proof}

\begin{proof}[Proof of Corollary \ref{cor:anisotropic}]
    This proof is also analogous to Theorem \ref{thm:main}. We apply Theorem \ref{thm:avg} \eqref{enu:weakweak1-drift} with $D = d = 3$, $\Gamma_t = \Omega$, $\alpha = 4$, $f = f _2 = \frac1{\eta _0} (\mm (\grad u) ^2 + |\grad ^2 P|)$, $p _1 = q _1 = 1$, we have $r _1 = -1$. Moreover, for $0 < p _2 < q _2 \le \infty$ with $\frac1{p _2} + \frac3{q _2} = 4$, we have 
    \begin{align*}
        \frac1{r _2} = \alpha - \frac2{p _2} - \frac{d}{q _2} = 4 - \frac2p - \frac3q = - \frac1{p _2}.
    \end{align*}
    So $\lambda = p _2 = \lambda p _1$ and $q _2 > p _2 = \lambda q _1$. By Theorem \ref{thm:avg} \eqref{enu:weakweak1-drift}, we have
    \begin{align*}
        \nor{\Aa ^< [f _2]} _{L ^{p _2, \infty} _t L ^{q _2, \infty} _x (\OmegaT)} ^{p _2} &\le C \nor{f _2} _{L ^1 (\OmegaT)} \le \frac C{\eta _0} \nor{\grad u} _{L ^2 (\OmegaT)} ^2.
    \end{align*}
    Here in a torus, $\grad ^2 P$ in $L ^1$ can be controlled by $\grad u$ in $L ^2$ (see Remark \ref{rmk:pressure} below). Therefore, we have
    \begin{align*}
        \nor{s _2 ^{-1} \ind*{s _2 < \rstar}} _{L ^{4 p _2, \infty} _t L ^{4 q _2, \infty} _x (\OmegaT)} ^{4 p _2} &\le \nor{\Aa ^< [f _2]} _{L ^{p _2} _t L ^{q _2} _x (\OmegaT)} ^{p _2} \le \frac C{\eta _0} \nor{\grad u} _{L ^2 (\OmegaT)} ^2.
    \end{align*}
    Moreover, when $\Omega = \mathbb T ^3$, $r _* (t, x) = \mins{1, \sqrt t}$. So 
    \begin{align*}
        \nor{r _* ^{-1}} _{L ^{4 p _2, \infty} (t _0, T; L ^{4 q _2, \infty} (\mathbb T ^3))} ^{4 p _2} &\le \int _{t _0} ^T t ^{-2 p _2} \vee 1 \dt 
        \le (T - t _0) \max \{t _0 ^{-2 p _2}, 1\}.
    \end{align*}
    Combined, we have proven that 
    \begin{align*}
        \nor{s _2 ^{-1}} _{L ^{4 p _2, \infty} (t _0, T; L ^{4 q _2, \infty} (\mathbb T ^3))} ^{4 p _2} \le C \pth{
            \nor{\grad u} _{L ^2 (\OmegaT)} ^2 + (T - t _0) \max \{t _0 ^{-2 p _2}, 1\}
        }.
    \end{align*}
    The theorem is proven by $|\grad ^n u| \le C _n s _2 ^{-n - 1}$ and setting $p = \frac{4 p _2}{n + 1}$, $q = \frac{4 q _2}{n + 1}$.

    Next, we apply Theorem \ref{thm:avg} \eqref{enu:Ba} with $\frac1{p _2} + \frac1{q _2} = 2$. We have 
    \begin{align*}
        \frac1{r _2} = \alpha - \frac2{p _2} - \frac{d}{q _2} = 2\pth{2 - \frac1p} - \frac3q = - \frac1{q _2}.
    \end{align*}
    So $\lambda = q _2 = \lambda q _1$ and $p _2 > q _2 = \lambda p _1$. By Theorem \ref{thm:avg} \eqref{enu:Ba}, we have
    \begin{align*}
        \nor{\Aa ^< [f _2]} _{L ^{p _2, \infty} _t L ^{q _2} _x (\OmegaT)} ^{q _2} &\le C \nor{f _2} _{L ^1 (\OmegaT)} \le \frac C{\eta _0} \nor{\grad u} _{L ^2 (\OmegaT)} ^2.
    \end{align*}
    Similarly, by bounding $r _* ^{-1}$, we have 
    \begin{align*}
        \nor{s _2 ^{-1}} _{L ^{4p _2, \infty} (t _0, T; L ^{4q _2} (\mathbb T ^3))} ^{4q _2} \le C \pth{
            \nor{\grad u} _{L ^2 (\OmegaT)} ^2 + (T - t _0) \max \{t _0 ^{-2 q _2}, 1\}
        }.
    \end{align*}
    The theorem is proven by setting $p = \frac{4 p _2}{n + 1}$, $q = \frac{4 q _2}{n + 1}$.

    Finally, we apply Theorem \ref{thm:avg} \eqref{enu:weakweak1-drift} to prove the last statement. The proof is similar so we omit the details.
\end{proof}

\begin{remark}
    \label{rmk:pressure} 
    By taking the divergence of \eqref{eqn:navier-stokes}, we know that 
    \begin{align*}
        -\La P = \div (u \cdot \grad u) = \div ((\grad u) ^\top) \cdot u + (\grad u) ^\top : \grad u.    
    \end{align*}
    Note that $\div ((\grad u) ^\top) = 0$, $\curl \grad u = 0$. When $\Omega = \RR3$ or $\mathbb T ^3$, by the div-curl lemma \cite{coifman1993}, we know that 
    \begin{align*}
        \nor{\La P (t)} _{\mathcal H ^1 (\Omega)} \le C \nor{\grad u (t)} _{L ^2 (\Omega)} ^2.
    \end{align*}
    where $\mathcal H ^1$ stands for the Hardy $\mathcal H ^p$ space with $p = 1$. Since the Riesz transform is bounded on $\mathcal H ^1$, we obtain 
    \begin{align*}
        \nor{\grad ^2 P} _{L ^1 (\OmegaT)} \le C \nor{\grad ^2 P} _{L ^1 (0, T; \mathcal H ^1 (\Omega))} \le C \nor{\grad u} _{L ^2 (\OmegaT)} ^2.
    \end{align*}
    See also \cite[section 3.2]{lions1996}.
\end{remark}

Corollary \ref{cor:anisotropic} applies to the case when $\Omega = \mathbb T ^3$ has no boundary. 
In fact, if we only wish to bound the vorticity and its derivatives, similar bounds hold in any domain $\Omega$ of finite measure because $f _1$ requires nothing from the pressure.
When there is a boundary, if we want to bound $\grad ^n u$ we still need to control the pressure term. In this case, we can use the following proposition to handle the pressure term by a splitting scheme.

\begin{proposition}
    \label{prop:gradu}
    Let $\Omega \subset \RR3$ be a bounded set with $C ^2$ boundary satisfying Assumption \ref{ass:lipschitz}. 
    Let $u$ be a classical solution to \eqref{eqn:navier-stokes}-\eqref{eqn:no-slip} in $(0, T) \times \Omega$, with initial kinetic energy $\nor{u (0)} _{L ^2 (\Omega)} \le E$ for some $E \ge 0$. 
    For any $\delta \in (0, r _0)$, denote 
    \begin{align*}
        \Gamma _T ^\delta &= \set{(t, x) \in \Gamma _T: \dist _\mathcal P ((t, x), \partial _\P \OmegaT) > \delta}.
    \end{align*} 
    Then for any $K \subset \Gamma _T ^\delta$, for every $n \ge 1$, the following holds.
    For any $0 < p < q \le \infty$ with $\frac1p + \frac dq = n + 1$:
    \begin{align*}
        \nor{\grad ^n u} _{L ^{p, \infty} _t L ^{q, \infty} _x (K)} \le C (\Omega, T, K, \delta, E, p, q) ^{n + 1}.
    \end{align*}
    For any $0 < q < p \le \infty$ with $\frac2p + \frac{d - 1}q = n + 1$:
    \begin{align*}
        \nor{\grad ^n u} _{L ^{p, \infty} _t L ^{q} _x (K)} \le C (\Omega, T, K, \delta, E, p, q) ^{n + 1}.
    \end{align*}
    When $p = q = \frac{d + 1}{n + 1}$: 
    \begin{align*}
        \nor{\grad ^n u} _{L ^{\frac{d + 1}{n + 1}, \infty} (K)} \le C (\Omega, T, K, \delta, E) ^{n + 1}.
    \end{align*}
\end{proposition}

\begin{proof}
    We only sketch how to control the pressure term using the idea of Lions \cite[section 3.3]{lions1996}. We can split $P = P _0 + P _1$, $\La P _0 = 0$, $\fint _\Omega P _0 (t) \dx = 0$ for a.e. $t \in (0, T)$, and $P _1 = 0$ on $\partial \Omega$. We have 
    \begin{align*}
        \nor{\grad P _1} _{L ^\frac54 ((\delta ^2, T) \times \Omega)} + \nor{\grad P _0} _{L ^\frac54 ((\delta ^2, T) \times \Omega)} \le C
            (\nor{u \cdot \grad u} _{L ^\frac54 (\Omega _T)} + 1).
    \end{align*}
    Using Riesz transform on $\La P _1$ and harmonicity on $P _0$, we have
    \begin{align*}
        \nor{\grad ^2 P _1} _{L ^1 (0, T; \mathcal H ^1 (\Omega))} &\le C \nor{\grad u} _{L ^2 (\Omega _T)} ^2, 
        \\
        \nor{\grad ^2 P _0} _{L ^1 (\delta ^2, T; L ^1 (\Omega ^\e))} &\le C (\nor{u \cdot \grad u} _{L ^\frac54 (\Omega _T)} + 1), \qquad k \ge 0.
    \end{align*}
    Here $\Omega ^\e = \set{x \in \Omega: \dist (x, \partial \Omega) > \e}$ for some $\e > 0$ to be chosen. Note that both $\nor{\grad u} _{L ^2} ^2$ and $\nor{u \cdot \grad u} _{L ^{\frac54}}$ are dominated by $\nor{u _0} _{L ^2} ^2$ (see the Proof of Theorem \ref{thm:main}). Therefore, we have
    \begin{align*}
        \int _{\delta ^2} ^T \int _{\Omega ^\e} \mm (\grad u) ^2 + |\grad ^2 P| \dx \dt \le C (\nor{u _0} _{L ^2 (\Omega)} ^2 + 1).
    \end{align*}
    Take $(t, x) \in K$, then $\dist (x, \partial \Omega) \ge \delta$. Suppose $s _2 (t, x) = \rho < \e$. Recall that $Q _\rho (t, x)$ are defined by \eqref{eqn:flow}-\eqref{eqn:Qe} with $b = u$. Since at any time $\nor{u (t)} _{L ^2 (\Omega)} \le \nor{u _0} _{L ^2 (\Omega)} \le E$, we have 
    \begin{align*}
        |b _\rho| (t, x) \le \nor{u (t)} _{L ^2} |B _\rho| ^{-\frac12} \le C \rho ^{-\frac32}.
    \end{align*}
    In a timespan of length $\rho ^2$, 
    \begin{align*}
        |X _\rho (s; t, x) - x| \le C \rho ^{-\frac32} \rho ^2 \le C \rho ^{\frac12}.
    \end{align*}
    This shows that $\dist (y, \partial \Omega) \ge \delta - C \e ^\frac12 - \e > 2\e$ for any $(s, y) \in Q _\rho (t, x)$, if we choose $\e$ sufficiently small. Thus, $Q _{s _2 (t, x)} (t, x) \subset \subset \Omega _\e$ for every $(t, x) \in K$,
    so the value of $s _2 = \Saw [f _2]$ inside $\Omega ^\delta$ is not affected by the value of $f _2 = \mm (\grad u) ^2 + |\grad ^2 P|$ outside $\Omega ^\e$. 
    
    Same as in Corollary \ref{cor:anisotropic}, for $0 < p _2 < q _2 \le \infty$ with $\frac1{p _2} + \frac d{q _2} = 4$ we have 
    \begin{align*}
        \frac1{r _2} = \alpha - \frac2{p _2} - \frac{d}{q _2} = 4 - \frac2{p _2} - \frac d{q _2} = - \frac1{p _2}.
    \end{align*}
    Thus
    \begin{align*}
        \nor{s _2 ^{-1} \ind*{s _2 < \rstar}} _{L ^{4 p _2, \infty} _t L ^{4 q _2, \infty} _x (K)} ^{4 p _2} &\le \nor{\Aa ^< [f _2]} _{L ^{p _2} _t L ^{q _2} _x ((0, T) \times \Omega ^\e)} ^{p _2} \le C.
    \end{align*}
    Because $r _*$ is bounded from below in $(\delta ^2, T) \times \Omega ^\delta$, we have
    \begin{align*}
        \nor{s _2 ^{-1}} _{L ^{4 p _2, \infty} _t L ^{4 q _2, \infty} _x (K)} ^{4 p _2} &\le C + \nor{r _* ^{-1}} _{L ^{4 p _2, \infty} _t L ^{4 q _2, \infty} _x (K)} ^{4 p _2} \le C.
    \end{align*}
    By $|\grad ^n u| \le s _2 ^{-1}$, we finish the proof for the first estimate.

    For $0 < q _2 < p _2 \le \infty$ with $\frac2{p _2} + \frac{d - 1}{q _2} = 4$. We have 
    \begin{align*}
        \frac1{r _2} = \alpha - \frac2{p _2} - \frac{d}{q _2} = \frac{d - 1}{q _2} - \frac dq = - \frac1{q _2}.
    \end{align*}
    Similarly, we have 
    \begin{align*}
        \nor{s _2 ^{-1}} _{L ^{4 p _2, \infty} _t L ^{4 q _2} _x (K)} ^{4 p _2} &\le C + \nor{r _* ^{-1}} _{L ^{4 p _2, \infty} _t L ^{4 q _2, \infty} _x (K)} ^{4 p _2} \le C.
    \end{align*}
    By $|\grad ^n u| \le s _2 ^{-1}$, we finish the proof for the second estimate. The third is similar so we do not repeat here.
\end{proof}

\begin{proof}[Proof of Proposition \ref{prop:blow-up}]

    First, we want to apply Theorem \ref{thm:avg} on $f _3$ defined in Lemma \ref{lem:pointwise}, with $D = d = 3$, $\Gamma _t = \Omega = \mathbb T ^3$, $\alpha = 3$, $p _1 = p / 3$, $q _1 = q / 3$, $p _2 = p' / 3$, $q _2 = q' / 3$, we have 
    \begin{align*}
        \frac1{r _1} = 3 - \frac2{p _1} - \frac3{q _1} = 3 \pth{1 - \frac2p - \frac3q} = 0
    \end{align*}
    and similarly $\frac1{r _2} = 0$. Since $r _1 = r _2 = \infty$, $r _2 = \lambda r _1$ holds for any $\lambda > 0$. Note that $q \le p < \infty$ implies $q _1 \le p _1 < \infty$. We follow the following steps:

    \noindent \textbf{Step 1: case $p = q = 5$ for \eqref{eqn:use-u}.}
    
    In this case, $p _1 = q _1 = \frac53$. By choosing $\lambda$ small, we can make sure $q _2 > \lambda q _1$ and $p _2 > \lambda p _1$. Using Theorem \ref{thm:avg} \eqref{enu:p=q>1}, we have 
    \begin{align*}
        \nor{\Aa ^< [f _3]} _{L ^{p} _t L ^{q} _x (\Omega _T)} ^{\lambda} &\le C \nor{f _3} _{L ^\frac53 (\Omega _T)}.
    \end{align*}
    Here 
    \begin{align*}
        \nor{f _3} _{L ^\frac53 (\OmegaT)} = C \nor{u} _{L ^5 ((0, T) \times \mathbb T ^3)} ^3 + C \nor{P} _{L ^\frac52 ((0, T) \times \mathbb T ^3)} ^{\frac32} \le C \nor{u} _{L ^5 ((0, T) \times \mathbb T ^3)} ^3.
    \end{align*}
    Here we used $P = (-\La) ^{-1} \div \div (u \otimes u)$, and $(-\La) ^{-1} \div \div$ is a Riesz transform, so $\nor{P} _{L ^\frac52} \le \nor{u \otimes u} _{L ^\frac52} \le \nor{u} _{L ^5} ^2$. Hence Lemma \ref{lem:pointwise} implies 
    \begin{align*}
        \nor{s _3 ^{-1} \ind*{s _3 < r _*}} _{L ^{3 p _2} _t L ^{3 q _2} _x (\Omega _T)} ^{3 \lambda} \le C \nor{u} _{L ^5 ((0, T) \times \mathbb T ^3)} ^3.
    \end{align*}
    Since $r _*$ is bounded from below in $(t _0, T) \times \mathbb T ^3$, we have
    \begin{align*}
        \nor{s _3 ^{-1}} _{L ^{p'} _t L ^{q'} _x ((t _0, T) \times \mathbb T ^3)} \le C \left( 
            \nor{u} _{L ^5 ((0, T) \times \mathbb T ^3)} ^\frac1\lambda + 1
        \right).
    \end{align*}
    Because Lemma \ref{lem:pointwise} also implies $|u| \le s _3 ^{-1}$ and $|\grad u| \le s _3 ^{-2}$, we conclude  
    \begin{align*}
        \nor{u} _{L ^{p'} _t L ^{q'} _x ((t _0, T) \times \mathbb T ^3)} + \nor{\grad u} _{L ^{\frac{p'}2} _t L ^{\frac{q'}2} _x ((t _0, T) \times \mathbb T ^3)} ^\frac12 \le C \pth{
            \nor{u} _{L ^5 ((0, T) \times \mathbb T ^3)} ^\frac1\lambda + 1
        }.
    \end{align*}
    We can pick $\gamma = \frac1\lambda$. This proves \eqref{eqn:use-u} for $p = q = 5$.

    \noindent \textbf{Step 2: case $3 < q < p < \infty$ for \eqref{eqn:use-u}.} 
    
    Given Step 1, it suffices to work with $p' = q' = 5$. Again, choose $\lambda$ sufficiently small so that $q _2 = \frac53 > \lambda q _1$ and $p _2 = \frac53 > \lambda p _1$. Note that $\alpha = 3 > \frac3{q _2} = \frac95$. Using Theorem \ref{thm:avg} \eqref{enu:Bb}, we have
    \begin{align*}
        \nor{\Aa ^< [f _3]} _{L ^{\frac53} _{t, x} (\Omega _T)} ^{\lambda} &\le C \nor{f _3} _{L ^{p _1} _t L ^{q _1} _x (\Omega _T)}.
    \end{align*}
    Similar as Step 1, we have $\nor{f _3} _{L ^{p _1} _t L ^{q _1} _x} \le \nor{u} _{L ^p _t L ^q _x} ^3$. So  
    \begin{align*}
        \nor{u} _{L ^{5} _{t, x} ((\frac{t _0}2, T) \times \mathbb T ^3)} + \nor{\grad u} _{L ^{\frac52} _{t, x} ((t _0, T) \times \mathbb T ^3)} ^\frac12 \le C \pth{
            \nor{u} _{L ^5 ((0, T) \times \mathbb T ^3)} ^\frac1\lambda + 1
        }.
    \end{align*}
    Combined with Step 1, we prove \eqref{eqn:use-u} for $3 < q < p < \infty$.

    \noindent \textbf{Step 3: case $3 < q \le \frac{15}4 < 10 \le p < \infty$ for \eqref{eqn:use-gradu}.} 
    
    Recall the conservation of momentum: $\int _{\mathbb T ^3} u (t) \dx = \int _{\mathbb T ^3} u _0 \dx = 0$. By Sobolev embedding in $\mathbb T ^3$, 
    \begin{align*}
        \nor{u} _{L ^\frac p2 _t L ^{q'} _x ((0, T) \times \mathbb T ^3)} \le C \nor{\grad u} _{L ^\frac p2 _t L ^\frac q2 _x ((0, T) \times \mathbb T ^3)}.
    \end{align*}
    Here $q' = \frac{3 \frac q2}{3 - \frac q2} = \frac{3 q}{6 - q}$, so 
    \begin{align*}
        \frac2{\frac p2} + \frac3{q'} = \frac4p + \frac{3 (6 - q)}{3 q} = \frac4p + \frac6q - 1 = 1.
    \end{align*}
    Moreover, $\frac p2 \ge 5$ and $q' \le 5$. By Step 1 and Step 2, \eqref{eqn:use-u} implies \eqref{eqn:use-gradu} in this range.

    Next, we work with $f _2$ defined in Lemma \ref{lem:pointwise} instead of $f _3$, with $D = d = 3$, $\Gamma _t = \Omega = \mathbb T ^3$, $\alpha = 4$, $p _1 = p / 4$, $p _2 = p / 4$, $q _1 = q / 4$, $q _2 = q / 4$. 

    \noindent \textbf{Step 4: case $q = p = 5$ for \eqref{eqn:use-gradu}.} 
    
    We let $p' = 10$ and $q' = \frac{15}4$. Choose $\lambda = \frac34$. Then $p _1 = q _1 = \frac54$, $p _2 = \frac52$, $q _2 = \frac{15}{16}$, $p _2 > \lambda p _1$ and $q _2 = \lambda q _1$. Using Theorem \ref{thm:avg} \eqref{enu:p=q>1-drift}, we have
    \begin{align*}
        \nor{\Aa ^< [f _2]} _{L ^{\frac52} _t L ^{\frac{15}{16}} _x (\Omega _T)} ^{\frac34} &\le C \nor{f _2} _{L ^{\frac54} _{t, x} (\Omega _T)} = C \nor{\grad u} _{L ^\frac52 _{t, x} (\OmegaT)} ^2 + C \nor{\grad ^2 P} _{L ^\frac54 _{t, x} (\OmegaT)} \\
        &\le C \nor{\grad u} _{L ^\frac52 _{t, x} (\OmegaT)} ^2.
    \end{align*}
    Here we used Remark \ref{rmk:pressure} to bound $\nor{\grad ^2 P} _{L ^\frac54}$ by $\nor{\grad u} _{L ^\frac52} ^2$ using boundedness of Riesz transform. Similar to Step 1, 
    \begin{align*}
        \nor{\grad u} _{L ^5 _t L ^\frac85 _x ((\frac{t _0}4, T) \times \mathbb T ^3)} ^\frac32 \le C \pth{
            \nor{\grad u} _{L ^\frac52 _{t, x} ((0, T) \times \Omega)} ^2 + 1
        }.
    \end{align*}
    This combined with Step 3 proves \eqref{eqn:use-gradu} for $p = q = 5$.

    \noindent \textbf{Step 5: case $4 < q < p < 8$ for \eqref{eqn:use-gradu}.} 

    We let $p' = q' = 5$. Choose $\lambda = \frac58$. Then $p _2 = q _2 = \frac54$, $2 > p _1 > q _1 > 1$, and $p _2 > \lambda p _1 > \lambda q _1$. Note $\alpha = 4 > \frac d{q _2} = \frac{12}5$. Using Theorem \ref{thm:avg} \eqref{enu:Bb}, we have
    \begin{align*}
        \nor{\Aa ^< [f _2]} _{L ^{\frac54} _t L ^{\frac54} _x (\Omega _T)} ^{\frac58} &\le C \nor{f _2} _{L ^{p _1} _t L ^{q _1} _x} = C \nor{\grad u} _{L ^{2 p _1} _t L ^{2 q _1} _x} ^2 + C \nor{\grad ^2 P} _{L ^{p _1} _t L ^{q _1} _x} \\
        & \le  \nor{\grad u} _{L ^{2 p _1} _t L ^{2 q _1} _x} ^2.
    \end{align*}
    Similarly, we can control 
    \begin{align*}
        \nor{\grad u} _{L ^\frac52 _{t, x} ((\frac{t _0}8, T) \times \mathbb T ^3)} ^\frac54 \le C \pth{
            \nor{\grad u} _{L ^{\frac p2} _t L ^{\frac q2} _x ((0, T) \times \mathbb T ^3)} ^2 + 1
        }.
    \end{align*}
    This combined with Step 4 proves \eqref{eqn:use-gradu} for $4 < q < p < 8$. We have now proved all the claims in Proposition \ref{prop:blow-up}.
\end{proof}

\begin{appendix}
    
\section{Unboundedness of averaging operator in $L ^1$}
\label{app:cantor}

We claim that the averaging operator $\Aa$ is not bounded in $L ^1 (\R)$ for $\alpha = \frac12$, by constructing a sequence of functions $f _k$ as the following. For $k \ge 1$, define
\begin{align*}
    S _{k} = \set{
        \sum _{j = 1} ^k a _j 4 ^{-j}: a _j = 0 \text{ or } 2
    }.
\end{align*}
$S _k$ contains fractions whose quaternary expression contains only $0$ and $2$. Define $A _0 = [0, 1)$, and for $k \ge 1$ define
\begin{align*}
    A _k = \bigcup _{x \in S _k} [x, x + 4 ^{-k}).
\end{align*}
We draw the definition of $A _k$ in Figure \ref{fig:cantor}. Clearly, $|A _k| = 2 ^{-k}$. Finally, we define 
\begin{align*}
    f _k = 2 ^\frac32 \cdot 2 ^k \ind{A _k}.
\end{align*}
So that $f _k$ are uniformly bounded in $L ^1 (\R)$.
\begin{figure}[htbp]
    \centering
    \begin{tikzpicture}
        \draw[ultra thick] (0, 0) -- (4, 0);

        \draw[ultra thick] (0, -0.1) -- (1, -0.1);
        \draw[ultra thick] (2, -0.1) -- (3, -0.1);

        \draw[ultra thick] (0, -0.2) -- (1/4, -0.2);
        \draw[ultra thick] (2/4, -0.2) -- (3/4, -0.2);
        \draw[ultra thick] (2, -0.2) -- (2 + 1/4, -0.2);
        \draw[ultra thick] (2 + 2/4, -0.2) -- (2 + 3/4, -0.2);

        \draw[ultra thick] (0, -0.3) -- (1/16, -0.3);
        \draw[ultra thick] (2/16, -0.3) -- (3/16, -0.3);
        \draw[ultra thick] (2/4, -0.3) -- (2/4 + 1/16, -0.3);
        \draw[ultra thick] (2/4 + 2/16, -0.3) -- (2/4 + 3/16, -0.3);

        \draw[ultra thick] (2 + 0, -0.3) -- (2 + 1/16, -0.3);
        \draw[ultra thick] (2 + 2/16, -0.3) -- (2 + 3/16, -0.3);
        \draw[ultra thick] (2 + 2/4, -0.3) -- (2 + 2/4 + 1/16, -0.3);
        \draw[ultra thick] (2 + 2/4 + 2/16, -0.3) -- (2 + 2/4 + 3/16, -0.3);
        \draw[] (1.5, -0.5) node[] {$\dots$};
    \end{tikzpicture}
    \caption{Construction of a Cantor-type measure}
    \label{fig:cantor}
\end{figure}
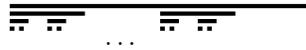

Now, for every $k \ge 1$, $x \in A _{k - 1} \setminus A _{k}$, $j \ge k$, we know that $x \in [x ^* + 4 ^{-k}, x ^* + 2 \cdot 4 ^{-k})$ for some $x ^* \in S _{k}$. Then 
\begin{align*}
    \int _{B _{2 \cdot 4 ^{-k}} (x)} f _j \ge \int _{x ^*} ^{x ^* + 4 ^{-k}} f _j = 2 ^\frac32 \cdot 2 ^k \cdot 4 ^{-k} = 2 ^\frac32 \cdot 2 ^{-k}.
\end{align*}
Thus 
\begin{align*}
    \fint _{B _{2 \cdot 4 ^{-k}} (x)} f _j \ge 2 ^{-\frac12} \cdot 2 ^{k} = (2 \cdot 4 ^{-k}) ^{-\frac12}.
\end{align*}
By definition, $\Sa [f _j] (x) \le 2 \cdot 4 ^{-k}$, and $\Aa [f _j] (x) \ge (2 \cdot 4 ^{-k}) ^{-\frac12} = 2 ^{-\frac12} \cdot 2 ^k$. So
\begin{align*}
    \abset{\Aa [f _j] (x) \ge 2 ^{-\frac12} \cdot 2 ^k} \ge \abs{A _{k - 1} \setminus A _k} = 2 ^{-k}, \qquad \forall k \le j.
\end{align*}
Hence $\set{\Aa [f _j]} _j$ is not uniformly bounded in $L ^{1, q} (\R)$ for any $q < \infty$. In particular, letting $f _j \to \infty$ we obtain a Cantor measure $\nu$ such that for any $k \ge 1$, $x ^* \in S _k$, 
\begin{align*}
    \nu ([x ^*, x ^* + 4 ^{-k})) &= 2 ^\frac32 \cdot 2 ^{-k}, 
    &
    \nu ([x ^* + 4 ^{-k}, x ^* + 2 \cdot 4 ^{-k})) &= 0.
\end{align*}
Then $|\nu| = 2 ^\frac32$ is a finite measure, but $\Aa (\nu)$ is not in $L ^{1, q} (\R)$ for any $q < \infty$.

\section{Lorentz norm}
\label{app:lorentz}

We claim that the quasinorms $L ^{1, \infty} _{t} L ^{1, \infty} _x$ and $L ^{1, \infty} _{t, x}$ are not comparable. 

\begin{lemma}
    Let $\Omega = (0, 1)$. For any $\e > 0$, there exist a pair of functions $u _1, u _2 \in L ^1 _\loc ((0, 1) \times \Omega)$ with 
    \begin{align*}
        \nor{u _1} _{L ^{1, \infty} (0, 1; L ^{1, \infty} (\Omega))} = \nor{u _2} _{L ^{1, \infty} (0, 1; L ^{1, \infty} (\Omega))} = 1,
    \end{align*}
    but $\nor{u _1} _{L ^{1, \infty} ((0, 1) \times \Omega)} \le \e$ and $\nor{u _2} _{L ^{1, \infty} ((0, 1) \times \Omega)} = +\infty$.
\end{lemma}

\begin{proof}
    Define 
    \begin{align*}
        u _1 (t, x) = e ^\frac t\e \ind{[0, e ^{-\frac t\e}]} (x), \qquad u _2 (t, x) = \frac1{t x}.
    \end{align*}
    Then for any $t \in (0, 1)$, $\nor{u _1 (t)} _{L ^{1, \infty} (\Omega)} = 1$ and $\nor{u _2 (t)} _{L ^{1, \infty} (\Omega)} = \frac1t$. They are both $L ^{1, \infty}$ functions in $t$.
    
    Note that for any $\alpha > 0$, 
    \begin{align*}
        \set{u _1 > \alpha} &= \set{(t, x) \in (0, 1) \times \Omega: t > \e \log _+ \alpha, 0 < x < e ^{-\frac t\e}} \\
        \implies 
        \abset{u _1 > \alpha} &= \int _{\e \log _+ \alpha} ^1 e ^{-\frac t\e} \d t = \e \int _{\log _+ \alpha} ^\frac1 \e e ^{-s} \d s < \frac\e\alpha,
    \end{align*}
    hence $\nor{u _1} _{L ^{1, \infty} ((0, 1) \times \Omega)} \le \e$. On the other hand, for any $\alpha > 1$,  
    \begin{align*}
        \set{u _2 > \alpha} &= \set{(t, x) \in (0, 1) \times \Omega:  0 < t < \frac1\alpha \text{ or } 0 < x < \frac1{t \alpha}} \\
        \implies \abset{u _2 > \alpha} &= \frac1\alpha + \int _{\frac1\alpha} ^1 \frac1{t \alpha} \d t = \frac1\alpha \pth{1 + \log \alpha}
    \end{align*}
    hence $\nor{u _2} _{L ^{1, \infty} ((0, 1) \times \Omega)} = +\infty$.
\end{proof}

\begin{remark}
    The example $u _1$ indeed satisfies $\nor{u _1 (t)} _{L ^{1, q} (\Omega)} \equiv 1$ for any $q \in [1, \infty]$. Moreover, by interpolation, we have $\nor{u _1} _{L ^{1, q} ((0, 1) \times \Omega)} \le \e ^\frac{p - 1}p$. Hence $$L ^{1, q} _{t, x} \not \hookrightarrow L ^{1, q} _t L ^{1, q} _x$$ for any $q > 1$. 
\end{remark}

Although they are not equivalent, we show below that it is still possible to interpolate between an isotropic norm and a repeated norm. 

\begin{lemma}
    \label{lem:lorentz-interpolation}
    Let $T \in (0, \infty]$, and let $\Omega$ be a measurable space. Suppose $\mu _t$ is a measure on $\Omega$ for every $t > 0$, and we define $\mu _T$ by $\d \mu _T = \dt \d \mu _t$ as before. 
    \begin{enumerate}[\upshape (a)]
        \item If $f \in L ^{\infty} _t L ^{q _0, \infty} _x$ and $f \in L ^{1, \infty} _{t, x}$ for some $q _0 \in (0, 1)$, then $f \in L ^{p, \infty} _t L ^{q, \infty} _x$ with 
        \begin{align}
            \label{eqn:pq-range}
            \frac{1 - q _0}p + \frac {q _0} q = 1, \qquad 1 < p < \infty, q _0 < q < 1.
        \end{align}
        Moreover, 
        \begin{align*}
            \nor{f} _{L ^{p, \infty} _t L ^{q, \infty} _x} \le C (p, q, q _0) \nor{f} _{L ^{1, \infty} _{t, x}} ^\frac1p \nor{f} _{L ^{\infty} _t L ^{q _0, \infty} _x} ^{1 - \frac 1p}.
        \end{align*}

        \item If $f \in L ^{p _0, \infty} _t L _x ^{\infty}$ and $f \in L ^{1, \infty} _{t, x}$ for some $p _0 \in (0, 1)$, then $f \in L ^{p, \infty} _t L ^{q, \infty} _x$ with 
        \begin{align}
            \label{eqn:pq-range-2}
            \frac{p _0} p + \frac{1 - p _0}q = 1,\qquad p _0 < p < 1, 1 < q < \infty.
        \end{align}
        Moreover, 
        \begin{align*}
            \nor{f} _{L ^{p, \infty} _t L ^{q, \infty} _x} \le C (p, q, p _0) \nor{f} _{L ^{1, \infty} _{t, x}} ^\frac1q \nor{f} _{L ^{p _0, \infty} _t L ^{\infty} _x} ^{1 - \frac 1q}.
        \end{align*}
    \end{enumerate}
    
    \begin{proof}
        Define $S (t, \alpha) = \mu _t (\set{x \in \Omega: f (t, x) > \alpha})$, and define $\alpha _k = 2 ^k$. Then 
        \begin{align*}
            \nor{f (t)} _{L ^{q, \infty} (\Omega, \mu _t)} ^q \approx \sup _k \alpha _k ^q S (t, \alpha _k).
        \end{align*}
        Fix $\beta > 0$, and denote 
        \begin{align*}
            B := \set{t \in (0, T): \nor{f (t)} _{L ^{q, \infty} (\Omega, \mu _t)} > \beta}.
        \end{align*}
        For every $t \in B$, there exists $k \in \mathbb Z$ such that $\alpha _k ^q S (t, \alpha _k) > \beta ^q$. Hence, we can partition $B = \cup _k B _k$ into a sequence of pairwise disjoint sets $B _k \subset B$, such that 
        \begin{align}
            \label{eqn:Lpq}
            \alpha _k ^q S (t, \alpha _k) > \beta ^q, \qquad \forall t \in B _k.
        \end{align}
        Moreover, if $\nor{f} _{L ^{1, \infty} _{t, x}} = K < \infty$, then for any $\alpha > 0$, 
        \begin{align*}
            \frac K\alpha > \mu _T (\set{f > \alpha}) = \int _0 ^T S (t, \alpha) \d \alpha \ge \sum _{k \in \mathbb Z} \int _{B _k} S (t, \alpha) \d t.
        \end{align*}
        In particular, for any $k \in \mathbb Z$, we have
        \begin{align}
            \label{eqn:L18}
            \frac K{\alpha _k} > \int _{B _k} S (t, \alpha) \d t > |B _k| \beta ^q \alpha _k ^{-q}.
        \end{align}

        \begin{enumerate}[\upshape (a)]
            \item Without loss of generality assume $\nor{f} _{L ^\infty _t L ^{q _0, \infty} _x} = 1$, thus for any $t \in (0, T)$ and any $\alpha > 0$, we have 
            \begin{align*}
                \alpha ^{q _0} S (t, \alpha) \le 1.
            \end{align*}
            In particular, choose $t \in B _k$ and $\alpha = \alpha _k$, \eqref{eqn:Lpq} yields
            \begin{align*}
                1 \ge \alpha _k ^{q _0} S (t, \alpha _k) > \beta ^q \alpha _k ^{q _0 - q}.
            \end{align*}
            So $\alpha _k > \beta ^{\frac q{q - q _0}}$, hence $k > k _* := \frac{q}{q - q _0}\log _2 \beta$. Moreover, \eqref{eqn:L18} yields 
            \begin{align*}
                |B| = \sum _{k > k _*} |B _k| \le \sum _{k > k _*} K \alpha _k ^{q - 1} \beta ^{-q} \le K a _{k _*} ^{q - 1} \beta ^{-q} = K \beta ^{\frac{q (q - 1)}{q - q _0} - q} = K \beta ^{-p}.
            \end{align*}
            
            \item Without loss of generality assume $\nor{f} _{L ^{p _0, \infty} _t L ^\infty _x} = 1$. Note that \eqref{eqn:Lpq} implies $S (t, \alpha _k) > 0$, hence $\nor{f (t)} _{L ^\infty (\Omega, \mu _t)} \ge \alpha _k$. Therefore, for any $k \in \mathbb Z$ we have 
            \begin{align*}
                \alpha _k ^{p _0} |B _k| < 1. 
            \end{align*}
            Combined with \eqref{eqn:L18}, we have 
            \begin{align*}
                |B _k| \le \mins{
                    \alpha _k ^{-p _0}, K \alpha _k ^{q - 1} \beta ^{-q} 
                }, \qquad \forall k \in \mathbb Z.
            \end{align*}
            Hence for some $k _*$ to be determined, we have
            \begin{align*}
                |B| = \sum _{k \in \mathbb Z} |B _k| = \sum _{k \le k _*} |B _k| + \sum _{k > k _*} |B _k| &\le \sum _{k \le k _*} K \alpha _k ^{q - 1} \beta ^{-q} + \sum _{k > k _*} \alpha _k ^{-p _0} \\
                & \le K \alpha _{k _*} ^{q - 1} \beta ^{-q} + \alpha _{k _*} ^{-p _0}.
            \end{align*}
            By selecting $k _* := \frac1{p _0 + q - 1} \log _2 \pthf{\beta ^q}K$, we have 
            \begin{align*}
                |B| \le \pthf{\beta ^q}K ^{-\frac{p _0}{p _0 + q - 1}} = K ^{\frac pq} \beta ^{-p}.
            \end{align*} 
        \end{enumerate}
    \end{proof}
\end{lemma}

\end{appendix}

\section*{Declarations}

\addtocontents{toc}{\protect\setcounter{tocdepth}{1}}

\subsection*{Funding}
The author was partially supported by the US National Science Foundation under Grant No.~DMS-1926686. 

\subsection*{Conflicts of interests}
The authors have no relevant financial or non-financial interests to disclose.

\bibliographystyle{alpha}
\bibliography{ref.bib}

\end{document}